\documentclass[12pt]{article}
\usepackage{graphicx}
\usepackage{natbib} 
\usepackage{url} 
\usepackage{enumitem}
\usepackage[utf8]{inputenc}
\usepackage{amssymb, psfig, epsfig, amsmath, multirow, caption}
\usepackage[colorlinks = true, citecolor = blue, urlcolor = blue]{hyperref}
\usepackage{graphicx}
\usepackage{subfiles}
\usepackage{epstopdf}
\usepackage[ruled,vlined]{algorithm2e}
\usepackage{float}
\usepackage{booktabs,xltabular}
\usepackage[english]{babel}
\usepackage{amsthm}
\DeclareGraphicsExtensions{.jpg,.fig,.eps}
\newtheorem{coro}{Corollary}

\newcommand{\EE}{\mathbb{E}}
\newcommand{\PP}{\mathbb{P}}

\newtheorem{remark}{Remark}%
\newtheorem{assumption}{Assumption}
\newtheorem{lemma}{Lemma}
\newtheorem{theorem}{Theorem}

\DeclareMathOperator*{\argmin}{arg\,min}
\newcommand{\blind}{0}

\begin{document}

\def\spacingset#1{\renewcommand{\baselinestretch}%
{#1}\small\normalsize} \spacingset{1}


\if0\blind
{
  \title{\bf An Online Learning Approach to Dynamic Pricing and Capacity Sizing in Service Systems}
  \author{Xinyun Chen\\
  	School of Data Science, Chinese University of Hong Kong (Shenzhen),\\
    Yunan Liu\\
    Dept. of Industrial and Systems Engineering, North Carolina State University,\\
	Guiyu Hong\\
	School of Data Science, Chinese University of Hong Kong (Shenzhen),
}
  \maketitle
} \fi

\if1\blind
{
  \bigskip
  \bigskip
  \bigskip
  \begin{center}
    {\LARGE\bf Title}
\end{center}
  \medskip
} \fi

\bigskip
\begin{abstract}
We study a dynamic pricing and capacity sizing problem in a $GI/GI/1$ queue, where the service provider's objective is to obtain the optimal service fee $p$ and service capacity $\mu$ so as to maximize the cumulative expected profit (the service revenue minus the staffing cost and delay penalty). 
Due to the complex nature of the queueing dynamics, such a problem has no analytic solution so that previous research often resorts to heavy-traffic analysis where both the arrival rate and service rate are sent to infinity.
In this work we propose an \textit{online learning} framework designed for solving this problem which does not require the system's scale to increase. Our framework is dubbed \textit{Gradient-based Online Learning in Queue} (GOLiQ). 
GOLiQ organizes the time horizon into successive operational cycles and prescribes an efficient procedure to obtain improved pricing and staffing policies in each cycle using data collected in previous cycles.
Data here include the number of customer arrivals, waiting times, and the server's busy times.
The ingenuity of this approach lies in its online nature, which allows the service provider do better by interacting with the environment.
Effectiveness of GOLiQ is substantiated by (i) theoretical results
including the algorithm convergence and regret analysis (with a logarithmic regret bound), and (ii) engineering confirmation
via simulation experiments of a variety of representative $GI/GI/1$ queues.
\end{abstract}

\noindent%
\vfill

\newpage
\spacingset{1.8} 
\section{Introduction}

\subsection{Problem Statement and Methodology}
\subsubsection{Pricing and capacity sizing in queue}
We study a service queueing model where the service provider manages congestion and revenue by dynamically adjusting the price and service capacity. Specifically, we consider a $GI/GI/1$ queue, in which  the demand for service is $\lambda(p)$ per unit of time when each customer is charged by a service fee $p$; the cost for providing service capacity $\mu$ is $c(\mu)$; and a holding cost $h_0$ incurs per job per unit of time. By choosing the appropriate service fee $p$ and capacity $\mu$, the service provider aims to maximize the net profit, which is the service fee minus the staffing cost and penalty of congestion, i.e.,
\begin{equation}\label{eq: maximizing revenue}
	\max_{\mu,p}~\mathcal{P}(\mu,p)\equiv p\lambda(p)-c(\mu)-h_0\EE[Q_\infty(\mu,p)],
\end{equation}
where $Q_\infty(\mu,p)$ is the steady-state queue length under service rate $\mu$ and price $p$.

Problems in this framework have a long history, see for example \cite{Kumar2010}, \cite{LeeWard2014}, \cite{LeeWard2019}, \cite{Maglaras2003}, \cite{Nair2016}, \cite{Kim2018}  and the references therein. 
Due to the complex nature of the queueing dynamics, exact analysis is challenging and often unavailable (computation of the optimal dynamic pricing and staffing rules is not straightforward even for the Markovian $M/M/1$ queue \citep{AS06}).
Therefore, researchers resort to heavy-traffic analysis to approximately obtain performance evaluation and optimization results. Commonly adopted heavy-traffic regimes require sending the arrival rate and service capacity (service rate or number of servers) to $\infty$. 
Although heavy-traffic analysis provides satisfactory results for large-scale queueing systems, approximation formulas based on heavy-traffic limits often become inaccurate as the system scale decreases. 

\subsubsection{An online learning method}
In this paper \textit{we propose an online learning framework designed for solving Problem} \eqref{eq: maximizing revenue}.
According to our online learning algorithm, the $GI/GI/1$ queue will be operated in successive cycles, where in each cycle the service provider's decisions on the service fee $p$ and service capacity $\mu$, {deemed the best by far}, are obtained using the system's data collected in previous operational cycles. Data hereby include (i) the number of customers who join for service, (ii) customers' waiting times, and (iii) the server's busy time, all of which are easy to collect. Newly generated data, which represent the response from the (random and complex) environment to the present operational decisions,  will be used to obtain improved pricing and staffing policies in the next cycle.
In this way the service provider can dynamically interact with the environment so that the operational decisions can evolve and eventually approach the optimal solution. 
{At the beginning of each} cycle $k$, the service provider's decisions $(p_k,\mu_k)$ will be computed and enforced throughout the cycle. 
At the heart of our procedure for computing $(p_k,\mu_k)$ is to obtain a sufficiently accurate estimator $H_{{k-1}}$ for the gradient of the objective function of \eqref{eq: maximizing revenue}, using past experience. Specifically, our online algorithm will update $(p_k, \mu_k)$ according to
\begin{align*}
	(\mu_{k},p_{k})\leftarrow (\mu_{k-1},p_{k-1})+\eta_{k-1} H_{k-1},
\end{align*}
where $\eta_k$ is the updating step size for cycle $k$. We call this algorithm \textit{Gradient-based Online Learning in Queue} (GOLiQ).

Besides showing that, under our online learning scheme, the decisions in cycle $k$, $(\mu_k, p_k)$ will converge to the optimal solutions $(\mu^*, p^*)$ as $k$ increases, we quantify the effectiveness of GOLiQ by computing the \textit{regret} - the cumulative loss of profit due to the suboptimality of $(\mu_k, p_k)$, namely, the maximum profit under the (unknown) optimal strategy minus the expected profit earned under the online algorithm over time. 
When GOLiQ's hyperparameters are chosen optimally, we show that our regret bound is logarithmic so that the service provider, with any initial pricing and staffing policy $(\mu_0, p_0)$, will quickly learn the optimal solutions without losing much profit in the learning process.




\subsection{Advantages, Challenges and Contributions}

In what follows, we first discusses the general advantages of the online learning approach by contrasting with heavy-traffic methods; we next explain the key challenges we face in the development of online learning algorithms for queueing systems.

\subsubsection{Online learning vs. heavy-traffic method.}
First, heavy-traffic solutions are derived from approximating models which arise as the system scale approaches infinity, so the fidelity of the solutions is sensitive to  the system scale. Unlike heavy-traffic methods, online learning approaches do not require any asymptotic scaling, so they can treat service systems at any scale (small or large). 
Second, heavy-traffic approaches usually require the knowledge of certain distributional information apriori (e.g., moments and distribution functions of service times), which serve as critical input parameters for the heavy-traffic models.
On the other hand, online learning methods require information of this kind to a lesser extent. Although certain distribution information can help fine-tune parameters of online algorithms, it is  {less crucial} to  {algorithm design and implementation}. So in this sense, the dependence on the distributional information is weaker than that of heavy-traffic analysis.
{Last, online learning is advantageous when the underlining problem focuses on performance optimization in the long run. Heavy-traffic analysis gives approximate solutions that are static, and in a longer time frame, the performance discrepancy (relative to the true optimal reward) should grow linearly as time increases. But online learning is a dynamic evolution, and its data-driven nature enables it to constantly produce improved solutions which will eventually reach optimality.
	In addition,  heavy-traffic solutions require the establishment of  heavy-traffic limit theorems and careful analysis of the dynamics of the limit processes (e.g., fluid and diffusion). Both steps can be quite sophisticated in general. 
	See Remarks \ref{rem:phi} and \ref{rem:info} for more detailed discussions; also see Section \ref{sec:AsymOpt} for numerical evidence.}



\subsubsection{Challenges of online learning in queueing systems.} 
{Online learning in queues is by no means an easy extension of online learning in other domains; its theoretical development has to account for the unique features in queueing systems.}
A crucial step is to develop effective ways to control the nonstationary error that arises at the beginning of every cycle due to the policy update. Towards this, we develop a new regret analysis framework for the transient queueing performance that not only helps establish desired regret bounds for the specific online $GI/GI/1$ algorithm, but may also be used to develop online learning method for other queueing models (see Section \ref{sec: framework}).
Another challenge we have to address here is to devise a convenient gradient estimator for the online learning algorithm (essentially, an estimator for the gradient of $\EE[Q_\infty(\mu,p)]$). The estimator should have a negligible bias to warrant a quick convergence of the algorithm, and at the same time, its computation (using previous data) should be sufficiently straightforward to ensure the ease of implementation
(The detailed gradient estimator of GOLiQ for the $GI/GI/1$ system is given in Section \ref{sec: direct}).

\subsubsection{{Main Contributions}}\label{subsec: contribution}
We summarize our contributions below.
\begin{itemize}
	\item
	To the best of our knowledge, the present work is the first to develop an online learning framework for joint pricing and staffing in a queueing system {with logarithmic regret bound in the total number of customers served (Theorem \ref{thm: regret direct})}. Due to the complex nature of queueing systems,
	previous research often resorts to asymptotic heavy-traffic analysis to approximately solve for desired operational decisions. The ingenuity  of our online learning method lies in the ability to obtain the optimal solutions without needing the system scale (e.g., arrival rate and service rate) to grow large. The other appeal of our method is its robustness, especially in its weaker dependence on the distributions of service and arrival times. 
	
	\item	
	A critical step in the regret analysis is the treatment of the transient system dynamics, because when improved operational decisions are obtained and implemented at the beginning of a new period, the queueing performance will shift away from previously established steady-state level.
	Towards this, we develop a new way to treat and bound the transient queueing performance in the regret analysis of our online learning algorithm {(Theorem \ref{thm: non-stationary error}). Bounding the transient error also guarantees convergence of the SGD iteration (Theorem \ref{thm: broadiezeevi}). } Comparing to previous literature (e.g., the regret bound is $O(T^{2/3})$ in \cite{Huh2009}), our analysis of the regret due to nonstationarity gives a much tighter logarithmic bound.
	In addition, the regret analysis in the present paper may be extended to other queueing systems which share similar properties to $GI/GI/1$.
	
	\item	
	Supplementing the theoretical results of our regret bound, we evaluate the practical effectiveness of our method by conducting comprehensive numerical experiments. Our simulations draw the following two main conclusions. First, our method is robust in several dimensions: (i) GOLiQ exhibits convincing performance for $GI/GI/1$ queues having representative arrival and service distributions; (ii) GOLiQ remains effective even when certain theoretical assumptions are relaxed. Furthermore, in order to clearly highlight the advantages of our online learning approach relative to the previous results of heavy-traffic limits, we provide a careful performance comparison of these two methods. We show that GOLiQ is more effective in any one of the following three cases: the system scale is not too large, {staffing cost is high}, or service times are more variable.
\end{itemize}

\subsection{{Organization of the paper}}
In Section \ref{sec: lit}, we review the related literature.
In Section \ref{sec: Model}, we introduce the model assumptions and provide an outline of our online learning algorithm.
In Section \ref{sec: framework}, we conduct the regret analysis for GOLiQ by separately treating
the \textit{regret of nonstationarity} - the part of regret arising from the transient system dynamics,
and the \textit{regret of suboptimality} - the part originating from the errors  due to suboptimal  pricing and staffing decisions.
In Section \ref{sec: direct}, we give the detailed description of GOLiQ and establish a logarithmic regret bound by appropriately selecting our algorithm parameters.
In Section \ref{sec: num} we conduct numerical experiments to confirm the effectiveness and  {robustness} of GOLiQ.
We conclude in Section \ref{sec: conclusions}.
In the e-companion, we give all technical proofs and provide additional numerical examples.

\section{Related Literature}\label{sec: lit}
The present paper is related to the following three streams of literature.
\paragraph{Pricing and capacity sizing in queues.}
There is a rich literature on pricing and capacity sizing in service systems under different settings. \cite{Maglaras2003}  studies pricing
and capacity sizing problem in a processor sharing queue motivated by internet applications; \cite{Kumar2010} considers a single-server system with nonlinear delay costs; \cite{Nair2016} studies $M/M/1$ and $M/M/k$ systems with network effect among customers; \cite{Kim2018} considers a dynamic pricing problem in a single-server system. The specific problem \eqref{eq: maximizing revenue} we consider here is most closely related to \cite{LeeWard2014}, i.e., joint pricing and capacity sizing for the $GI/GI/1$ queue. Later, the authors extend their results to the $GI/GI/1+G$ model with customer abandonment in \cite{LeeWard2019}.
As there is usually no closed-form solution for the optimal strategy or equilibrium, asymptotic analysis is adopted under large-market assumptions. In detail, their analysis is rooted in a deterministic static planning problem which requires both the service capacity and the demand rate  to scale to infinity.
Most of the papers conclude that heavy-traffic regime is economically optimal. (There are some exceptions where heavy-traffic regime is not optimal, for example, \cite{Kumar2010} shows that agent is forced to decrease its utilization if the delay cost is concave.)  Our algorithm is motivated by the pricing and capacity sizing problem for service systems, however, as explained previously, our methodology is very different from the asymptotic analysis used in these papers.

\paragraph{Reinforcement learning for queueing systems.} Our paper is also related to a small but growing literature on reinforcement learning (RL) for queueing systems. \cite{Dai2021} studies an actor-critic algorithm for queueing networks. \cite{Xie2019} and \cite{Xie2020} develop RL techniques to treat the unboundedness of the state space of queueing systems. 
\cite{Shi2020} studies a price-based revenue management problem in an $M/M/c$ queue with a discrete price space; their methodology draws from the multi-armed bandit framework (with each price treated as an ``arm"). \cite{Krishnasamy2021learning} develops  bandit methods for scheduling problem in a multi-server queue with unknown service rates.
{Our work draws distinction from the above-mentioned literature in two dimensions. First, we are the first to develop an online learning method for  joint pricing and capacity sizing in queue. In addition, our method applies to settings of continuous decision variables.} Comparing to the more general RL literature, our algorithm design and regret analysis take advantage of the specific queueing system structure so as to establish tight regret bounds and more accurate control of the convergence rate. In some sense, the algorithm developed in the present paper may be viewed as a version of the \textit{policy gradient} method, a special class of RL methods \citep{SB18}, see Remark \ref{rmk: SGD} for detailed discussions. 

\paragraph{Stochastic gradient decent algorithms.} In general, our algorithm falls into the broad class of \textit{stochastic gradient descent} (SGD) methods. There are some early papers on SGD algorithms for steady-state performance of queues (see \cite{Fu1990}, \cite{Chong1993}, \cite{LEcuyer94b}, \cite{LEcuyer1994} and the references therein). In particular, these papers have established convergence results of SGD algorithms for capacity sizing problems with a variety of gradient estimating designs. In this paper, we consider a more general setting in which the price is also optimized jointly with the service capacity. Besides, in order to establish theoretical bounds for the regret, we conduct a careful analysis on the convergence rate of the algorithm and provide an explicit guidance for the optimal choice of algorithm parameters, which is not discussed in this early literature.
Our algorithm design and analysis are also related to the online learning methods in recent inventory management literature \citep{Burnetas2000,Huh2009,Huh2013,Zhang2019,Yuan2021marrying}. Among these papers, our work is perhaps most closely related to \cite{Huh2009} where the authors develop an SGD based learning method for an inventory model with a bounded replenishment lead time. 
Still, due to the unique natures of queueing models, we develop a new regret analysis framework as we shall explain with details in Section \ref{subsec: contribution}.

\section{Problem Setting and Algorithm Outline}\label{sec: Model}
In Section \ref{subsec: assumptions} we describe the queueing model and technical assumptions.
In Section \ref{subsec: algorithm outline}, we  provide a general outline of GOLiQ.
Finally, in Section \ref{subsec: queueing dynamics} we conduct preliminary analysis of the queueing performance under GOLiQ.

\subsection{Model and Assumptions}\label{subsec: assumptions}
We study a $GI/GI/1$ queueing system having customer arrivals according to a renewal process with generally distributed interarrival times (the first $GI$), \textit{independent and identically distributed} (i.i.d.) service times following a general distribution (the second $GI$), and a single server that provides service under the first-in-first-out (FIFO) discipline.
Each customer upon joining the queue is charged by the service provider a fee $p>0$. The demand arrival rate (per time unit) depends on the service fee $p$ and is denoted as $\lambda(p)$. To maintain a service rate $\mu$, the service provider continuously incurs a staffing cost at a rate $c(\mu)$ per time unit.

For $\mu\in[\underline{\mu},\bar{\mu}]$ and $p\in[\underline{p}, \bar{p}]$, the service provider's goal is to determine the optimal service fee $p^*$ and service capacity $\mu^*$ with the objective of  maximizing the steady-state expected profit \eqref{eq: maximizing revenue},
or equivalently minimizing the objective function $f(\mu,p)$ as follows
\begin{align}\label{eqObjmin}
	\min_{(\mu,p)\in\mathcal{B}} f(\mu, p) \equiv h_0 \EE[Q_{\infty}(\mu,p)] + c(\mu)- p\lambda(p),\qquad \mathcal{B}\equiv [\underline{\mu},\bar{\mu}]\times[\underline{p}, \bar{p}].
\end{align}

We shall impose the following assumptions on the above service system throughout the paper.
\begin{assumption}\label{assmpt: uniform}\textbf{$($Demand rate, staffing cost, and uniform stability$)$}
	\begin{enumerate}
		\item[$(a)$] The arrival rate $\lambda(p)$ is continuously differentiable and non-increasing in $p$.
		\item[$(b)$] The staffing cost $c(\mu)$ is continuously differentiable and non-decreasing in $\mu$.
		\item[$(c)$] The lower bounds $\underline{p}$ and $\underline{\mu}$ satisfy that  $\lambda(\underline{p})<\underline{\mu}$ so that the system is uniformly stable for all feasible choices of the pair $(\mu,p)$.
	\end{enumerate}
\end{assumption}

Part $(c)$ of Assumption \ref{assmpt: uniform} 
is commonly used in the literature of SGD methods for queueing models to ensure that the steady-state mean waiting time $\EE[W_\infty(\mu,p)]$ is differentiable with respect to model parameters (see \cite{Chong1993}, \cite{Fu1990}, \cite{LEcuyer94b}, \cite{LEcuyer1994}, also see Theorem 3.2 of \cite{Glasserman_1992}). 
In the our numerical experiments (see Section \ref{sec:UnifStab}), we show that our online algorithm remains effective when this assumption is relaxed.

We do not require full knowledge of service and inter-arrival time distributions. But in order to develop explicit bounds for the part of the regret due to the nonstationarity of the queueing processes, we require both distributions to be light-tailed. Specifically, since the actual service and interarrival times  are subject to our pricing and staffing decisions, we model the interarrival and service times by two scaled random sequences $\{U_n/\lambda(p)\}$ and $\{V_n/\mu\}$, where $U_1,U_2,\ldots$ and $V_1,V_2,\ldots$ are two independent i.i.d. sequences of random variables having unit means, i.e., $\EE[U_n]=\EE[V_n]=1$. We make the following assumptions on $U_n$ and $V_n$.
\begin{assumption}\label{assmpt: light tail} \textbf{$($Light-tailed service and interarrival times$)$} \\
	There exists a sufficiently small constant $\eta>0$ such that the moment-generating functions
	$$\EE[\exp(\eta V_n)]<\infty\quad \text{and}\quad \EE[\exp(\eta U_n)]<\infty. $$ 
	In addition, there exist constants $0<\theta <\eta/2\bar{\mu}$, $0<a<(\underline{\mu}-\lambda(\underline{p}))/(\underline{\mu}+\lambda(\underline{p}))$ and $\gamma >0$ such that
	\begin{equation}\label{eq: gamma0}
		\phi_U(-\theta) < -(1-a)\theta -\gamma\quad \text{and}\quad \phi_V(\theta) < (1+a)\theta -\gamma,
	\end{equation}
	where $\phi_V(\theta)\equiv \log\EE[\exp(\theta V_n)]$ and $\phi_U(\theta)\equiv \log\EE[\exp(\theta U_n)]$ are the cummulant generating functions of $U$ and $V$.
\end{assumption}
Note that $\phi'_U(0)=\phi'_V(0)=1$ as $\EE[U]=\EE[V]=1$. 
Suppose $\phi_U$ and $\phi_V$ are smooth around 0, then we have $\phi_U(-\theta) = -\theta +o(\theta)$ and $\phi_V(\theta) = \theta +o(\theta)$ by Taylor's expansion. This implies that, for any $a>0$, we can make $\theta$ small enough, such that $\phi_U(-\theta)<-(1-a)\theta$ and $\phi_V(\theta)<(1+a)\theta$. To obtain the bound in \eqref{eq: gamma0}, we can simply take 
$\gamma={\frac{1}{2}}\min(-(1-a)\theta-\phi_U(-\theta), (1+a)\theta - \phi_V(\theta))>0$.
Hence, a sufficient condition that warrants \eqref{eq: gamma0} is to require that $\phi_U$ and $\phi_V$ be smooth around 0, which is true for many distributions of $U$ and $V$ considered in common queueing models. Assumption \ref{assmpt: light tail} will be used in our proofs to build an explicit bound for the regret of nonstationarity.

Finally, in order to warrant the convergence of our online learning algorithm, we require a convex structure for the problem in \eqref{eqObjmin}, which is common in the SGD literature; see \cite{BoradieZeevi2011}, \cite{KushnerYin2003} and the references therein.

Let $x^*\equiv (\mu^*,p^*)$ and $x\equiv(\mu, p)$. Let $\nabla f(x)$ denote the gradient of a function $f(x)$ and $\|\cdot\|$ denote the Euclidean norm.
\begin{assumption}\label{assmpt: convexity}\textbf{$($Convexity and smoothness$)$}\\
	There exist finite positive constants $K_0\leq 1$ and $K_1>K_0$ such that for all $x\in\mathcal{B}{,}$
	\begin{enumerate}
		\item[$(a)$] $(x-x^*)^T\nabla f(x)\geq K_0\|x-x^*\|^2$;
		\item[$(b)$] $\|\nabla f(x)\|\leq K_1\|x-x^*\|${.}
	\end{enumerate}
	
\end{assumption}
\begin{remark}
	
		Our simulation experiments  show that our algorithm works effectively for some representative $GI/GI/1$ queues with conditions in Assumption \ref{assmpt: convexity} relaxed; see Section \ref{sec: num} and Section \ref{sec: AddNum} in the Supplement Material.
		In addition, we later provide some sufficient conditions for Assumption \ref{assmpt: convexity} in the special case of $M/GI/1$ queues in Section \ref{appx: demand function}.  
\end{remark}

\subsection{Outline of GOLiQ}\label{subsec: algorithm outline}
In general, an SGD algorithm for a minimization problem $\min_x f(x)$ over a compact set $\mathcal{B}$ relies on updating the decision variable via the recursion
\begin{align*}
	x_{k+1} = \Pi_{\mathcal{B}}(x_k -\eta_k H_k ),\qquad k\geq 1.
\end{align*}
where $\eta_k$ is the step size, $H_k$ is a random estimator for $\nabla f(x_k)$, $x_k$ is the decision variable by step $k$, and the projection operator $\Pi_{\mathcal{B}}$ restricts the updated decision in $\mathcal{B}$. For problem \eqref{eqObjmin}, we let $x_k\equiv (\mu_k,p_k)$ represent the service capacity and price at step $k$,  We define
\begin{equation}\label{eq: BV}
	B_k\equiv\EE[\|\EE[H_k-\nabla f(x_k) | \mathcal{F}_k]\|^2]^{1/2}\quad \text{and}\quad \mathcal{V}_k \equiv \EE[\|H_k\|^2],
\end{equation}
where $\mathcal{F}_k$ is the $\sigma$-algebra including all events in the first $k-1$ iterations. Intuitively, $B_k$ measures the bias of the gradient estimator $H_k$ and $\mathcal{V}_k$ measures its variability. As we shall see later, $B_k$ and $\mathcal{V}_k$ play important roles in  designing the algorithm and establishing desired regret bounds.

The standard SGD algorithm iterates in discrete step $k$. In our setting, however, the queueing system and objective function $f(\mu,p)$ are defined in continuous time (in particular, $Q_\infty(\mu,p)$ is the steady-state queue length observed in continuous time). To facilitate the regret analysis, we first transform the objective function into an expression of customer waiting times that are observed in discrete time.
By Little's law, we can rewrite the objective function $f(\mu,p)$ as, for all $(\mu,p)\in\mathcal{B}$,
\begin{align}\label{eq:obj}
	f(\mu,p) = h_0\lambda(p)\left(\EE[W_\infty(\mu,p)]+ \frac{1}{\mu}\right)+c(\mu)-p\lambda(p),
\end{align}
where $W_\infty(\mu,p)$ is the steady-state waiting time under $(\mu,p)$. 
In each cycle $k$, our algorithm adopts the average of $D_k$ observed customer waiting times to estimate $\EE[W_\infty(\mu,p)]$, where $D_k$ denotes the number of customers that enter service in cycle $k$ (we refer to $D_k$ as the cycle length or sample size of cycle $k$).
But any finite $D_k$ will introduce a bias to our gradient estimate $H_k$.
To mitigate the bias due to the transient performance of the queueing process, we shall let the cycle length $D_k$ be increasing in $k$ (in this way
the transient bias will vanish eventually). We give the outline of the algorithm below.

\textbf{Outline of GOLiQ:}
\begin{enumerate}
	\item[0.] Input: $\{D_k\}$ and $\{\eta_k\}$ for $k=1,2,..,L$, initial policy $x_1=(\mu_1, p_1)$. \medskip
	
	For $k = 1, 2,...,L$,
	\item[1.] In the $k^{\rm th}$ cycle, operate the $GI/GI/1$ queue under policy $x_k=(\mu_k,p_k)$ until $D_k$ customers enter service.
	
	\item[2.] Collect and use the data (e.g., customer delays) to build an estimator $H_k$ for $\nabla f(\mu_k,p_k)$.
	
	\item[3.] Update $x_{k+1} = \Pi_\mathcal{B}(x_k-\eta_k H_k)$.
\end{enumerate}
\vspace{1ex}
\begin{remark}[Exploration vs. exploitation]\label{rmk: SGD}
	The online nature of this algorithm makes it possible to obtain improved decisions by learning from past experience, which is in the spirit of  the essential ideas of \textit{reinforcement learning} where an agent (hereby the service provider) aims to tradeoff between \textit{exploration} (Step 1) and \textit{exploitation} (Steps 2 and 3). Effectiveness of the algorithms lies in properly choosing the algorithm parameters and devising an efficient gradient estimator $H_k$. {For example, if $D_k$ is too small, we are unable to generate sufficient data (we do not have much to \textit{exploit} in order for devising a better policy); if $D_k$ is too large, we incur a higher profit loss due to suboptimality of the policy in use (we do not \textit{explore} enough for seeking potentially better policies).}
	In particular, GOLiQ may be viewed as a special case of the \textit{policy gradient} (PG) algorithm 
	(the general idea of PG is to estimate the policy parameters using the gradient of the value function learned via continuous interaction with the system, see for example \cite{SB18}). To put this into perspective, the policy in the present paper is specified by a pair of parameters $(\mu,p)$, and in each iteration, we update the policy parameters using an estimated policy gradient $H_k$ learned from data of the queueing model.
	In the subsequent sections, we give detailed regret analysis that can be used to establish optimal algorithm parameters (Section \ref{sec: framework}) and develop an efficient gradient estimator (Section \ref{sec: direct}).
\end{remark}

\subsection{System Dynamics under GOLiQ}\label{subsec: queueing dynamics}
We explain explicitly the dynamic of the queueing system under GOLiQ, with the system starting empty. We first define notations for relevant performance functions.  For $k\geq 1$, let $T_k$ be the length of cycle $k$ in the units of time, and let $D_k$ be the total number of customers who enter service in cycle $k$.  For  {$n= 1,2,...,D_k$}, let $W_n^k$ be the waiting time of the $n^{\rm th}$ customer that enters service in cycle $k$.  {We define  $W_0^k\equiv W_{D_{k-1}}^{k-1}$.} 
We use the two i.i.d. random sequences $V_n^k$ and $U_n^k$ to construct the service and inter-arrival times in cycle $k$, $n=1,2,...,D_k$.  {In particular, $V_n^k$ corresponds to the service time of the customer $n-1$, and $U_n^k$ corresponds to the inter-arrival time between customers $n-1$ and $n$ in cycle $k$.} Let $\lambda_k\equiv \lambda(p_{k})$. Last, we use $Q_k$ to denote the number of existing customers (those who arrive in previous cycles) at the beginning of cycle in $k$, with $Q_1=0$. We will have $Q_k\geq 1$ for $k\geq 2$, as we shall explain soon, according to our updating procedure.
The detailed dynamics of the queueing system in cycle $k$ is summarized as follows:
\begin{itemize}
	\item {\bf Updating the control policy.} In cycle $k$, we adopt the pricing and staffing policy $(p_k,\mu_k)$. The service time of customer $n-1$ in cycle $k$ is $S_{n}^k=V_{n}^k/\mu_k$ for $n=1, ..., D_k$. Cycle $k$ ends as soon as a total number of $D_k$ (of which the value is to be determined later) customers have entered service. So, customer $D_k$ will receive service in cycle $k+1$ (with service time $S_1^{k+1}$) and the queue leftover consists of at least one customer, i.e., $Q_{k+1}\geq 1$ for a new cycle $k+1$, which begins under a new policy $(p_{k+1},\mu_{k+1})$ as follows:
	\begin{itemize}
		\item {\bf  Service rate.} The service rate is updated to $\mu_{k+1}$ immediately as the new cycle begins, so that all existing customers will undergo service times with rate $\mu_{k+1}$.
		\item {\bf Service fee.}
		The price remains $p_k$ 
		at the beginning of cycle $k+1$ and evolves to $p_{k+1}$ immediately after the first new customer arrives in the new cycle; we charge this customer with $p_{k}$ (because its interarrival time is modulated by $p_{k}$) and all subsequent customers in cycle $k+1$ with $p_{k+1}$. 
	\end{itemize}
	\item {\bf Leftovers from previous cycles.}
	For $k\geq 2$, at the beginning of cycle $k$, there are $Q_k-1$ customers waiting in queue indexed by $n$ from 1 to $Q_k-1$. The  customer who just enters service  is indexed by 0. 
	We update the price from $p_{k-1}$ to $p_k$ right after the first new customer (indexed by $Q_k$) arrives in a new cycle. As a consequence, the prices charged to customers $1, 2, ...,Q_k$ are not yet updated to $p_k$. Denote by $p_n^k$ and $\lambda_n^k\equiv \lambda(p_n^k)$ as the price and arrival rate for customer $n$  in cycle $k$, respectively, for $1\leq n\leq Q_k$. The corresponding interarrival time is $\tau_n^k=U_n^k/\lambda_n^k$. In case $Q_{k-1}>D_{k-1}$, some queueing leftover are customers from earlier cycles. So here $p_n^k\in\{p_1, p_2,...,p_{k-1}\}$. In addition, in case $Q_k>D_k$, part of $Q_k$ will continue to remain in cycle $k+1$ and we will have, for example, $p^{k+1}_1=p^k_{D_k+1}$.
	\item {\bf New arrivals.}
	We denote interarrival times for new customers in cycle $k$ by  $\tau_n^k = U_n^k/\lambda_k$ for $n=Q_{k}+1,...,D_k$ if $D_k\geq Q_k+1$. (As will soon become clear, the case $D_k \leq Q_k$ is a rare event with a negligible probability under appropriate algorithm settings, see Remark \ref{rmk: left-overs}.)
	\item {\bf Customer delay.}
	Customers' waiting times  in cycle $k$ are characterized by the recursions
	\begin{equation}\label{eq:RecDelay}
		W_n^k =
		\begin{cases}
			\left(W_{n-1}^k + \frac{V^k_n}{\mu_k} - \frac{U^k_n}{\lambda^k_n}\right)^+ & \text{ for }1\leq n\leq Q_k\wedge D_k;\\
			\left(W_{n-1}^k + \frac{V^k_n}{\mu_k} - \frac{U^k_n}{\lambda_{k}}\right)^+ & \text{ for } (Q_k+1)\wedge{(D_k+1)}\leq n\leq D_k.
		\end{cases}, \quad W_0^k = W_{D_{k-1}}^{k-1},
	\end{equation}
	where $x^+\equiv \max\{x,0\}$.
	\item {\bf Server's busy time.}
	The age of the server's busy time observed by customer $n$ upon arrival, which is the length of time the server has been busy since the last idleness, is given by the recursions
	\begin{equation}\label{eq:RecBusyPeriod}
		X_n^k =
		\begin{cases}
			\left(X_{n-1}^k + \frac{U^k_n}{\lambda^k_n}\right){\bf 1}_{\{W_n^k>0\}}& \text{ for }1\leq n\leq Q_k\wedge D_k;\\
			\left(X_{n-1}^k + \frac{U^k_n}{\lambda_{k}}\right){\bf 1}_{\{W_n^k>0\}}& \text{ for } (Q_k+1)\wedge{ (D_k+1)}\leq n\leq D_k.
		\end{cases}, \quad X_0^k = X_{D_{k-1}}^{k-1},
	\end{equation}
	where the indicator ${\bf 1}_A$ is 1 if $A$ occurs and is 0 otherwise.
\end{itemize}
We provide explanations for \eqref{eq:RecDelay} and \eqref{eq:RecBusyPeriod}. First, recursion \eqref{eq:RecDelay} simply follows from Lindley's equation.
Next, recursion \eqref{eq:RecBusyPeriod} follows from the fact that, for customer $n$, if the queue is empty upon its arrival, the observed busy time is simply 0 by definition; otherwise, the server must have been busy since the arrival of the previous customer and therefore, the observed busy time by customer $n$ should extend that of customer $n-1$ by an additional inter-arrival time.
As we shall see later, both the delay and busy time observed by customers will be important ingredients (i.e., data) for building the gradient estimator of the online learning algorithm.
\begin{remark}[Clearance of the leftover $Q_k$]\label{rmk: left-overs}
	As explained above,  $Q_k$ is random and unbounded, while in our algorithm design, the cycle length $D_k$ is deterministic. So it is indeed possible the remaining queue content may not be all cleared in cycle $k$ (i.e., $D_k<Q_k$). We will see later in the regret analysis that our choice of $D_k$ leads to a small probability of uncleared leftovers  and thus the impact of the rare event $\{D_k<Q_k\}$ is negligible.
\end{remark}

In Figure \ref{fig:timeline}, we further illustrate how the service price and service rate are updated by showing the ordering of all relative events as a new cycle begins. We emphasize that (i) the service rate $\mu_{k-1}$ is updated to $\mu_k$ immediately when a new cycle $k$ begins, which is triggered as soon as the last one of $D_{k-1}$ customers enters service; and (ii) the service price $p_{k-1}$ is updated to $p_k$ only after the first external arrival occurs in the new cycle $k$ (we honor our previous prices for all customers who arrive in the previous cycle).
\begin{figure}
	\centering
	\includegraphics[width=0.9\linewidth]{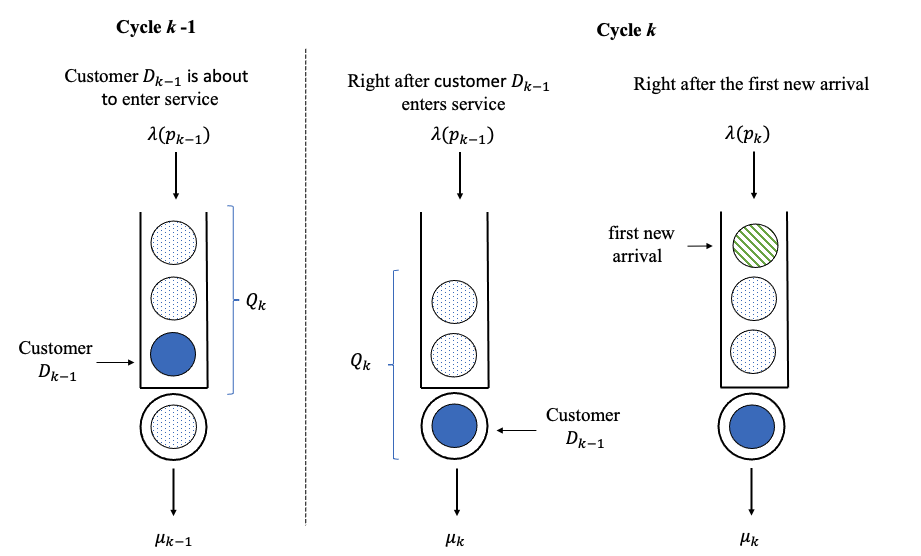}
	\caption{On the timing of the update of $p_k$ and $\mu_k$ under GOLiQ.}
	\label{fig:timeline}
\end{figure}
	
	We end this section by providing a uniform boundedness result for all relevant queueing functions. This result below will be used in the next sections to establish desired regret bounds. The proof follows from a stochastic ordering approach and is given in Section \ref{subsect: uniform bound}.
	\begin{lemma}\textbf{$($Uniform boundedness of relevant queueing functions$)$}\label{lmm: uniform bound}\\
		Under Assumptions \ref{assmpt: uniform} and \ref{assmpt: light tail},  there exists a finite positive constant $M>0$ such that for any sequences $(\mu_k,p_k)\in\mathcal{B}$ and $D_k\geq 1$, we have, for all $k\geq 1$, $1\leq n\leq D_k$ and $1\leq m\leq 4$, and $\eta>0$ as defined in Assumption \ref{assmpt: light tail},
		\begin{equation*}
			\EE[(W_n^k)^m], \quad \EE[(X_n^k)^m], \quad \EE[(Q_k)^m], \quad \EE[\exp(\eta W_n^k)]\quad \text{and}\quad \EE[\exp(\eta Q_k)]
		\end{equation*}
		are all bounded by $M$.
	\end{lemma}
	\section{Regret Analysis}\label{sec: framework}
	The online learning approach described in Section \ref{subsec: algorithm outline} is a data-driven method, it should contunue to generate improved solutions that will eventually converge to the true optimal solution as the server's experience accumulates (by serving more and more customers). The performance of GOLiQ is measured by the so-called \textit{regret}, which can be interpreted as the cost to pay, over the time or the number of samples, for the algorithm to learn the optimal policy. In this section, we give a formal definition of the regret  and conduct the regret analysis for our online learning algorithm. 

	The expected net cost of the queueing system incurred in cycle $k$ is
	\begin{align}\label{regret_rhok}
		\rho_k = \EE\left[\sum_{n=1}^{Q_k\wedge D_k}(h_0(W_n^k+S_n^k)-p^k_n)+\sum_{n=Q_k+1}^{D_k}(h_0(W_n^k+S_n^k)-p_{k})+c(\mu_k)T_k\right],
	\end{align}
	where the summation $\sum_{n=Q_k+1}^{D_k}\cdot$ is $0$ in case $D_k<Q_k+1$. The total regret accumulated in the first $L$ cycles is
	\begin{align}\label{regret_def}
		R(L)\equiv \sum_{k=1}^L R_k, \quad \text{where}\quad R_k\equiv \rho_k-f(\mu^*,p^*)\EE[T_k]
	\end{align}	
	is regret in cycle $k$ (the expected system  cost in cycle $k$ minus the optimal cost).
	\begin{remark}
		Following \cite{Huh2009} and \cite{Shi2020}, our regret defined in \eqref{regret_def} is computed by accumulating the difference between the steady-state maximum profit under $(\mu^*,p^*)$ and 
		the expected profit earned under GOLiQ. 
		However, one may find such a definition to be somewhat too demanding; it appears to be more reasonable if we were to benchmark with the nonstationary dynamics under $(\mu^*, p^*)$,  rather than the steady-state performance. Nevertheless, our numerical studies confirm that the nuance of the two aforementioned regret definitions is negligible. See Section \ref{sec:RegretAlternative}. 
	\end{remark}

	\paragraph{\textbf{Separation of regret.}} To treat the total regret defined in \eqref{regret_def},  we separate it into two parts: \textit{regret of nonstationarity} which quantifies the error due to the system's transient performance, and \textit{regret of suboptimality} which accounts for the suboptimality error due to the present policy. In detail, we write
	\begin{align}\label{eq:regretk}
		R_k= \underbrace{(\rho_k -\EE[f(\mu_k,p_k)T_k])}_{\equiv R_{1,k}} + \underbrace{\EE[T_k(f(\mu_k,p_k)-f(\mu^*,p^*))]}_{\equiv R_{2,k}},
	\end{align}
	so that
	\begin{align}\label{regret_decomp}
		R(L)= \sum_{k=1}^L R_{1,k} + \sum_{k=1}^L R_{2,k}\equiv R_1(L) + R_2(L).
	\end{align}
	Intuitively, $R_{1,k}$ measures the performance error due to transient queueing dynamics (regret of nonstationarity), while $R_{2,k}$ accounts for the suboptimality error of control parameters $(\mu_k,p_k)$ (regret of suboptimality).
	
	
	In what follows, we will analyze the two terms $R_1(L)$ and $R_2(L)$ separately. To treat $R_1(L)$, we develop in  Section \ref{subsec: transient analysis} a new framework to analyze the transient queueing behavior using the coupling technique (Theorem \ref{thm: non-stationary error}). The development of the theoretical bound for $R_2(L)$ is given in Section \ref{subsec: convergence rate} (Theorem \ref{thm: broadiezeevi}).
	Results in these sections provide convenient conditions that facilitate the convergence and regret bound analysis of our GOLiQ algorithm for GI/GI/1 queues (which is to be given in Section \ref{sec: direct}). The roadmap of the theoretical analysis is depicted in Figure \ref{fig:ProofSteps}.
	
	\begin{figure}[h]
		\vspace{-0.1in}
		\centering
		\includegraphics[width=1.02\textwidth, height = 3.2in]{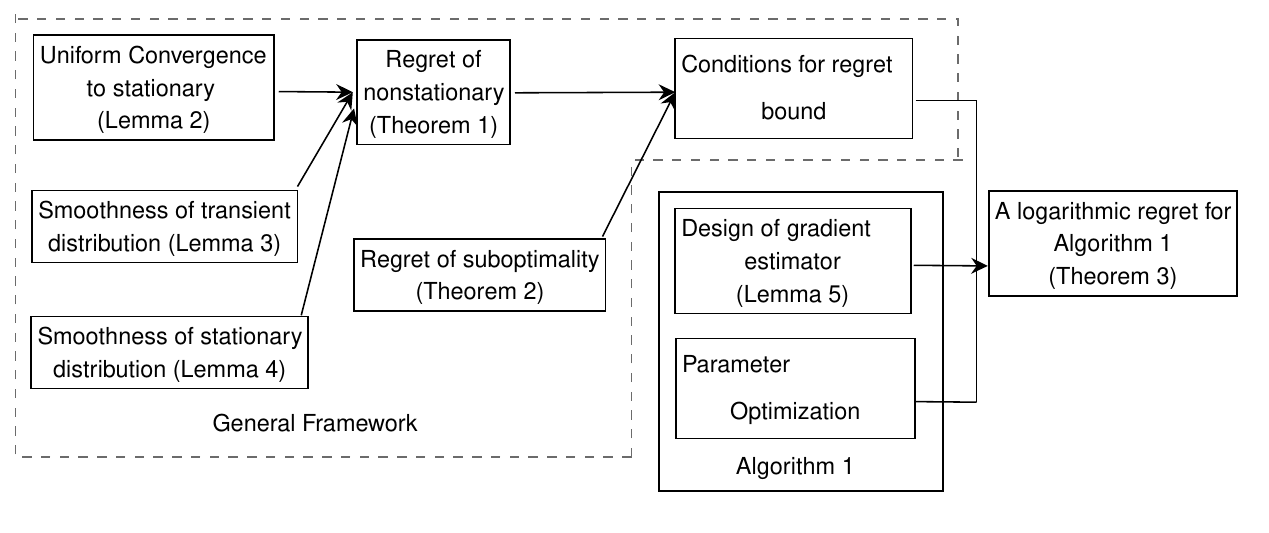}
		\vspace{-0.35in}
		\caption{Roadmap of regret analysis and algorithm design}\label{fig:ProofSteps}
		\vspace{-0.1in}	
	\end{figure}
	
	\subsection{Regret of Nonstationarity}\label{subsec: transient analysis}
	In this part, we analyze the transient queueing dynamics, base on which we develop a theoretical upper bound for $R_1(L)$. As we shall see later in Section \ref{sec: direct}, this analysis is also essential to bounding the bias $B_k$ and variance $\mathcal{V}_k$ of the gradient estimators for GOLiQ.
	
	\textbf{A crude $O(L)$ bound.} Roughly speaking, since the parameters $\mu, p$ and functions $\lambda(\cdot)$, $c(\cdot)$ are all bounded, the regret $R_1(L)$ is in the same order as the transient bias of the waiting time process, i.e.,
	\begin{equation*}
		\begin{aligned}
			R_1(L)\approx \sum_{k=1}^L O\left(\sum_{n=1}^{D_k}\left(\EE[W_n(\mu_k,p_k)]-\EE[W_\infty(\mu_k,p_k)]\right) \right).
		\end{aligned}
	\end{equation*}
	Here we use $W_\infty(\mu,p)$ to denote the steady-state waiting time of the $GI/GI/1$ queue with parameter   $(\mu,p)\in \mathcal{B}$. Under the uniform stability condition (Assumption \ref{assmpt: uniform}), it is not difficult to show that there exist positive constants $\gamma>0$ and $K>0$, independent of $k$ and $(\mu_k,p_k)$ such that
	$$\left|\EE[W_n^k] - \EE[W_\infty(\mu_k,p_k)]\right|\leq e^{-\gamma n}K.$$
	Then, as a direct consequence,  we have
	$$\sum_{n=1}^{D_k}\left(\EE[W_n(\mu_k,p_k)]-\EE[W_\infty(\mu_k,p_k)]\right)\leq \frac{K}{1-e^{-\gamma}}~\Rightarrow~ R_1(L)=O(L).$$
	An analogue of the above $O(L)$ bound is given by \cite{Huh2009} (Lemma 11) in an inventory model.
	
	\textbf{An improved $o(L)$ bound.} In the rest of this subsection, we will conduct a more delicate analysis on the transient performance of the queueing system, and our analysis will render a (tighter) sub-linear bound $R_1(L)=o(L)$ (of which the exact order depends on the concrete algorithm, as we shall see later). 
	
	\begin{theorem}\label{thm: non-stationary error}\textbf{$($Regret of nonstationarity$)$} Suppose that Assumptions \ref{assmpt: uniform} and \ref{assmpt: light tail} hold. In addition, assume that the following conditions are satisfied for some constant $K_2>0$ and $0<\alpha\leq 1$:
		\begin{enumerate}
			\item[$(a)$] $\lceil6\log(k)/\min(\gamma,\eta)\rceil\leq D_k\leq K_2k^{2-\alpha}$;
			\item[$(b)$]  $\EE[\|x_k-x_{k+1}\|^{2}]\leq K_2 k^{-2\alpha}$,
		\end{enumerate}
		where the constants $\eta$ and $\gamma$ are defined in Assumption \ref{assmpt: light tail}. Then, there exists a positive constant $K>0$ such that
		\begin{equation}\label{eq:R_1}
			R_{1,k}\leq K\cdot k^{-\alpha}\log(k), \quad k\geq 2~ \quad\text{and}\quad  R_1(L) \leq K\sum_{k=1}^L k^{-\alpha}\log(k), \quad L\geq 2.
		\end{equation}
		
	\end{theorem}
	\begin{remark}
		As will become clear later in Section \ref{sec: direct}, we obtain a bound  $R_1(L)=O(\log(L)^{{2}})$ for Algorithm \ref{alg: direct} by validating Condition (b) in Theorem \ref{thm: non-stationary error} with $\alpha =1$, which is much tighter than the crude $O(L)$ bound. This $O(\log(L)^{{2}})$ bound for $R_1(L)$ is critical to achieving an overall logarithmic regret bound in the total number of served customers. 
		An explicit expression of constant $K$ is given in \eqref{eq: K}. 
\end{remark}

\subsubsection{{Roadmap of the proof of Theorem \ref{thm: non-stationary error}}} Our point of departure in proving Theorem \ref{thm: non-stationary error} is to decompose $R_{1,k}$ into three terms. We shall split each cycle into a warm-up period consisting the first  $\tilde{d}_k=\lceil 5\log(k)/\min(\gamma,\eta)\rceil<D_k$ customers and the near-stationary period consisting of all remaining customers, where $\gamma, \eta>0$ are as defined in Assumption \ref{assmpt: light tail}. 
The three parts are: transient error in the near-stationary period ($I_1$), transient error in the warm-up period ($I_2$) and the remaining error ($I_3$). The detailed separation is given below
\begin{equation*}
	\begin{aligned}
		R_{1,k}&=\rho_k-\EE[f(\mu_k,p_k)T_k]\\
		&=\EE\left[\sum_{n=1}^{Q_k\wedge D_k}(h_0(W_n^k+S_n^k)-p_n^{k})+\sum_{n=Q_k+1}^{D_k}(h_0(W_n^k+S_n^k)-p_{k})+c(\mu_k)T_k -f(\mu_k,p_k)T_k\right]\\
		& = h_0\underbrace{\EE\left[\sum_{n=\tilde{d}_k+1}^{D_k} \left(W_n^k-w(\mu_k,p_k)\right)\right]}_{\equiv I_1} +h_0\underbrace{\EE\left[\sum_{n=1}^{\tilde{d}_k} \left(W_n^k-w(\mu_k,p_k)\right)\right]}_{\equiv I_2}\\
		&~~~+\underbrace{\EE\left[(D_k-\lambda_{k}T_k)(h_0 w(\mu_k,p_k)+\frac{h_0}{\mu_k}-p_k)\right]+\EE\left[\sum_{n=1}^{Q_k\wedge D_k}(p_k-p_n^k)\right]}_{\equiv I_3}.
	\end{aligned}
\end{equation*}
The term $w(\mu,p)\equiv \EE[W_\infty(\mu,p)]$ is a function in $(\mu,p)$ and equals to the steady-state expected waiting time under parameter $(\mu,p)\in\mathcal{B}$. 
To prove $R_{1,k} = O(k^{-\alpha}\log(k))$, it suffices  to show that $I_i=O(k^{-\alpha}\log(k))$ for $i=1,2,3$. Below we  explain the main ideas of our treatment to $I_1$, $I_2$ and $I_3$:
\begin{itemize}
	\item \textbf{$I_1$:} We will first show that, after serving $d_k\equiv \lceil4\log(k)/\min(\gamma,\eta)\rceil<\tilde{d}_k$ customers, with a sufficiently high probability, all $Q_k$ existing customers have left the system and $\{W_n^k:n=d_k,...,D_{k}\}$ follows the dynamic of a GI/GI/1 queue with arrival rate $\lambda_{k}$ and service rate $\mu_k$. Then, we show that $W_n^k$, for $n\geq d_k$, will converge exponentially fast to the steady state (Lemma \ref{lmm: thm1 1}). Hence $W_n^k$ is close to  $W_\infty(\mu_k,p_k)$ for $n\geq \tilde{d}_k$, warranting a small   transient error $I_1$.
	\item \textbf{$I_2$:} Note  that the $\tilde{d}_k$ customers in the warm-up period include those leftovers from previous periods, and their arrival rates $\lambda_n^k$ are different from $\lambda_k$. To control the impact of such difference between $\lambda_n^k$ and $\lambda_k$,  we first establish almost sure Lipschitz continuity of waiting times (for queues having customer-heterogeneous arrival rates) with respect to the arrival rate sequence and the initial state (Lemma \ref{lmm: Lipschitz}). As a consequence, we can prove that $|\EE[W_n^k-w(\mu_{k-1},p_{k-1})]|=O(k^{-\alpha})$ taking advantage of the fact that the initial state $W_0^k=W_{D_{k-1}}^{k-1}$ is close to the steady-state $W_\infty(\mu_{k-1},p_{k-1})$. Then, we show that  the steady-state distribution is smooth in the parameter $(\mu, p)$ (Lemma \ref{lmm: steady state continuity}), i.e., $\EE[|w(\mu_{k-1},p_{k-1})-w(\mu_k,p_k)|]=O(\EE|\mu_k-\mu_{k-1}|+\EE|p_k-p_{k-1}|)=O(k^{-\alpha})$, which completes the analysis for $I_2$.
	\item \textbf{$I_3$:} The term $I_3$ will be under control because $W_{D_k}^k$ is close to the steady-state (Lemma \ref{lmm: thm1 1}) and $Q_k$ is uniformly bounded (Lemma \ref{lmm: uniform bound}).
\end{itemize}
Also see in Figure \ref{fig: Thm1RoadMap} for a graphical illustration. 
\begin{figure}
	\centering
	\includegraphics[scale = 0.5]{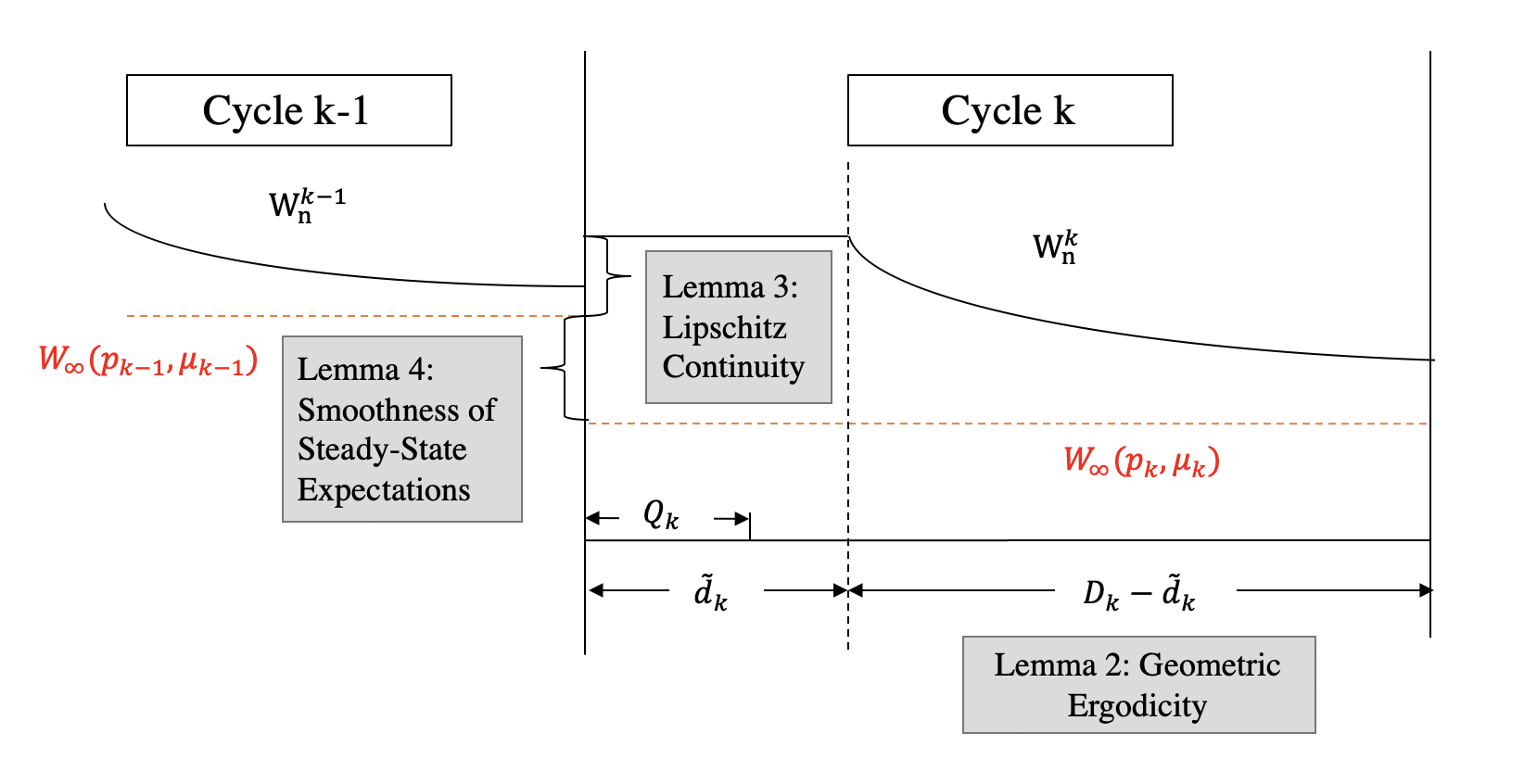}
	\caption{Roadmap of the analysis of the regret of nonstationarity.}\label{fig: Thm1RoadMap}
\end{figure}

Following the above roadmap, we next give detailed analysis for $I_i, i=1,2,3$ by establishing three lemmas (Lemmas \ref{lmm: thm1 1}--\ref{lmm: steady state continuity}). We believe that these results are not only essential to the transient analysis in the present paper, but may also be of independent interest for theoretic studies of other queueing models.\\


\textbf{Bounding $I_1$.}
We first establish the rate at which waiting times converge to their steady state distributions. For two given sequences $V_n$ and $U_n$, we say two $GI/GI/1$ queues with the same parameter $(\mu,p)\in \mathcal{B}$ are \textit{synchronously coupled} if their waiting times  $W^1_n$ and $W^2_n$ satisfy
$$W^i_n = \left(W^i_{n-1}+\frac{V_n}{\mu}-\frac{U_n}{\lambda(p)}\right)^+,\text{ for }i=1,2,\text{ and }n \geq 1,$$
i.e., the two systems share the same sequences of service and interarrival times.
The proof of Lemma \ref{lmm: thm1 1} is given in Section \ref{sec: transient analysis proofs}. 

\begin{lemma}\label{lmm: thm1 1}\textbf{$($Exponential loss of memory of initial state$)$}
	Suppose two $GI/GI/1$ queues with parameter $(\mu,p)\in\mathcal{B}$  are synchronously coupled with initial waiting times $W^1_0$ and $W^2_0$, respectively.
	Then, for the two positive constants $\gamma$ and $\theta$ defined in Assumption \ref{assmpt: light tail} and any $m\geq 1$, we have, conditional on $W^1_0$ and $W^2_0$,
	$$\EE\left[|W^1_n-W^2_n|^m~|~W^1_0,W^2_0\right]\leq e^{-\gamma n}(2+e^{\mu\theta W^1_0}+e^{\mu\theta W^2_0})|W_0^1-W^2_0|^m.$$
\end{lemma}

In order to bound $I_1$, at the beginning of each cycle $k$, given $(\mu_k,p_k)$, we couple $W_0^k$ with $\bar{W}_0^k$ that is independently drawn from the steady-state waiting time distribution $W_\infty(\mu_k,p_k)$. The sequence $\bar{W}_n^k$ is defined as
$$\bar{W}_n^k = \left(\bar{W}_{n-1}^k+\frac{V_n^k}{\mu_k}-\frac{U_n^k}{\lambda_k}\right)^+, \text{ for all }1\leq n\leq D_k.$$
Then, by definition, conditional on $(\mu_k,p_k)$, $\EE[\bar{W}_n^k]=w(\mu_k,p_k)$ for all $1\leq n\leq D_k$, and therefore, 
$$\left|\EE[W_n^k-w(\mu_k,p_k)]\right|\leq \EE[|W_n^k-\bar{W}_n^k|].$$
As we will show in the proof of Corollary \ref{coro: convergence to stationarity}, $\{W_n^k:n=d_k+1,...,D_k\}$ is  coupled with $\bar{W}_n^k$ except on a set of negligible set, with $d_k\equiv \lceil 4\log(k)/\min(\gamma,\eta)\rceil<\tilde{d}_k$. As a result, we can use Lemma \ref{lmm: thm1 1} to construct a bound on $\EE[|W^k_n-\bar{W}^k_n|]$ for $n=\tilde{d}_k+1,...,D_k$.
\begin{coro}\label{coro: convergence to stationarity}
	Under the conditions of Theorem \ref{thm: non-stationary error}, there exists a constant $A\geq 1$ independent of $k$ and $(\mu_k,p_k)$, such that for all $k\geq 1$ and $n\geq d_k\equiv \lceil 4\log(k+1)/\min(\gamma,\eta)\rceil$,
	\begin{equation}
		\EE[|W_n^k - \bar{W}_n^k|]\leq e^{-\gamma(n-d_k)}A + 2Mk^{-2}.
	\end{equation}
	As a direct consequence, we have
	$I_1=O(k^{-\alpha})$.
\end{coro}

\textbf{Bounding $I_2$.}
We first show that the waiting times $W_n$ of a queueing model having customer-heterogeneous arrival rates are Lipschitz continuous with respect to the rates $(\mu_n,\lambda_n)$ and the initial state almost surely. 
\begin{lemma}\label{lmm: Lipschitz}\textbf{(Lipschitz continuity)}
	Consider two waiting time sequences $W_n$ and $\tilde{W}_n$ for $n\geq 1$ with initial values $W_0$ and $\tilde{W}_0$ respectively. Let $(\mu_n,\lambda_n)$ and $(\tilde{\mu}_n,\tilde{\lambda}_n)\in \mathcal{B}$ be the corresponding sequences of service and arrival rates, respectively, i.e.,
	$$W_n = \left(W_{n-1}+\frac{V_n}{\mu_n}-\frac{U_n}{\lambda_n}\right)^+\quad\text{and}\quad \tilde{W}_n = \left(\tilde{W}_{n-1}+\frac{V_n}{\tilde{\mu}_n}-\frac{U_n}{\tilde{\lambda}_n}\right)^+, \quad \text{ for }n\geq 1.$$
	Suppose there exist two constants $c_\mu, c_\lambda>0$ such that
	$$|\mu_n-\tilde{\mu}_n|\leq c_\mu \quad \text{and}\quad |\lambda_n-\tilde{\lambda}_n|\leq c_\lambda, \quad \text{ for all }n\geq 1.$$
	Then we have, for all $n\geq 1$, 
	$$|W_n-\tilde{W}_n|\leq |W_0-\tilde{W}_0|+\left(\frac{c_\mu}{\underline{\mu}}+\frac{c_\lambda}{\underline{\lambda}}\right)\max(X_n,\tilde{X}_n)+\frac{c_\mu}{\underline{\mu}}\max(W_n,\tilde{W}_n),$$
	where $X_n$ and $\tilde{X}_n$ are the corresponding observed busy periods. In particular, $X_n$ and $\tilde{X}_n$ satisfy the recursion \eqref{eq:RecBusyPeriod} defined in Section \ref{subsec: queueing dynamics} with any given initial values of $X_0\geq 0$ and $\tilde{X}_0\geq 0$.
\end{lemma}

\vskip 1ex

As discussed above, controlling $I_2$ also involves bounding the difference between the mean steady-state waiting times in two consecutive cycles. Hence,  we next establish a uniform high-order smoothness result for the steady-state waiting times with respect to the model parameter $(\mu,p)$.

\begin{lemma}\label{lmm: steady state continuity}\textbf{$($Smoothness in $\mu$ and $p)$}
	Suppose $(\mu_i,p_i)\in \mathcal{B}$ for $i=1,2$. Let $W_\infty(\mu_i,p_i)$ be the steady-state waiting time of the GI/GI/1 queue under parameter $(\mu_i,p_i)$, respectively. Then, the steady-state waiting times $(W_\infty(\mu_1,p_1),W_\infty(\mu_2,p_2))$ can be coupled such that, there exists a constant $B>0$ independent of $(\mu_i,p_i)$ satisfying that, for all  $1\leq m\leq 4$,
	$$\EE[|W_\infty(\mu_1,p_1)-W_\infty(\mu_2,p_2)|^m]\leq B\left(|\mu_1-\mu_2|^m+|p_1-p_2|^m\right),$$
	where a closed-form expression of constant $B$ is given in \eqref{eq: B}.
\end{lemma}
We adopt a ``\textit{coupling from the past}" (CFTP) approach in the proof of Lemma \ref{lmm: steady state continuity} (see Section \ref{sec: transient analysis proofs}). 
Roughly speaking, CFTP is a synchronous coupling starting from infinite past. In the proof of Lemma \ref{lmm: steady state continuity}, we shall explicitly explain how to construct the CFTP.

Now we are ready to analyze  $I_2$. Essentially, we shall compare $\EE[W_n^k]$ in the warm-up period with $w(\mu_{k-1},p_{k-1})=\EE[W_\infty(\mu_{k-1},p_{k-1})]$. For each cycle $k$, recall that we have already coupled $W_n^{k-1}$ with a stationary sequence $\bar{W}_n^{k-1}$ in cycle $k-1$, {we then extend the sequence  $\bar{W}_n^{k-1}$ to cycle $k$} in the sense that
$$\bar{W}_{D_{k-1}+n}^{k-1} =\left(\bar{W}_{D_{k-1}+n-1}^{k-1}+\frac{V_n^k}{\mu_{k-1}}-\frac{U_n^k}{\lambda_{k-1}}\right)^+, \text{ for }n=1, 2, ..., D_{k}.$$
Then, conditional on $(\mu_{k-1},p_{k-1})$, $\EE[\bar{W}_{D_{k-1}+n}^{k-1}]=w(\mu_{k-1},p_{k-1})$. So we have
\begin{align*}
	\left|\EE[W_n^k-w(\mu_k,p_k)]\right|& \leq \left|\EE[W_n^k-w(\mu_{k-1},p_{k-1})]\right| + \EE\left[|w(\mu_{k-1},p_{k-1})-w(\mu_k,p_k)|\right]\\
	&\leq \EE\left[|W_n^k-\bar{W}_{D_{k-1}+n}^{k-1}|\right]+\EE\left[|w(\mu_{k-1},p_{k-1})-w(\mu_k,p_k)|\right].
\end{align*}
Bounding the first term by Lemma \ref{lmm: Lipschitz} and the second term by Lemma \ref{lmm: steady state continuity} yields the following bound on $I_2$.
\begin{coro}\label{coro: error of leftovers}
	Under the conditions of Theorem \ref{thm: non-stationary error}, for all $k\geq 2$ and $1\leq n\leq D_k$, we have
	\begin{equation}
		\EE[|W_n^k-w(\mu_k,p_k)|]=O(k^{-\alpha}). \quad	
	\end{equation}
	As a direct consequence, $|I_2|=O(k^{-\alpha}\log(k))$.
\end{coro}

\textbf{Bounding $I_3$.} We complete our analysis on the regret of nonstationarity by showing that $I_3=O(k^{-\alpha})$. The proof of Corollary \ref{coro: I3} below basically follows from Lemma \ref{lmm: uniform bound} and Lemma \ref{lmm: thm1 1} with some similar argument as used in the proof of Corollary \ref{coro: error of leftovers}. 
\begin{coro}\label{coro: I3}
	Under the conditions of Theorem \ref{thm: non-stationary error}, $|I_3|=O(k^{-\alpha})$.
\end{coro}

\textbf{Finishing the Proof of Theorem \ref{thm: non-stationary error}.} Then, Theorem \ref{thm: non-stationary error} follows immediately from Corollaries \ref{coro: convergence to stationarity} to \ref{coro: I3}. A complete proof of Theorem \ref{thm: non-stationary error}, including the proofs of Corollaries \ref{coro: convergence to stationarity} to \ref{coro: I3}, is given in Section \ref{subsect: non stationary proof} of e-companion. In particular, we provide an explicit expression of the constant $K$ in terms of the model parameters in \eqref{eq: K}.

\begin{remark}
	We advocate that Theorem \ref{thm: non-stationary error} may apply to other queueing models (its scope is beyond the $GI/GI/1$ queue), as long as one can verify three conditions for the designated model: (i) uniform boundedness for the rate of convergence to the steady state, i.e., Lemma \ref{lmm: thm1 1}, (ii) path-wise Liptschize continuity, i.e., Lemma \ref{lmm: Lipschitz}, and (iii) smoothness of the stationary distributions in the control variables, i.e., Lemma \ref{lmm: steady state continuity}. 
\end{remark}

\subsection{Regret of Suboptimality}\label{subsec: convergence rate}
To bound  the regret of suboptimality $R_2(L)$, we need to control the rate at which $x_k$ converges to $x^*$. This depends largely on the effectiveness of the estimator $H_k$ for $\nabla f(x_k)$. In our algorithm, such effectiveness is measured by the bias $B_k$ and variance $\mathcal{V}_k$.
The following result shows that, if $B_k$ and $\mathcal{V}_k$ can be appropriately bounded, then, $x_k$ will converge to $x^*$ rapidly and hence $R_2(L)$ can be properly bounded.
\begin{theorem}\label{thm: broadiezeevi}\textbf{$($Regret of suboptimality$)$}
	Suppose Assumptions \ref{assmpt: convexity} holds. If there exists a constant $K_3\geq 1$ such that the following conditions hold for all $k$,
	\begin{enumerate}
		\item[$(a)$] $\left(1+\frac{1}{k}\right)^\beta\leq 1 +\frac{K_0}{2}\eta_k$,
		\item[$(b)$] $B_k \leq  \frac{K_0}{8}k^{-\beta}$,
		\item[$(c)$]  $\eta_k\mathcal{V}_k \leq K_3k^{-\beta}$,
	\end{enumerate}
	where $0<\beta\leq 1$ is a constant, and $\eta_k\to 0$ is the step size,
	then, there exists a constant $C\geq 8K_3/K_0$ with an explicit expression given in \eqref{eq: C}, such that for all $k\geq 1$,
	\begin{equation}\label{eq: convergence rate}
		\EE[\|x_k-x^*\|^2]\leq Ck^{-\beta},
	\end{equation}
	and as a consequence, 
	\begin{equation}\label{eq:R_2}
		R_2(L)\leq CK_1\sum_{k=1}^L\left( \frac{D_k}{\lambda(\bar{p})}+M\right)k^{-\beta}=O\left(\sum_{k=1}^L D_kk^{-\beta}\right).
	\end{equation}
\end{theorem}
\begin{remark}[Selecting the ``optimal" $D_k$] The above expression \eqref{eq:R_2} indicates a trade-off in the selection of the parameter $D_k$. On the one hand, increasing the sample size $D_k$ reduces the bias $B_k$ for the gradient estimator, and hence leads to a smaller value of  $k^{-\beta}$. On the other hand, a larger $D_k$ makes the system operate under a sub-optimal decision for a longer time. To this end, one may choose an optimal order (in $k$) for $D_k$  by minimizing the order of the regret as in \eqref{eq:R_2}.
\end{remark}

Our proof of Theorem \ref{thm: broadiezeevi} follows an inductive approach as used in \cite{BoradieZeevi2011}. Let $b_k\equiv \EE[\|x_k-x^*\|^2] $. 
According to the SGD iteration  $x_{k+1}=\Pi_\mathcal{B}(x_k-\eta_kH_k)$, we have
$$\EE[\|x_{k+1}-x^*\|^2|x_k]\leq \EE[\|x_k-\eta_kH_k-x^*\|^2|x_k]=\|x_k-x^*\|^2- 2\eta_k \EE[H_k|x_k](x_k-x^*) +\eta_k^2\EE[\|H_k\|^2|x_k].$$
Then, by Assumption \ref{assmpt: convexity} and the definition of $B_k, \mathcal{V}_k$ by \eqref{eq: BV}, we derive the following recursive inequality for $b_k$:
\begin{equation*}
	\begin{aligned}
		b_{k+1}\leq (1-K_0\eta_k+\eta_k B_k)b_k + \eta_k B_k +\eta^2_k \mathcal{V}_k, \quad k\geq 1,
	\end{aligned}
\end{equation*}
and we prove \eqref{eq: convergence rate} by induction. The full proof is given in Section \ref{subsect: R1 analysis} of the e-companion.

In Section \ref{sec: direct}, we apply Theorem \ref{thm: broadiezeevi} to treat our online learning algorithm (Algorithm \ref{alg: direct}) by verifying that Conditions (a)--(c) are satisfied.
Because in Theorem \ref{thm: broadiezeevi}, Conditions (a)--(c) are stated explicitly in terms of the step size $\eta_k$, bias $B_k$ and variance $\mathcal{V}_k$ of the gradient estimator, these conditions may serve as useful building blocks for the design and analysis of online learning algorithms in other queueing models as well.

\section{GOLiQ for the $GI/GI/1$ Queue}\label{sec: direct}
In this section, we provide a concrete GOLiQ algorithm that solves the optimal pricing and capacity sizing problem \eqref{eq: maximizing revenue} for a $GI/GI/1$ queueing system. 
We show that the gradient $\nabla f(\mu,p)$ can be estimated ``directly" from past experience (i.e., data of delay and busy times generated under the present policy).  Applying the regret analysis developed in Section \ref{sec: framework}, we provide a theoretic upper bound for the overall regret in Theorem \ref{thm: regret direct}.

\subsection{A Gradient Estimator}
Following the algorithm framework outlined in Section \ref{subsec: algorithm outline}, we now develop a detailed gradient estimator  $H_k$. Regarding 
the objective function in \eqref{eq:obj},
it suffices to construct estimators for the partial derivatives
\begin{align}\label{eq:partials}
	\frac{\partial}{\partial \mu}\EE[W_\infty(p,\mu)] \qquad \text{and} \qquad \frac{\partial}{\partial p}\EE[W_\infty(p,\mu)].
\end{align}
Following the \textit{infinitesimal perturbation analysis} (IPA) approach
(see, for example, \cite{Glasserman_1992}), we next show that the partial derivatives in \eqref{eq:partials} can be expressed in terms of the  steady-state distributions $W_\infty(p,\mu)$ and   $X_\infty(p,\mu)$  of the waiting time process $W_n$ and observed busy period process $X_n$, of which the  dynamics are characterized by \eqref{eq:RecDelay}--\eqref{eq:RecBusyPeriod}. 

\begin{lemma}\label{lmm: derivative process}
	Suppose Assumptions 1 and 2 holds. Then, for any $(\mu,p)\in\mathcal{B}$, $\EE[W_\infty(\mu,p)]$ are differentiable in $\mu$ and $p$. Besides,
	\begin{equation}\label{eq: derivative}
		\begin{aligned}
			\frac{\partial}{\partial p}f(\mu,p)&= -\lambda(p) -p\lambda'(p) + h_0\lambda'(p)\left(\EE[W_\infty(\mu,p)]+\EE[X_\infty(\mu,p)]+\frac{1}{\mu}\right)\\
			\frac{\partial}{\partial \mu}f(\mu,p)& = c'(\mu) - h_0\frac{\lambda(p)}{\mu}\left(\EE[W_\infty(\mu,p)] + \EE[X_\infty(\mu,p)]  + \frac{1}{\mu}\right)
		\end{aligned}
	\end{equation}
\end{lemma}

\begin{proof}[Proof of Lemma \ref{lmm: derivative process}]
	To prove Equation \eqref{eq: derivative}, it suffices to work with the partial derivatives of the steady-state expectation $\EE[W_\infty(\mu,p)]$. We follow the IPA analysis in \cite{Glasserman_1992} and \cite{WSC2014}. 
	
	Given $(\mu,p)$, we define $r(p)=1/\lambda(p)$ and rewrite the recursion \eqref{eq:RecDelay}  as
	$$W_n(\mu,p) = \left(W_{n-1}(\mu,p)+\frac{V_n}{\mu} - r(p)U_n\right)^+.$$
	Define the derivative process $Z_n \equiv \frac{\partial}{\partial r}W_n(\mu,p)$, then by chain rule, we have
	$$Z_n = \frac{\partial}{\partial r}W_n(\mu,p) = \frac{\partial}{\partial r}\left(W_{n-1}(\mu,p)+\frac{V_n}{\mu} - rU_n\right)^+ =
	\begin{cases}
		\frac{\partial}{\partial r}W_{n-1} - U_n = Z_{n-1} - U_n &\text{ if }W_n>0;\\
		0&\text{ if }W_n=0.
	\end{cases}$$
	and obtain a recursion $Z_n = (Z_{n-1} - U_n){\bf 1}_{\{W_n>0\}}$.
	Let $\tilde{Z}_n \equiv -Z_n/\lambda(p)$. Then, it is straightforward to see that $\tilde{Z}_n$ satisfies the recursion given in \eqref{eq:RecBusyPeriod} as the observed busy period $X_n$, i.e.,
	$$\tilde{Z}_n =
	\left(\tilde{Z}_{n-1} + \frac{U_n}{\lambda(p)}\right){\bf 1}(W_n>0).$$
	Under the assumption that the queueing system is stable, the limit $\tilde{Z}_\infty$ should be equal in distribution to $X_\infty$. Therefore, we formally derive
	\begin{equation}\label{eq: IPA}
		\frac{\partial}{\partial r}\EE[W_\infty(\mu,p)] = \EE[Z_\infty] =-\lambda(p) \EE[\tilde{Z}_\infty] = -\lambda(p)\EE[X_\infty(\mu,p)].
	\end{equation}
	The above heuristics can be made rigorous by verifying exchanges of limits using the results in \cite{Glasserman_1992}, and we refer the readers to Section \ref{subsec: glasserman} for detailed explanations.  
	Using \eqref{eq: IPA}, we can derive the partial derivative of the steady-state waiting time with respect to price $p$ as below:
	\begin{align*}
		\frac{\partial}{\partial p}\EE[W_\infty(\mu,p)] =\frac{\partial}{\partial r}\EE[W_\infty(\mu,p)] \frac{\partial r(p)}{\partial p} = -\lambda(p)\EE[X_{\infty}(\mu,p)]\cdot - \frac{\lambda'(p)}{\lambda(p)^2} = \EE[X_{\infty}(\mu,p)]\frac{\lambda'(p)}{\lambda(p)}.
	\end{align*}
	Now we turn to $\frac{\partial}{\partial\mu}\EE[W_\infty(\mu,p)]$. Let $\hat{Z}_n \equiv \mu W_n(\mu,p)$, it is easy to check that $\hat{Z}_n = \left(\hat{Z}_{n-1} + V_n - \mu U_n/\lambda(p)\right)^+$. Then, following steps similar to those for \eqref{eq: IPA}, we have
	$$\frac{\partial}{\partial \mu}\EE[\hat{Z}_\infty(\mu,p)] =- \EE[X_{\infty}(\mu,p)].
	$$
	Therefore,
	\begin{equation*}
		\begin{aligned}
			- \EE[X_\infty(\mu,p)]=\frac{\partial}{\partial \mu}\EE[\hat{Z}_\infty(\mu,p)]=\frac{\partial}{\partial \mu}\EE[\mu W_\infty(\mu,p)] =\mu \frac{\partial}{\partial \mu}\EE[W_\infty(\mu,p)]  + \EE[W_\infty(\mu,p)],
		\end{aligned}
	\end{equation*}
	and hence,
	$
	\partial\EE[W_\infty(\mu,p)]/\partial \mu =-(\EE[X_\infty(\mu,p)] + \EE[W_\infty(\mu,p)])/\mu.$
	
	Finally, plugging the expressions of the two partial derivatives
	into $\nabla f$ yields \eqref{eq: derivative}. 
\end{proof}

\subsection{GOLiQ: A $G/G/1$ Version}
Utilizing results in Lemma \ref{lmm: derivative process},  we are ready to design a $G/G/1$ version of the GOLiQ algorithm, where we estimate the terms $\EE[W_\infty(\mu,p)]$ and $\EE[X_\infty(\mu,p)])$ in the partial derivatives \eqref{eq: derivative} using the finite-sample averages of $W^k_n$ and $X_n^k$ observed in each cycle $k$. The formal description of the algorithm is given in Algorithm \ref{alg: direct}.\medskip

\begin{algorithm}[h]
	\SetAlgoLined
	\KwIn{number of cycles $L$\;
		parameters $0<\xi<1$, $D_k$, $\eta_k$ for $k=1,2,...,L$\;
		initial value $x_1=(\mu_1,p_1)$\;}
	\For{$n=1,2,...,D_k$}{
		operate the system under $x_k=(\mu_k,p_k)$ until $D_k$ customers enter service\;
		observe $(W^k_n,X^k_n)$ for $n=1,2,...,D_k$\;
		randomly draw $Z\in\{1,2\}$\;
		\eIf{$Z=1$}{
			$h \leftarrow-\lambda(p_k) -p_k\lambda'(p_k)+h_0\lambda'(p_k)\left[\frac{1}{\lceil D_k(1-\xi)\rceil}\sum_{n> \xi D_k}^{D_k} \left(X^k_{n} +W^k_{n}\right) +\frac{1}{\mu_k}\right]$\;
			$H_k\leftarrow (2h,0)$\;
		}{
			$h \leftarrow c'(\mu_k) - h_0\frac{\lambda(p)}{\mu_k}\left[\frac{1}{\lceil D_k(1-\xi)\rceil}\sum_{n> \xi D_k}^{D_k} \left(X^k_{n} +W^k_{n}\right)+\frac{1}{\mu_k}\right]$\;
			$H_k\leftarrow (0,2h)$\;
		}
		\textbf{update: } $x_{k+1} = \Pi_{\mathcal{B}}(x_k-\eta_k H_k)$\;
	}
	\caption{GOLiQ for $GI/GI/1$ Queues}
	\label{alg: direct}
\end{algorithm}
\vspace{-0.2in}

\begin{remark}[On the queueing leftover] 
	We elaborate more on our treatment of $Q_k$,  the existing queue content at the beginning of cycle $k$. First, the content of $Q_k$ includes customer arrivals in cycle $k-1$ and possibly even earlier cycles. Second, it is also possible to have $Q_k>D_k$.
	Nevertheless, these above cases do not affect the implementation of Algorithm \ref{alg: direct} (note that Algorithm \ref{alg: direct} gives a gradient estimator using $\lceil(1-\xi)D_k\rceil$ samples without specifying any of the above events). 
	Of course, the event $\{Q_k>D_k\}$ does play a role in our theoretic regret analysis, but it is a rare event with a negligible probability (in fact, we show that the probability will be suppressed to $O(k^{-3})$, also see  Remark \ref{rmk: left-overs}. 
\end{remark}

\paragraph{\textbf{Selecting the ``optimal" hyperparameters.}} 
The effectiveness of Algorithm \ref{alg: direct} largely hinges upon carefully selecting the three hyperparameters: (i) the warm-up time $\xi\in(0,1)$, (ii) the learning step size $\eta_k>0$, and (iii) the exploration sample size $D_k>0$. Except for $\xi$ which has no bearing on the theoretical order of the regret, both the other two parameters $D_k$ and $\eta_k$ will play critical roles in our regret analysis. We next give the forms of the two parameters.  
First, The step size $\eta_k$ satisfies
\begin{align}\label{par:eta}
	\eta_k=c_\eta/k,\qquad \text{with}\quad c_\eta\geq 2/K_0,
\end{align}
where $K_0$ is the convexity bound specified in Assumption \ref{assmpt: convexity}.
Next, the sample size $D_k$ satisfies
\begin{align}\label{par:D}
	D_k=a_D+b_D\log(k),\quad \text{with}\quad 
	a_D\geq \frac{C_D}{\min(\gamma,\eta)\xi}\quad\text{and}\quad b_D\geq \frac{8}{\min(\gamma,\eta)\xi},
\end{align}
for any warm-up parameter $\xi\in(0,1)$, where $\gamma$ and $\eta$ are the constants specified in Assumption \ref{assmpt: light tail}, and the explicit formula of $C_D$ is given in \eqref{eq: Dk direct}.

The above-mentioned forms of $\eta_k$ and $D_k$ are obtained from our detailed regret analysis  where we show that the structure of \eqref{par:eta} and \eqref{par:D} ``minimizes" the order of the overall regret {(in the sense of maximizing $\alpha$ and $\beta$ as in Theorems \ref{thm: non-stationary error} and \ref{thm: broadiezeevi})}. 
Although the theoretical bounds of parameters $a_D$, $b_D$ and $c_\eta$ are imposed to facilitate our regret analysis, our numerical experiments show that GOLiQ remains effective even when the theoretical bounds are relaxed, {confirming the robustness of GOLiQ to these hyperparameters}; see  Section \ref{sec: para_robust} for details. 
Next, we show that Algorithm \ref{alg: direct} has a regret bound of $O((\log(M_L)^2)$ with $M_L\equiv\sum_{k=1}^L D_k$ being the cumulative number of customers served by cycle $L$. We do so by verifying that our choices of $D_k$ and $\eta_k$ (along with the corresponding $B_k$ and $\mathcal{V}_k$), will satisfy the conditions in Theorem \ref{thm: non-stationary error} and Theorem \ref{thm: broadiezeevi}.

\begin{theorem}\label{thm: regret direct}
	\textbf{$($Regret Bound for Algorithm \ref{alg: direct}$)$}\\
	Suppose Assumptions \ref{assmpt: uniform} to \ref{assmpt: convexity} hold, and $\eta_k$ and $D_k$ are selected according to \eqref{par:eta} and \eqref{par:D}.
	Then
	\begin{enumerate}
		\item[$(i)$] There exists a positive constant $K_3>0$ such that
		\begin{align*}
			B_k\leq \frac{K_0}{8k} \quad \text{and}\quad \eta_k\mathcal{V}_k\leq \frac{K_3}{k}.
		\end{align*}
		\item[$(ii)$] There exists a positive constant $K_2>0$ such that
		\begin{equation}\label{eq: switch rate}
			\EE[\|x_k-x_{k+1}\|^{2}]\leq K_2k^{-2}.
		\end{equation}
		\item[$(iii)$] As a consequence of $(i)$ and $(ii)$, the regret for Algorithm \ref{alg: direct}
		\begin{align}\label{eq:constantK}
			R(L) \leq K_{\text{alg}}\log(M_L)^{2}=O(\log(M_L)^2).
		\end{align}
	\end{enumerate}
\end{theorem}
\begin{remark}[On the logarithmic regret bound \eqref{eq:constantK}]
	Below we provide some additional discussions on the regret bound \eqref{eq:constantK}:
	\begin{enumerate}
		\item[$(i)$] {\bf On the constant $K_{\text{alg}}$.}
		The explicit expression  for the constant $K_{\text{alg}}$, although complicated, is given by \eqref{eq: K_alg}. It involves error bound corresponding to the transient behavior of the queueing system, the bias and variance of the gradient estimator, moment bounds on the queue length and other model parameters. One can verify that $K_{\text{alg}}$ is decreasing in the convergence rate coefficient $\gamma$ and increasing in the moment bounds of the queue length $M$. 
		
		\item[$(ii)$] {\bf On the first logarithmic term.	}
		Consider an SGD algorithm in that an \textit{unbiased} gradient estimator $H_k$ with a bounded variance can be evaluated using a \textit{single} data point (i.e., $B_k=0$, $\mathcal{V}_k=O(1)$), it has been proved the scaled error $k^{-1/2}(x_k-x^*)$ converges in distribution to a non-zero random variable (Theorem 2.1 in Chapter 10 of \cite{KushnerYin2003}). Hence, the  convergence rate for $\|x_k-x^*\|^2$ that any SGD-based algorithm can achieve is at best $O(k^{-1})$ (yielding a cumulative regret of order $O(\log(k))$), which is exactly the rate of convergence established by our online algorithm (taking $\beta=1$ in Theorem \ref{thm: regret direct}).
		In this sense, GOLiQ is already achieving an {``optimal" convergence rate.} We point out that, due to the nonstationary error of the queueing system, our gradient estimator is obtained using an increasing number of data points in order to guarantee a reasonably small bias. 
		\item[$(iii)$] {\bf On the second logarithmic term.}
		In order to control the regret of nonstationarity, the queueing system need to be operated in each cycle for a duration of order $O(\log(k))$. Because the queueing performance converges to its steady state \textit{exponentially} fast, this inevitably introduces an extra logarithmic term in our regret bound (which explains the ``square" in $\log(M_L)^2$).
		The question that remains open is whether this $O(\log(M_L)^2)$ bound is optimal. We conjecture that the answer is yes but admit that a rigorous treatment of a lower regret bound can be quite challenging. For example, establishing a lower regret bound requires a lower bound on the convergence rate of a $GI/GI/1$ queue, which by itself is an open question.  We leave this question to future research. 
	\end{enumerate}
	
	%
\end{remark}
\begin{remark}[Controlling the length of cycle $k$]
	We use $D_k$ (the number of customers served in cycle $k$), instead of the clock time $T_k$, to control and measure the regret bound. The benefit of using $D_k$ (rather $T_k$) as the cycle length is that it facilitates the technical analysis, because $D_k$ is directly related to the number of samples used to estimate our gradient estimator. 
	In fact, using $D_k$ instead of $T_k$ has no bearing on the order of the regret bound. To see this, note that the arrival rate is assumed to fall into a compact set $[\lambda(\bar{p}),\lambda(\underline{p})]$. Therefore, since $T_L$ is the total units of clock time elapsed after cycle $L$, we have $M_L/\lambda(\underline{p})\leq \EE[T_L]\leq M_L/\lambda(\bar{p})$ for all $L$.
\end{remark}

\section{Numerical Experiments}\label{sec: num}
To confirm the practical effectiveness of our online learning method, we conduct numerical experiments to visualize the algorithm convergence, benchmark the outcomes with known exact optimal solutions, estimate the true regret and compare it to the theoretical upper bounds.
Our base example is an $M/M/1$ queue, having Poisson arrivals with rate $\lambda(p)$, and exponential service times with rate $\mu$.
In our optimization, we consider
a commonly used logistic demand function \citep{BZ15}
\begin{align}\label{logisticLmd}
	\lambda(p)=n\lambda_0(p),\qquad  \lambda_0(p) = \frac{\exp(a-p)}{1+\exp(a-p)},
\end{align}
where $n$ is the system scale (also referred to as the market size). We also consider the following convex cost function for the service rate
\begin{align}\label{convexMu}
	c(\mu)  = c_0\mu^2.
\end{align}
See the top left panel of Figure \ref{fig:FixedMu} for $\lambda(p)$ in \eqref{logisticLmd}. In particular, the optimal pricing and staffing problem in \eqref{eq: maximizing revenue} now becomes
\begin{equation}\label{eq:ObjectiveMM1}
	\max_{\mu,p} \ \left\{ p\lambda(p) - c_0 \mu^2 -h_0 \frac{\lambda(p)/\mu}{1-\lambda(p)/\mu}\right\}.
\end{equation}

In light of the closed-form steady-state formulas of the $M/M/1$ queue, we can obtain the exact values of the optimal solutions $(\mu^*,p^*)$ and the corresponding objective value $f(\mu^*,p^*)$, with which we are able to benchmark the solutions from our online optimization algorithm. 

We first consider two one-dimensional online optimization problems in Section \ref{sec:OneDim}. We next treat the two-dimensional \textit{pricing and staffing problem} in Section \ref{sec:TwoDim}. In Section \ref{sec:AsymOpt}, we compare our results to previously established asymptotic heavy-traffic solutions in \cite{LeeWard2014}.
Additional numerical experiments are provided in the e-companion:
In Section \ref{sec: para_robust}  we investigate the robustness of GOLiQ to the hyperparameters. In Section \ref{sec: cmp_Tim}  we benchmark the performance of GOLiQ to other online learning methods. Section \ref{sec: AddNum} includes more experiments regarding the relaxation of uniform stability and GOLiQ's performance  in queues having other inter-arrival and service time distributions.

\subsection{One-Dimensional Online Optimizations}\label{sec:OneDim}

Algorithm \ref{alg: direct} covers special cases where there is only one decision variable. For example, if the service capacity $\mu$ (service fee $p$) is an exogenous parameter and the only decision is the service fee $p$ (service capacity $\mu$),  then one can simply fix $Z=1$ ($Z=2$) throughout the learning process. The theoretical regret bound (as in Theorem \ref{thm: regret direct}) for these one-dimensional cases remains unchanged. 

\subsubsection{Online optimal pricing with a fixed service capacity}
Motivated by revenue management problems in revenue generating service system, our first example focuses on the one-dimensional optimization of price $p$ with service rate $\mu = \mu_0$ held fixed. In this case we can simply omit the term $c_0 \mu^2$ in \eqref{eq:ObjectiveMM1}. Fixing the other model parameters as $a = 4.1, n = 10$, $h_0 = 1$ and $\mu_0=10$, we first obtain the exact optimal price $p^*=3.531$ (top right panel of Figure \ref{fig:FixedMu}). 
According to Algorithm \ref{alg: direct} and Theorem \ref{thm: regret direct}, we set the step size $\eta_k=1/k$ and cycle length $D_k=10+10\log(k)$.
\begin{figure}[H]
	\vspace{-0.05in}
	\centering
	\includegraphics[width=0.91\textwidth]{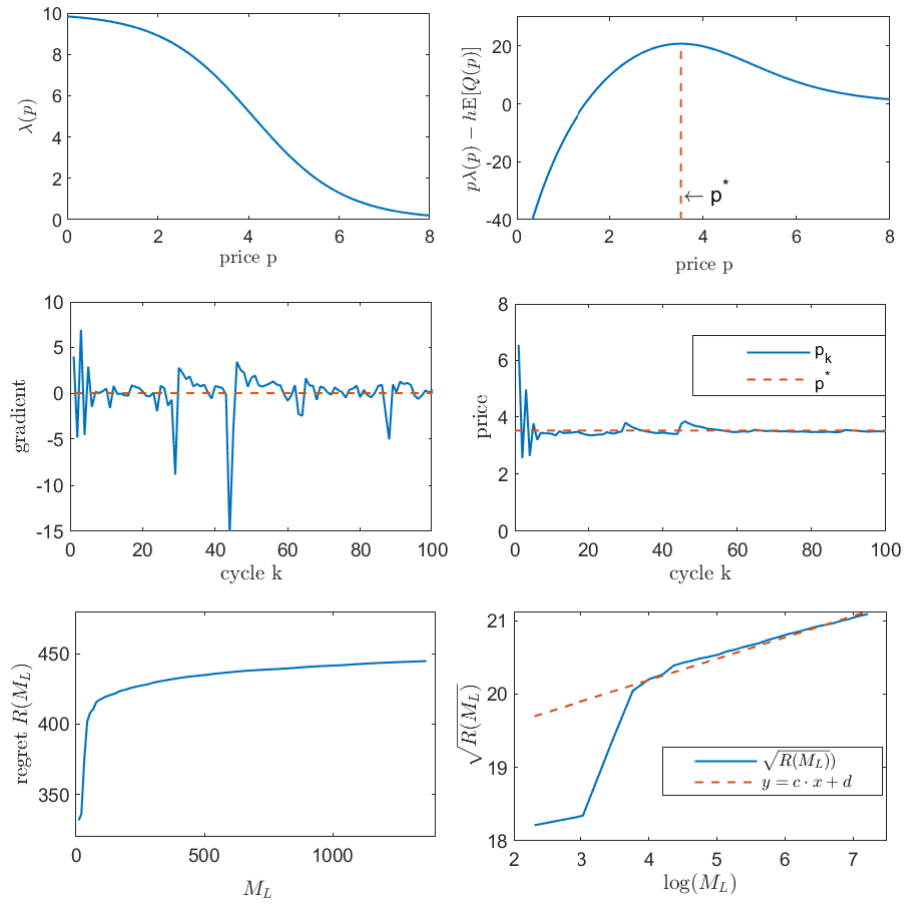}
	\vspace{-0.1in}
	\caption{Online optimal pricing for an $M/M/1$ queue with fixed service rate, with $\mu_0=10,a=4.1,p_0=6.5,p^*=3.531$, $\eta_k=1/k$ and $D_k=10+10\log(k)$: (i) demand function (top left); (ii) revenue function (top right); (iii) sample path of the gradient (middle left); (iv) sample path of the price (middle right); (v) estimated regret (bottom left); (vi) square root of regret versus logarithmic of served customers, with $c=0.24$, $d=19.04$ (bottom right). }\label{fig:FixedMu}
	\vspace{-0.15in}	
\end{figure}
In Figure \ref{fig:FixedMu}, we give the sample paths of the gradient $H_k$ and price $p_k$ as functions of the number of cycles $k$, and the mean regret (estimated by averaging 500 independent sample paths) as a function of the cumulative number of service completions $M_L$. We observe that {although the objective function $f(\mu,p)$ is not convex in $p$}, the pricing decision $p_k$ quickly converges to the optimal value $p^*$, and the regret grows as a logarithmic function of $M_L$. In particular, a simple linear regression for the pair $\left(\sqrt{R(M_L)}, \log(M_L)\right)$ (bottom right panel) verifies our regret bound given in Theorem \ref{thm: regret direct}.

\subsubsection{Online optimal staffing problem with an exogenous arrival rate}
Motivated by conventional service systems where customers are served based on good wills (e.g., hospitals), we next solve an online optimal staffing problem, with the objective of minimizing the combination of the steady-state queue length (or equivalently the delay) and the staffing cost, with the arrival rate (or equivalently, the price $p$) held fixed. Namely, we omit the term $p\lambda(p)$ in \eqref{eq:ObjectiveMM1}. Fixing $\lambda = \lambda_0 = 6.385$, $h_0=1$, and $c_0=0.1$, we obtain the exact optimal service capacity $\mu^*=8.342$ (top right panel of Figure \ref{fig:fixedP}). Also by Algorithm \ref{alg: direct} and Theorem \ref{thm: regret direct}, we set the step size $\eta_k=0.4k^{-1}$ and cycle length $D_k=10+10\log(k)$ with initial service rate $\mu_0=10$.
\begin{figure}[H]
	\includegraphics[width=.93\textwidth]{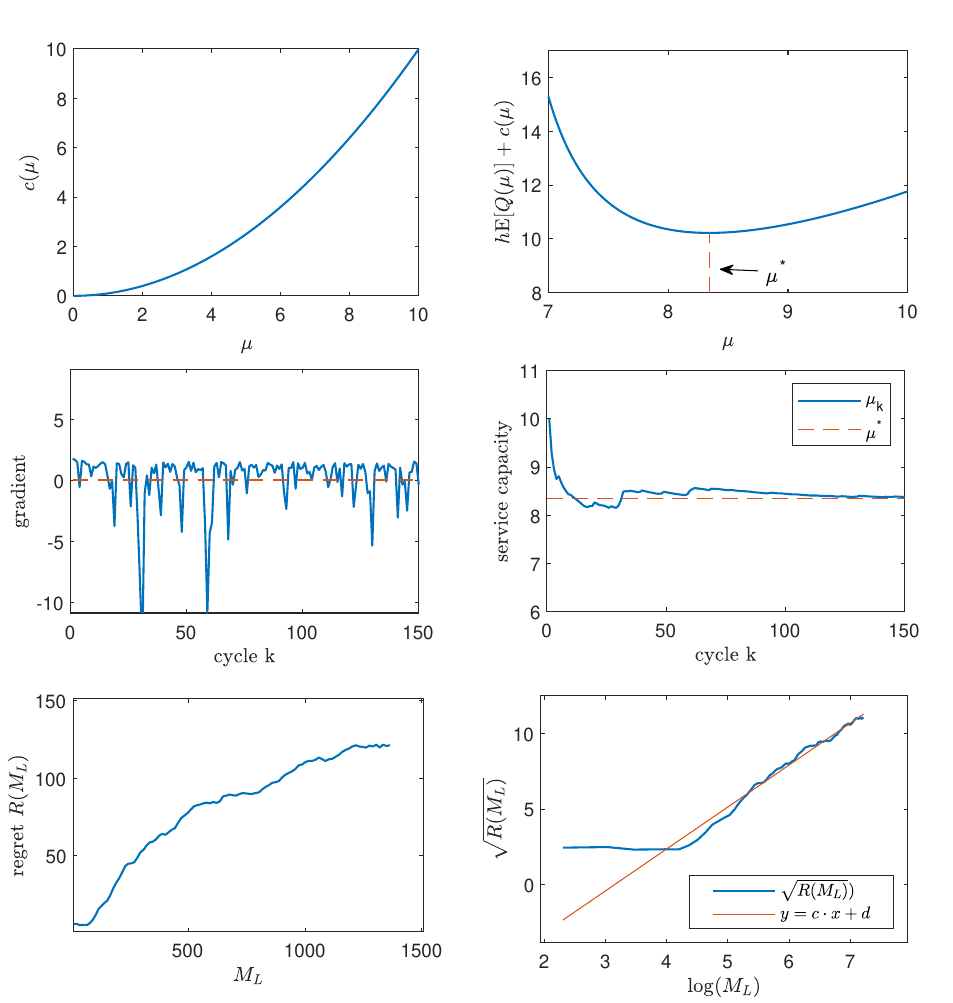}
	\caption{Online optimal staffing for an M/M/1 queue with fixed price with $\lambda_0 =6.385, M =10, \eta_k =0.4k^{-1}$ and $D_k =10+10\log(k)$: (i) staffing cost (top left); (ii) cost function (top right); (iii) sample path of gradient (middle left); (iv) sample path of service capacity (middle right); (v) estimated regret (bottom left); (vi) square root of regret versus logarithmic of served customers, with $c = 2.76$, $d =-8.68$ (bottom right).}
	\label{fig:fixedP}
\end{figure}
In Figure \ref{fig:fixedP}, we again give sample paths of the gradient $H_k$ and service capacity $\mu_k$, and estimation of the regret.
As the number of cycles $k$ increases, our stage-$k$ staffing decision $\mu_k$ quickly converges to $\mu^*$ (bottom right panel) and the regret also grows as a logarithmic function of $M_L$ (bottom left panel). 

\subsection{Joint Pricing and Staffing Problem}\label{sec:TwoDim}


We next consider a joint staffing and pricing problem having the objective function in \eqref{eq:ObjectiveMM1}, with the logistic demand function in \eqref{logisticLmd} and parameters $a=4.1$, $n=10$, $h_0=1$ and $c_0=0.1$. The optimal price $p^*=4.02$ and service rate $\mu^*=7.10$ are given as benchmarks (top right panel in Figure \ref{fig2dMM1}). 
\begin{figure}[H]
	\includegraphics[width=1.01\textwidth]{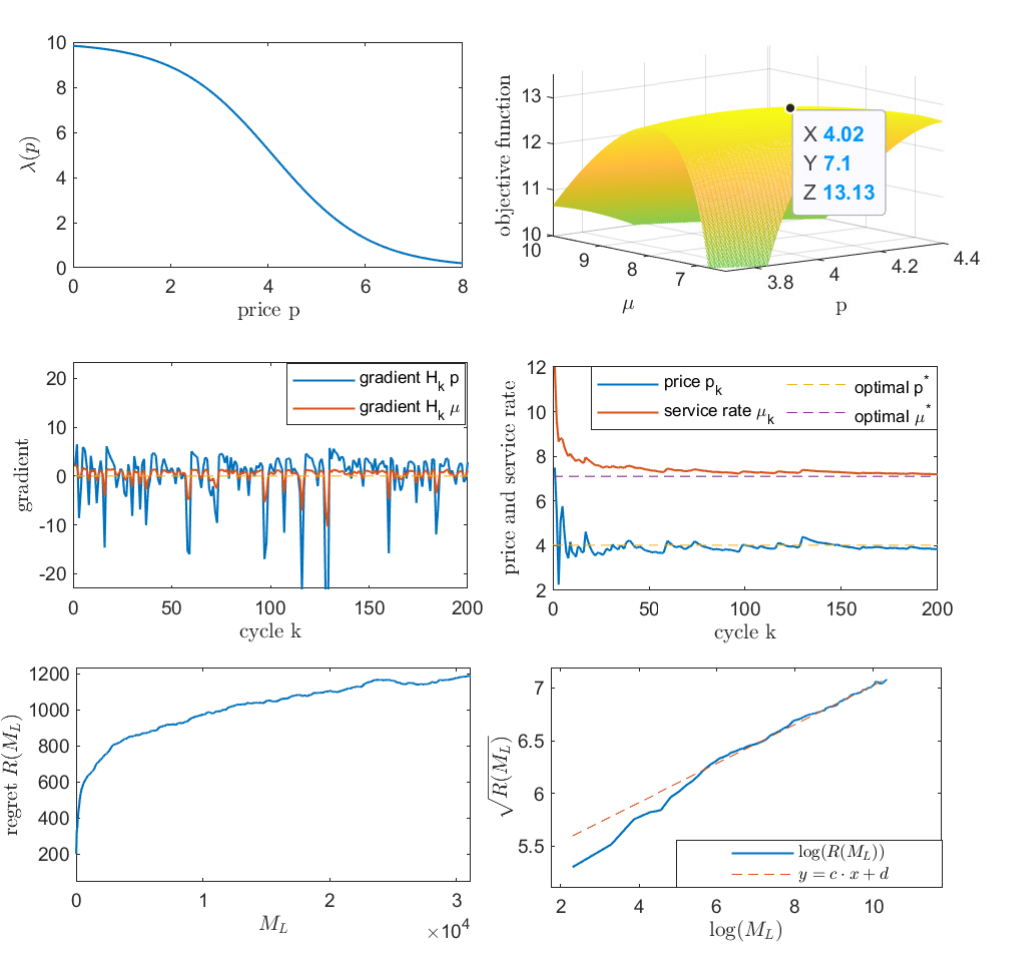}
\caption{Joint pricing and staffing for an $M/M/1$ queue with $p_0=7.5, \mu_0=12, \eta_k = 1/k$ and $D_k = 10+10\log(k)$: (i) demand function (top left); (ii) revenue function (top right); (iii) sample path of gradient (middle left); (iv) sample path of decision parameters (middle right); (v) estimated regret (bottom left); (vi) square root of regret versus logarithmic of served customers, with $c=0.186$, $d=5.17$ (bottom right).}
\label{fig2dMM1}
\end{figure}
In Figure \ref{fig2dMM1}, we show that $\mu_k$ and $p_k$ converge quickly to their corresponding optimal target levels $\mu^*$ and $p^*$ {(although the objective $f(\mu,p)$ is not always convex when $\mu>\lambda(p)$)}. And similar to the one-dimensional cases, the regret grows as a logarithmic function of $M_L$ (bottom left panel). 

\subsection{Comparison to Heavy-Traffic Methods}\label{sec:AsymOpt}
In this subsection, we provide numerical analysis to contrast the performance of GOLiQ to that of the heavy-traffic approach in \cite{LeeWard2014}. 
In \cite{LeeWard2014}, the objective is to find the optimal decisions $p^*$ and $\mu^*$ for the $GI/GI/1$ optimization problem \eqref{eq: maximizing revenue} 
with a linear staffing cost $c(\mu) = c\mu.$ 
Because this problem is not amenable to analytic treatments (due to the complex $GI/GI/1$ queueing dynamics), the authors resort to the heavy-traffic approximation by constructing a sequence of $GI/GI/1$ queues indexed by a scaling factor $n$, where the $n^{\rm th}$ model has an arrival rate $\lambda^n(p)\equiv n\lambda_0(p)$ which grows to infinity as $n$ increases.  
The authors propose an asymptotically optimal solution 
\begin{equation}\label{eq: Amy's result}
(\tilde{p}^{(n)},\tilde{\mu}^{(n)})=\left(\hat{p}^*,n\hat{\mu}^*+\sigma\sqrt{\frac{h_0n}{2c}}\right)
\end{equation}
where $\sigma=\sqrt{\text{Var}(U_i)+\text{Var}(V_i)}$, and $U_i$ and $V_i$ are defined in Assumption \ref{assmpt: light tail}, and $(\hat{p}^*,\hat{\mu}^*)$ solves a deterministic static planning problem:
\begin{equation}\label{eq: Amy SPP}
\min_{p,\mu} f_0(p,\mu)=-p\lambda_0(p)+c\mu.
\end{equation}
We remark that the solution in \cite{LeeWard2014} requires the precise knowledge of the second moments of service and arrival times (e.g., the term $\sigma$ in \eqref{eq: Amy's result}), but such information is not needed in GOLiQ.

\paragraph{\textbf{Experiment settings.}}
We consider an $M/GI/1$ model with a phase-type service-time distribution, and a logit demand $\lambda(p) = n\lambda_0(p)$ in \eqref{logisticLmd} where the base demand rate $\lambda_0(p)$ has $a=4.1$ and the market size $n$ plays the role of the scaling factor. We fix the delay cost $h_0=1$ throughout this experiment. To quantify the regret, we obtain the exact optimal policy using the Pollaczek-Khinchine formula for the queue-length function 
\begin{equation}
\EE[Q_\infty(p,\mu)]=\rho+\frac{\rho^2}{1-\rho}\frac{1+c_s^2}{2},
\label{eq: PK}
\end{equation}
where $c_s^2\equiv Var(U_i)/\EE[U_i]^2$ is the \textit{squared coefficient of variation} (SCV) for the service time.
We next describe the detailed settings for comparing GOLiQ to heavy-traffic solution in  \cite{LeeWard2014}, dubbed LW. In order to benchmark the regret of our GOLiQ to that of LW, we continue to consider a dynamic environment where the number of cycles $k$ increases. In the $k^{\rm th}$ cycle,
\begin{itemize}
\item the LW policy remains fixed at $(\tilde{p}^{(n)},\tilde{\mu}^{(n)})$ as in \eqref{eq: Amy's result} (it does not evolve with $k$);
\item our online learning policy is dynamically updated according to GOLiQ  (Algorithm \ref{alg: direct}).
\end{itemize}
Because the LW policy is an approximation, it will yield a linear regret as $k$ increases. But LW's linear regret should not be too steep when $n$ is large enough. In contrast, although GOLiQ is guaranteed to generate a sublinear regret, it is expected to have a larger reget increment at the earlier ``exploration" stage, because it is learning without the supervision of the fluid or diffusion limits (as in the LW approach).  
Nevertheless, we expect that GOLiQ will eventually outperform the LW method  (exhibiting a lower regret level) when $k$ is sufficiently large. 
We next numerically study how soon GOLiQ surpasses LW and the impact of the following three parameters:
\begin{enumerate}
\item[(i)] staffing cost $c$;
\item[(ii)] service-time SCV $c_s^2$;
\item[(iii)] market size $n$ (i.e., system scale).
\end{enumerate}
We intentionally set the initial decision $(\mu_0,p_0)$ of GOLiQ far from the optimal solution $(\mu^*,p^*)$ in the experiment.

\paragraph{\textbf{Experiment results.}} In Figure \ref{fig:Compare_Amy}, we report results of regret for both GOLiQ and LW. For the three factors $c$, $c_s^2$ and $n$, we change one at a time (with the other two held fixed). 
In Panels (a)-(c), we vary the  staffing cost $c$ from 0.5 to 2. 
In Panels (d)-(f), we vary the service-time SCV $c_s^2$ from 0.1 to 10. Here the cases $c_s^2 = 0.1$, 1, and 10 are achieved by considering Erlang, exponential, and hyperexponential service-time distributions. In Panels (g)-(i), we vary the system scale $n$ from 1 to 25. 
In all of the cases, we use hyper-parameter $\eta_k=5k^{-1}$ and $D_k=10+10\log(k)$. Monte-Carlo estimates of the regret curves are obtained by averaging 100 independent runs.

We can see from Figure \ref{fig:Compare_Amy} that, in all cases, GOLiQ will eventually establish a lower regret level than the LW policy. Varying these three factors clearly has an significant impact on how soon GOLiQ outperforms LW. 
\begin{figure}[H]
\vspace{-0.05in}
\hspace{-0.05in}
\includegraphics[width=1.05\linewidth]{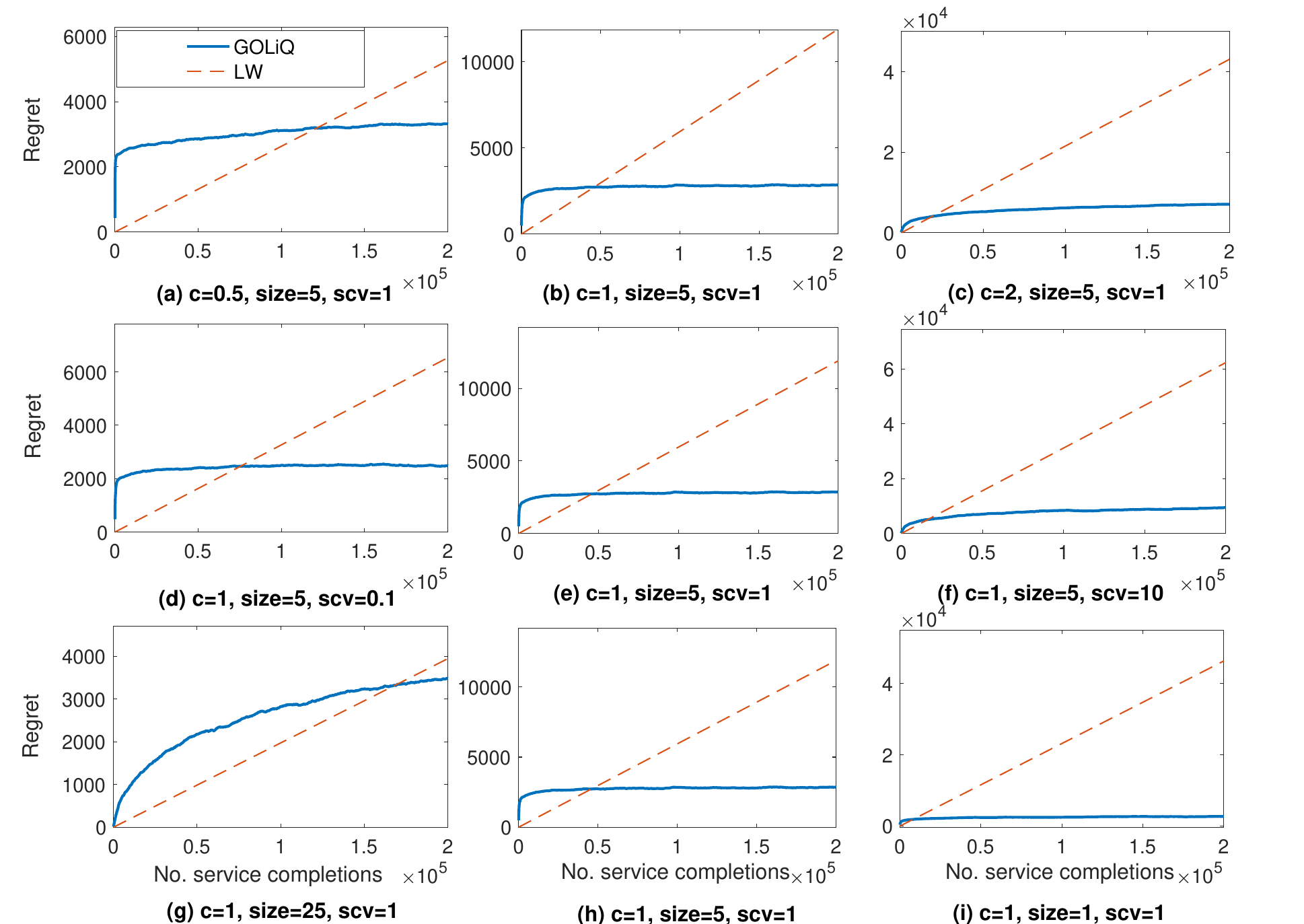}
\caption{Regret comparison to heavy-traffic approximation, with varying (i) staffing cost $c$ (panels (a)--(c)), (ii) service variability $c_s^2$ (panels (d)--(f)), and (iii) market size $n$ (panels (g)--(i)).
	Hyperparameters are $\eta_k=5k^{-1} $and $D_k=10+10\log(k)$ for all instances. All regret are estimated by averaging 500 independent simulation runs.}
\label{fig:Compare_Amy}
\vspace{-0.15in}
\end{figure}
Our findings are summarized below:
\begin{itemize}
\item \textbf{Staffing cost $c$:} Figure \ref{fig:Compare_Amy} shows that GOLiQ intends to outperform LW when $c$ is relatively large. We provide our explanations below. First, a larger staffing cost $c$ will induce a smaller $\mu^*$, which leads to a longer waiting queue. On the other hand, note that the LW solution is primarily based on solving the deterministic static problem \eqref{eq: Amy SPP}; and unlike the stochastic revenue optimization problem \eqref{eq: maximizing revenue}, the objective function of \eqref{eq: Amy SPP} overlooks the queue-length holding cost. This explains why GOLiQ gains  its advantage over LW as $c$ increases.    See Panels (a)-(c) of Figure \ref{fig:Compare_Amy}.
\item \textbf{Service SCV $c_s^2$:} When the service-time SCV is smaller, the LW method intends to work better, because the basic idea of LW stems from solutions of a fluid model (where the service times are assumed deterministic). On the other hand, when $c_s^2$ is larger, the system becomes more variable so that our learning-based algorithm begins to excel (because GOLiQ takes into account real-time information dynamically). See Panels (d)-(f) of Figure \ref{fig:Compare_Amy}.
\item \textbf{Market size $n$:} When $n$ is small, LW loses its advantages because it arises from the large-scale limit of the $GI/GI/1$ queue which requires $n$ to be sufficiently large. While the performance of our GOLiQ is robust to the system scale. See Panels (g)-(i).
\item \textbf{Performance in the long run:} 
GOLiQ is a more effective approach in the long run, because the LW solution remains static and its error grows linearly as time increases.
\end{itemize}
\begin{remark}[Different philosophies: online learning vs. heavy traffic]\label{rem:phi}
We emphasize that online learning and heavy-traffic analysis are two
methodologies developed based on distinct philosophies. First, when the system size is large, heavy-traffic models are able to produce high-fidelity solutions, but they require more prior knowledge of the system as inputs. On the other hand, online learning requires less prior understanding of the system, because the data-driven nature allows it to dynamically evolve and improve (whereas heavy-traffic solutions are static). 
Second, the notions of asymptotic optimality are different.
As an approximate method, heavy-traffic analysis is said to be asymptotically optimal in the sense that as the system size grows large, its solution will become close to the true optimal solution. On the other hand, the solution of the online learning method will converge to the true optimal solution as the server's experience accumulates (by serving more and more customers). 
\end{remark}


\section{Conclusion}\label{sec: conclusions}
In this paper we develop an online learning framework designed for dynamic pricing and staffing in queueing systems. The ingenuity of this approach lies in its online nature, which allows the service provider to continuously obtain improved pricing and staffing policies by interacting with the environment. The environment here is interpreted as everything beyond the service provider's knowledge, which is the composition of the random external demand process and the complex internal queueing dynamics.
The proposed algorithm organizes the time horizon into successive operational cycles, and prescribes an efficient way to update the service provider's policy in each cycle using data collected in previous cycles. Data include the number of customer arrivals, waiting times, and the server's busy times.

A key appeal of the online learning approach is its insensitivity to the scale of the queueing system, as opposed to the heavy-traffic analysis, which requires the system to be in large scale (with the arrival and service rate both approaching infinity). 
Effectiveness of our online learning algorithm is substantiated by (i) theoretical results including the algorithm convergence and regret analysis, and (ii) engineering confirmation via simulation experiments of a variety of representative $GI/GI/1$ queues.
Theoretical analysis of the regret bound in the present paper may shed lights on the design of efficient online learning algorithms (e.g., bounding gradient estimation error and controlling proper learning rate) for more general queueing systems.

{There are several venues for future research. One natural extension would be to develop new regret analyses that do not require the uniform stability condition.
Another interesting and promising direction is to develop an online learning method without assuming the knowledge of the arrival rate function $\lambda(p)$, where the learner (hereby the service provider), during the interactions with the environment, will have to resolve the tension between obtaining an accurate estimation of the demand function and optimizing returns over time.
A third dimension is to extend the methodology to more general model settings (e.g., queues having customer abandonment and multiple servers), which will make the framework more practical for service systems such as call centers and healthcare. {In this regard, results in the present paper may serve as useful foundations; in particular, Theorems \ref{thm: non-stationary error} and \ref{thm: broadiezeevi} will help construct desired regret bounds as long as their associated conditions can be verified. Doing so usually requires two main steps in a new queueing model: 
	(i) proving a new ergodicity (or rate of convergence to stationarity) result that can be used to bound the regret of nonstationarity; 
	(ii) designing a new gradient estimator which is easily  computed from data (here a good gradient estimator should have small bias and variance subject to conditions in Theorem \ref{thm: broadiezeevi}).}} 

\bibliography{references}
\bibliographystyle{chicago}
\bigskip
\begin{center}
{\large\bf SUPPLEMENTARY MATERIAL}
\end{center}

This Supplementary Material provides supplementary materials to the main paper.
In Section \ref{sec: transient analysis proofs}, we give all the
technical proofs omitted from the main paper.
In Section \ref{sec: para_robust}, we test the robustness of GOLiQ with respect to key algorithmic hyperparameters.
In Section \ref{sec: cmp_Tim}, we compare GOLiQ to the online learning method in \cite{Huh2009}.
In Section \ref{sec: AddNum}, we report additional numerical studies.
To facilitate readability, we formally summarize all notations in Table \ref{tab: notation table} including all model parameters and functions, algorithmic hyperparameters, and constants in the regret analysis.



\section{Proofs}\label{sec: transient analysis proofs}

\subsection{Proof of Lemma \ref{lmm: uniform bound}}\label{subsect: uniform bound}
Let $Q_n^k$ be the queue length when customer $n-1$ in cycle $k$ leaves the system. Then $Q_k=Q_{D_{k-1}}^{k-1}+1$. The proof follows a stochastic ordering argument for $GI/GI/1$ models. Let $\hat{W}_n^k$, $\hat{X}_n^k$ and $\hat{Q}_n^k$ be the waiting times, observed busy periods, and queue length process in a $GI/GI/1$ queue with stationary control parameter $\mu_k\equiv \underline{\mu}$ and $p_k\equiv \underline{p}$, and with steady-state initial state, i.e.,  $\hat{W}_0^1\stackrel{d}{=} W_\infty(\underline{\mu},\underline{p})$, $\hat{X}_0^1\stackrel{d}{=} X_\infty(\underline{\mu},\underline{p})$ and $\hat{Q}_0^1\stackrel{d}{=} Q_\infty(\underline{\mu},\underline{p})$. Let's call this system the dominating system. Then, for all $k$,
$$\frac{U_n^k}{\lambda_{n}^k}\geq \frac{U_n^k}{\lambda(\underline{p})}, \text{ for }n=1,2,...,Q_k, \text{ and }\frac{U_n^k}{\lambda_k}\geq \frac{U_n^k}{\lambda(\underline{p})}, \text{ for }n=Q_k+1,2,...,D_k,$$
i.e., the arrival process in the dominating queue is the \textit{upper envelope process} (UEP) for all possible arrival processes corresponding to any control sequence $(\mu_k,p_k)$. Similarly, the service process in the dominating queue is the \textit{lower envelope process} (LEP) for all possible service processes corresponding to any control sequence. As a consequence, since $W_0^1=0$ and $Q_0^1=0$, $$W_n^k\leq_{st}\hat{W}_n^k, ~X_n^k\leq_{st}\frac{\lambda(\underline{p})}{\lambda(\bar{p})}\cdot\hat{X}_n^k, ~Q_n^k\leq_{st}\hat{Q}_n^k.$$
{Under Assumption \ref{assmpt: light tail}, the moment generating function of the random variable  $V_n/\underline{\mu}-U_n/\lambda(\underline{p})$ exists around the origin. Following \cite{BlanchetChen_2015}, under Assumption \ref{assmpt: uniform}, this condition can further imply that there exists a constant $\bar{\eta}>0$ such that
	$\EE[\exp\left(\bar{\eta}(V_n/\underline{\mu}-U_n/\lambda(\underline{p}))\right)]=1.$ (See the Remark on p.3222 in \cite{BlanchetChen_2015})}
Then, following Theorem 1 of \cite{abate1995exponential}, there exists a constant $\alpha>0$ such that
$\mathbb{P}(\hat{W}_n^k>x)\leq \alpha\exp(-\bar{\eta}x), \text{ for all }x>0$.
As a consequence, $\EE[\exp(\eta \hat{W}_n^k)]$ is finite for $\eta<\bar{\eta}$, and so are $\EE[(\hat{W}_n^k)^m]$ for all $m\geq 1$.
Given that the moments of waiting times are finite, we can conclude that  $\EE[(\hat{Q}_n^k)^m]$ and $\EE[\exp(\eta \hat{Q}_n^k)]$ are finite for all $m\geq 1$, applying Theorem 10.4.3 in 
\cite{AsmussenBook}. Finally, the moments of the observed busy period $\EE[(\hat{X}_n^k)^m]$ are finite following Proposition 4.2 in \cite{nakayama2004finite}.
Therefore, we choose
$$M=\max_{1\leq m\leq 4}\left\{\EE[(\hat{W}_n^k)^m], \ \frac{\lambda(\underline{p})^m}{\lambda(\bar{p})^m}  \EE[(\hat{X}_n^k)^m], \ \EE[(\hat{Q}_n^k+1)^m], \ \EE[\exp(\eta \hat{W}_n^k)], \ \EE[\exp(\eta( Q_n^k+1))]\right\},$$
and this closes our proof.\hfill $\Box$

\subsection{Proof of Lemma \ref{lmm: thm1 1}}
For $i\in\{1,2\}$, define stopping times $\Gamma_i = \min\{n: W^i_n = 0\}$. 
For  a fixed pair of inter-arrival and service time sequences, the consequent waiting time sequence $W_k$ in a single-server queue is monotone in its initial state $W_0$.  Without loss of generality, assume $W^1_0\geq W^2_0$. Then, $W^1_n\geq W^2_n$ for all $n\geq 1$ and therefore, $W^1_{\Gamma_1} = W^2_{\Gamma_1}=0$.
As the two queues are coupled with the same arrival and service time sequences, we will have $W^1_n = W^2_n$ for all $n\geq \Gamma_1$. Therefore, we can conclude $W^1_n = W^2_n$ for all $n\geq \max(\Gamma_1,\Gamma_2)$. For $n\leq \max(\Gamma_1,\Gamma_2)$, we have $|W^1_n-W^2_n|\leq |W^1_0-W^2_0|$ following \cite{KellaRamasubramanian_2012}.

For simplicity of notation, we write $\lambda=\lambda(p)$. 
For $i\in\{1,2\}$,  define a random walk $R^i_{n+1} = R^i_n + S_n- \tau_n$ with $R^i_0= W^i_0$. (Recall that $S_n$ and $\tau_n$ are the sequences of service and inter-arrival times.) By Lindley recursion,	$\Gamma_i = \min\{n: R^i_n\leq 0\}$.
Then, for any  $n \geq 1$,
\begin{align*}
	\PP(\Gamma_i \leq n) &\geq \PP\left(\sum_{k=1}^n (S_k-\tau_k)<-W_0^i\right)\\
	&\geq \PP\left(\lambda\sum_{k=1}^n \tau_k \geq n(1-a), \mu\sum_{k=1}^n S_k \leq n(1+a) -\mu W^i_0\right),
\end{align*}
where the second inequality holds as $(1-a)/\lambda>(1+a)/\mu$ given that $0<a<(\underline{\mu}-\lambda(\underline{p}))/(\underline{\mu}+\lambda(\underline{p}))$ and that $\lambda/\mu\leq \lambda(\underline{p})/\underline{\mu}$.
Recall that $\tau_k=U_k/\lambda$ and $S_k=V_k/\mu$. Therefore,
\begin{align*}
	\PP(\Gamma_i > n) \leq \PP\left( \sum_{k=1}^n U_k <  n(1-a)\right) +\PP\left(\sum_{k=1}^n V_k > n(1+a)-\mu W^i_0\right).
\end{align*}
Following Chebyshev's Inequality, we have
\begin{equation*}
	\begin{aligned}
		\PP\left(\sum_{k=1}^n V_k >n(1+a)-\mu W^i_0\right)&\leq\frac{\EE[\exp(\theta \sum_{k=1}^n V_k)]}{\exp(n\theta(1+a)-\mu\theta W^i_0)} = \exp(n(\phi_V(\theta)-(1+a)\theta))\exp(\mu\theta W^i_0)\\&\leq \exp(-n\gamma)\exp(\mu\theta W^i_0),
	\end{aligned}
\end{equation*}
where the last inequality follows from Assumption \ref{assmpt: light tail}. On the other hand, let $Q$ be an exponentially tilted probability measure with respect to $U$, such that the likelihood ratio $\frac{dQ}{dP}(U) = \exp(-\theta U -\phi_U(-\theta))$. Then,
\begin{align*}
	&\PP\left(\sum_{k=1}^n U_k < n(1-a)\right)=\EE^Q\left[\exp\left(\theta\sum_{k=1}^n U_k +n\phi_U(-\theta)\right){\bf 1}_{\left\{\sum_{k=1}^n U_k<n(1-a)\right\}}\right] \\
	\leq & \exp(n(1-a)\theta + n\phi_U(-\theta))=\exp(n((1-a)\theta +\phi_U(-\theta)))\leq \exp(-n\gamma).
\end{align*}
In summary, we have $\PP(\Gamma_i>n)\leq \exp(-n\gamma)\left(1+\exp(\mu\theta W^i_0)\right)$, $i=1,2$.
So, we can conclude
\begin{equation*}
	\begin{aligned}
		\EE[|W^1_n-W^2_n|^m]&\leq \PP(\max(\Gamma_1, \Gamma_2)>n)|W^1_0-W^2_0|^m\\
		&\leq e^{-\gamma n}\left(2+e^{\mu\theta W^1_0}+e^{\mu\theta W^2_0}\right)|W^1_0-W^2_0|^m.\qquad \Box
	\end{aligned}
\end{equation*}
\subsection{Proof of Lemma \ref{lmm: Lipschitz}}
Define two auxiliary random walks: $$Y_n=W_0+\sum_{i=1}^n \left(\frac{V_i}{\mu_i}-\frac{U_i}{\lambda_i}\right), \tilde{Y}_n=\tilde{W}_0+\sum_{i=1}^n \left(\frac{V_i}{\tilde{\mu}_i}-\frac{U_i}{\tilde{\lambda}_i}\right).$$
Then, for any $n\geq 1$, we could express $W_n$ and $\tilde{W}_n$  as
$$W_n = Y_n -\min_{1\leq m\leq n} Y_m \wedge 0,\quad \tilde{W}_n = \tilde{Y}_n -\min_{1\leq m\leq n} \tilde{Y}_m \wedge 0.$$
Let $\tau=\argmin_{1\leq m\leq n} Y_m$ and  $\tilde{\tau}=\argmin_{1\leq m\leq n} \tilde{Y}_m$. Note that following the above notation, for each $n$, $W_n$ is the waiting time of customer $n$ and as a consequence, $\frac{U_n}{\lambda_n}$ should be understood as the inter-arrival time between customers $n-1$ and $n$, and $\frac{V_n}{\mu_n}$ as the service time of customer $n-1$.

\textbf{Case 1:} If $Y_\tau\leq 0$ and $\tilde{Y}_{\tilde{\tau}}\leq 0$, i.e., both $W_t$ and $\tilde{W}_t$ hit zero before $n$, we have
\begin{align*}
	Y_n-Y_{\tilde{\tau}} -(\tilde{Y}_n-\tilde{Y}_{\tilde{\tau}})\leq W_n-\tilde{W}_n = Y_n-Y_\tau -(\tilde{Y}_n-\tilde{Y}_{\tilde{\tau}})\leq Y_n-Y_\tau -(\tilde{Y}_n-\tilde{Y}_{\tau}).
\end{align*}
So, in this case
$$|W_n-\tilde{W}_n|\leq \sum_{i=\tau\wedge\tilde{\tau}+1}^{n}\left|\frac{1}{\mu_i}-\frac{1}{\tilde{\mu}_i}\right|V_i+\sum_{i=\tau\wedge\tilde{\tau}+1}^{n}\left|\frac{1}{\lambda_i}-\frac{1}{\tilde{\lambda}_i}\right|U_i.$$
Recall that $X_n$ (and $\tilde{X}_n$) is the age of the server's busy time observed by customer $n$ upon arrival. By definition, $W_{\tau}=0$ and therefore,
$$X_n = \sum_{i=\tau+1}^n\frac{U_i}{\lambda_i},\quad X_n+W_n= \sum_{i=\tau+1}^n\frac{V_i}{\mu_i}.$$
The second equation holds as the server has just served $n-\tau$ customers (indexed from $\tau$ to $n-1$) in the current busy cycle when customer $n$ enters service. Then,
$$\sum_{i=\tau+1}^{n}\left|\frac{1}{\mu_i}-\frac{1}{\tilde{\mu}_i}\right|V_i+\sum_{i=\tau+1}^{n}\left|\frac{1}{\lambda_i}-\frac{1}{\tilde{\lambda}_i}\right|U_i\leq \frac{c_\mu}{\underline{\mu}}(X_n+W_n) + \frac{c_\lambda}{\underline{\lambda}}X_n.$$
Following a similar argument, we have
$$\sum_{i=\tilde{\tau}+1}^{n}\left|\frac{1}{\mu_i}-\frac{1}{\tilde{\mu}_i}\right|V_i+\sum_{i=\tilde{\tau}+1}^{n}\left|\frac{1}{\lambda_i}-\frac{1}{\tilde{\lambda}_i}\right|U_i\leq \frac{c_\mu}{\underline{\mu}}(\tilde{X}_n+\tilde{W}_n) + \frac{c_\lambda}{\underline{\lambda}}\tilde{X}_n.$$
Therefore, in this case, we have
$$|W_n-\tilde{W}_n|\leq \left(\frac{c_\mu}{\underline{\mu}}+\frac{c_\lambda}{\underline{\lambda}}\right)\max(X_n,\tilde{X}_n)+\frac{c_\mu}{\underline{\mu}}\max(W_n,\tilde{W}_n).$$

\textbf{Case 2:}
If $Y_\tau> 0$ or $\tilde{Y}_{\tilde{\tau}}> 0$, we can inductively derive that
$$|W_n-\tilde{W}_n|\leq |W_0-\tilde{W}_0|+\sum_{i=1}^n\left|\frac{1}{\mu_i}-\frac{1}{\tilde{\mu}_i}\right|V_i+\sum_{i=1}^n\left|\frac{1}{\lambda_i}-\frac{1}{\tilde{\lambda}_i}\right|U_i.$$
{
In detail, it suffices to show that, for all $1\leq m\leq n$, 
\begin{equation}|W_m-\tilde{W}_m|\leq |W_{m-1}-\tilde{W}_{m-1}|+\left|\frac{1}{\mu_m}-\frac{1}{\tilde{\mu}_m}\right|V_m+\left|\frac{1}{\lambda_m}-\frac{1}{\tilde{\lambda}_m}\right|U_m.\label{eq: lmm3 case2}
\end{equation}
Without loss of generality, we assume $Y_\tau> 0$. By definition, $Y_\tau=\min_{1\leq l\leq n} Y_l$ and hence $W_l=Y_l>0$ for all $1\leq l\leq n$. Then,
\begin{equation*}
	|W_m-\tilde{W}_m|=\left|  W_{m-1}-\frac{U_m}{\lambda_m}+\frac{V_m}{\mu_m}- \left(  \tilde{W}_{m-1} -\frac{U_m}{\tilde{\lambda}_m}+\frac{V_m}{\tilde{\mu}_m} \right)^+     \right|.
\end{equation*}
If $\tilde{W}_m>0$, we have 
\begin{align*}
	|W_m-\tilde{W}_m|&=\left|  W_{m-1}-\frac{U_m}{\lambda_m}+\frac{V_m}{\mu_m}- \left(  \tilde{W}_{m-1} -\frac{U_m}{\tilde{\lambda}_m}+\frac{V_m}{\tilde{\mu}_m} \right)     \right|\\
	&\leq |W_{m-1}-\tilde{W}_{m-1}|+\left|\frac{1}{\mu_m}-\frac{1}{\tilde{\mu}_m}\right|V_m+\left|\frac{1}{\lambda_m}-\frac{1}{\tilde{\lambda}_m}\right|U_m.
\end{align*}
On the other hand, if $\tilde{W}_{m}=0$, we have $\tilde{W}_{m-1}-\frac{U_m}{\tilde{\lambda}_m}+\frac{V_m}{\tilde{\mu}_m}\leq 0.$ So,
\begin{align*}
	|W_m-\tilde{W}_m|&=W_m-0\leq W_m-\left(  \tilde{W}_{m-1} -\frac{U_m}{\tilde{\lambda}_m}+\frac{V_m}{\tilde{\mu}_m} \right)\\
	&=\left|  W_{m-1}-\frac{U_m}{\lambda_m}+\frac{V_m}{\mu_m}- \left(  \tilde{W}_{m-1} -\frac{U_m}{\tilde{\lambda}_m}+\frac{V_m}{\tilde{\mu}_m} \right)     \right|\\
	&\leq|W_{m-1}-\tilde{W}_{m-1}|+\left|\frac{1}{\mu_m}-\frac{1}{\tilde{\mu}_m}\right|V_m+\left|\frac{1}{\lambda_m}-\frac{1}{\tilde{\lambda}_m}\right|U_m.
\end{align*}
This closes the proof of \eqref{eq: lmm3 case2}.
}

As a result of \eqref{eq: lmm3 case2}, if $Y_\tau>0$, we can conclude the system (associated with $(\mu_n,\lambda_n)$) was kept busy from time 0 until customer $n$ enters service. As a consequence, as $X_0\geq 0$, we have
$$X_n \geq \sum_{i=1}^n\frac{U_i}{\lambda_i},\quad X_n+W_n\geq \sum_{i=1}^n\frac{V_i}{\mu_i}.$$
Therefore, 
$$\sum_{i=1}^{n}\left|\frac{1}{\mu_i}-\frac{1}{\tilde{\mu}_i}\right|V_i+\sum_{i=1}^{n}\left|\frac{1}{\lambda_i}-\frac{1}{\tilde{\lambda}_i}\right|U_i\leq \frac{c_\mu}{\underline{\mu}}(\max(X_n,\tilde{X}_n)+\max({W}_n,\tilde{W}_n)) + \frac{c_\lambda}{\underline{\lambda}}\max(X_n,\tilde{X}_n),$$
and hence we can also conclude
$$|W_n-\tilde{W}_n|\leq |W_0-\tilde{W}_0|+\left(\frac{c_\mu}{\underline{\mu}}+\frac{c_\lambda}{\underline{\lambda}}\right)\max(X_n,\tilde{X}_n)+\frac{c_\mu}{\underline{\mu}}\max(W_n,\tilde{W}_n).$$
\hfill$\Box{}$

\subsection{Proof of Lemma \ref{lmm: steady state continuity}}
By the inequality that $(a+b)^m\leq 2^{m-1}(a^m+b^m)$ for $m\geq 1$, we have
\begin{align*}
&\EE[|W_\infty(\mu_1,p_1)-W_\infty(\mu_2,p_2)|^m]\\
&\leq 2^{m-1}\left(\EE[|W_\infty(\mu_1,p_1)-W_\infty(\mu_2,p_1)|^m+|W_\infty(\mu_2,p_1)-W_\infty(\mu_2,p_2)|^m]\right).
\end{align*}
It  suffices to prove that there exist two constant $B_1, B_2 >0$ such that for $1\leq m\leq 4$,
\begin{equation*}
\begin{aligned}
	\EE[|W_\infty(\mu_1,p_1)-W_\infty(\mu_2,p_1)|^m]&\leq B_1|\mu_1-\mu_2|^m, \\
	\EE[|W_\infty(\mu_2,p_1)-W_\infty(\mu_2,p_2)|^m]&\leq B_2|p_1-p_2|^m.
\end{aligned}
\end{equation*}
Without loss of generality, assume $\mu_{1}<\mu_2$. We now construct two
stationary sequences
$\{(W_n^{\mu_i}:n\leq 0), i=1,2\}$
that are coupled ``from the past". Let $V_j$ and $U_j$ be two i.i.d sequences corresponding to the service and inter-arrival times. For each $i$, we define a random walk:
$$Y^{\mu_i}_0=0, ~Y^{\mu_i}_n = \sum_{j=1}^{n} \left(\frac{V_j}{\mu_i}-\frac{U_j}{\lambda(p_1)}\right),~ \forall n\geq 1.$$
It is clear that  $Y^{\mu_i}_n$ is a random walk with negative drift for $i=1, 2$. Define
$$W_{-n}^{\mu_i} = \max_{j\geq n}Y^{\mu_i}_j - Y^{\mu_i}_n, n\geq 0.$$
It is known in literature (see, for example, \cite{BlanchetChen_2015}) that $W_{-n}^{\mu_i}$ is a stationary waiting time process of a $GI/GI/1$ queue, starting from $-\infty$, with parameter $(\mu_i,p_1)$. In particular, the dynamics of $W_{-n}^{\mu_i}$ satisfies that
$$W_{-n+1}^{\mu_i} = \left(W_{-n}^{\mu_i}+\frac{V_{n}}{\mu_i}-\frac{U_n}{\lambda(p_1)}\right)^+,\text{ for }n\geq 1,$$
with $V_n/\mu_i$ being the service time of customer $-n$ and $U_n/\lambda(p_1)$ being the inter-arrival time between customer $-n$ and $-n+1$.
For a fixed sequence of $(V_n, U_n)$, we have
$$W_0^{\mu_1}=\max_{j\geq 0}Y_j^{\mu_1},\quad\text{and }W_0^{\mu_2}=\max_{j\geq 0}Y_j^{\mu_2}.$$
As $Y^{\mu_1}_j\geq Y^{\mu_2}_j$, we have $W_0^{\mu_1}\geq W_0^{\mu_2}$. Besides, let $\tau=\arg\max_{j\geq 0}Y^{\mu_1}_j$, we have
$$W_0^{\mu_1}-W_0^{\mu_2}=\max_{j\geq 0}Y_j^{\mu_1}-\max_{j\geq 0}Y_j^{\mu_2}=Y_\tau^{\mu_1}-\max_{j\geq 0}Y_j^{\mu_2}\leq Y_\tau^{\mu_1}-Y_\tau^{\mu_2}.$$
As a consequence, we have
$$|W_0^{\mu_1}-W_0^{\mu_2}|\leq \sum_{n=1}^{\tau}\left(\frac{V_n}{\mu_1}-\frac{V_n}{\mu_2}\right)\leq \frac{\mu_2-\mu_1}{\mu_1}\sum_{n=1}^{\tau}\frac{V_n}{\mu_1},\quad \text{ with }\tau = \inf\{n: W_{-n}=0\}.$$
Note that $V_n/\mu_1$ is the service time of customer $-n$ in the system with parameter $(p_1,\mu_1)$. By the definition of $\tau$, customer $-\tau$ enters service immediately upon the arrival and the queue remains busy by  arrival of customer $0$. Therefore, the summation of service times on the right hand side equals to the time between the arrival of customer $-\tau$ and the departure of customer $-1$, which equals to the observed busy period at the arrival of customer $0$ plus its waiting time, i.e.,
$$|W_0^{\mu_1}-W_0^{\mu_2}|\leq \frac{\mu_2-\mu_1}{\mu_1}\sum_{n=1}^{\tau}\frac{V_n}{\mu_1}= \frac{\mu_2-\mu_1}{\mu_1}(X_0^{\mu_1}+W_0^{\mu_1}).$$
Therefore, for each $n$,
$$\EE[|W_0^{\mu_1}-W_0^{\mu_2}|^m]\leq  \frac{(\mu_2-\mu_1)^m}{\mu_1^m}\EE[(X_0^{\mu_1}+W_0^{\mu_1})^m]\leq \frac{(\mu_2-\mu_1)^m}{\underline{\mu}^m}\EE[(X_0^{\mu_1}+W_0^{\mu_1})^m].$$
Following Lemma \ref{lmm: uniform bound}, $\EE[(X_0^{\mu_1}+W_0^{\mu_1})^m]\leq 2^m M$.
Let $B_1=\max_{1\leq m\leq 4}2^mM/\underline{\mu}^m$ and we conclude, for $1\leq m\leq 4$,
$$ \EE[|W_0^{\mu_1}-W_0^{\mu_2}|^m]\leq B_1|\mu_1-\mu_2|^m.$$

The bound for $\EE[|W_\infty(\mu_2,p_1)-W_\infty(\mu_2,p_2)|^m]$ follows a similar argument and therefore we only provide a sketch of the proof. Without loss of generality, we assume $p_1<p_2$ and consider two stationary waiting time process $\{(W_n^{p_i}:n\leq 0), \lambda_i=\lambda(p_i), i=1,2\}$ that are coupled from past with the same sequence $(V_n,U_n)$ in a similar way as we introduced previously. Then, we have $|W_0^{p_1}-W_0^{p_2}|\leq (\lambda_1-\lambda_2)X_0^{p_1}/\lambda_2$,
and therefore,
$$ \EE[|W_0^{p_1}-W_0^{p_2}|^m]\leq B_2|p_1-p_2|^m,\text{ with }B_2 =\max_{1\leq m\leq 4, \underline{p}\leq p\leq \bar{p}}(M|\lambda'(p)|^m/\lambda(\bar{p})^m).$$    As a consequence, we can take
\begin{equation}\label{eq: B}
B = 8\cdot\max_{1\leq m\leq 4}(2^m M/\underline{\mu}^m)\vee \max_{1\leq m\leq 4, \underline{p}\leq p\leq \bar{p}}(M|\lambda'(p)|^m/\lambda(\bar{p})^m).
\end{equation}

\subsection{Full Proof of Theorem \ref{thm: non-stationary error}}\label{subsect: non stationary proof}
We first give the proofs of Corollaries \ref{coro: convergence to stationarity}--\ref{coro: I3}.

\begin{proof}[Proof of Corollary \ref{coro: convergence to stationarity}]
For any $n\geq   d_k$,
$$\EE[|W_n^k-\bar{W}_n^k|] = \EE[|W_n^k-\bar{W}_n^k|1(Q_k<d_k)]+\EE[|W_n^k-\bar{W}_n^k|1(Q_k\geq d_k)].$$
Given that $Q_k<d_k$, by definition, $W_n^k$ is synchronously coupled with $\bar{W}_n^k$ for $n\geq d_k+1$. Note that given $Q_k<d_k$, $U_n^k$ and $V_n^k$ are independent of $Q_k$ for $n\geq d_k+1$. As a consequence, by Lemma \ref{lmm: thm1 1}, the conditional expectation
\begin{equation*}
	\begin{aligned}
		\EE\left[|W_n^k-\bar{W}_n^k|~\Big|~Q_k<d_k, W_{d_k}^k, \bar{W}_{d_k}^k\right]\leq e^{-\gamma(n-d_k)}(2+e^{\bar{\mu}\theta W_{d_k}^k}+e^{\bar{\mu}\theta \bar{W}_{d_k}^k})\left|W_{d_k}^k-\bar{W}_{d_k}^k\right|.
	\end{aligned}
\end{equation*}
Therefore,
\begin{equation*}
	\begin{aligned}
		\EE[|W_n^k-\bar{W}_n^k|1(Q_k<d_k)]&\leq e^{-\gamma(n-d_k)}\EE\left[(2+e^{\bar{\mu}\theta W_{d_k}^k}+e^{\bar{\mu}\theta \bar{W}_{d_k}^k})\left|W_{d_k}^k-\bar{W}_{d_k}^k\right|1(Q_k<d_k)\right]\\
		&\leq e^{-\gamma(n-d_k)}\EE\left[(2+e^{\bar{\mu}\theta W_{d_k}^k}+e^{\bar{\mu}\theta \bar{W}_{d_k}^k})\left|W_{d_k}^k-\bar{W}_{d_k}^k\right|\right]\\
		&\leq e^{-\gamma(n-d_k)}\left(2+\EE\left[\left(e^{\bar{\mu}\theta W_{d_k}^k}+e^{\bar{\mu}\theta \bar{W}_{d_k}^k}\right)^2\right]^{1/2}\right)\EE\left[\left|W_{d_k}^k-\bar{W}_{d_k}^k\right|^{2}\right]^{1/2}.
	\end{aligned}
\end{equation*}
By Lemma \ref{lmm: uniform bound} and Assumption \ref{assmpt: light tail}, we have
\begin{align*}
	\EE\left[\left(e^{\bar{\mu}\theta W_{d_k}^k}+e^{\bar{\mu}\theta \bar{W}_{d_k}^k}\right)^2\right]&\leq 2\left(\EE[e^{2\bar{\mu}\theta W_{d_k}^k}]+\EE[e^{2\bar{\mu}\theta \bar{W}_{d_k}^k}]\right)\leq 4M,\\
	\EE\left[|W_{d_k}^k-\bar{W}_{d_k}^k|^2\right]&\leq 2\left(\EE[(W_{d_k}^k)^2]+\EE[(\bar{W}_{d_k}^k)^2]\right)\leq 4M.
\end{align*}
As a consequence, we have
$$\EE[|W_n^k-\bar{W}_n^k|1(Q_k<d_k)]\leq  e^{-\gamma(n-d_k)}A, \text{with }A=4\sqrt{M}+4M.$$
On the other hand,
$$\EE[|W_n^k-\bar{W}_n^k|1(Q_k\geq d_k)]\leq \EE\left[|W_n^k-\bar{W}_n^k|^{2}\right]^{1/2}\PP(Q_k\geq d_k)^{1/2}.$$
Again, by Lemma \ref{lmm: uniform bound},
$\EE\left[|W_n^k-\bar{W}_n^k|^{2}\right]\leq 4M.$
As $d_k=\lceil 4\log(k)/\min(\gamma,\eta)\rceil$, 
$$\PP(Q_k\geq d_k)\leq e^{-\eta d_k}\EE\left[e^{\eta Q_k}\right]\leq k^{-4}M.$$
In summary, we have, for $n\geq d_k+1$,
$$\EE[|W_n^k-\bar{W}_n^k|] \leq e^{-\gamma(n-d_k)}A + 2Mk^{-2}.$$
As a direct consequence,
\begin{align*}
	|I_1|&=\left|\EE\left[\sum_{n=\tilde{d}_k+1}^{D_k}W_n^k-w(\mu_k,p_k)\right]\right|
	\leq \sum_{n=\tilde{d}_k+1}^{ D_k}\EE[|W_n^k-\bar{W}_n^k|]\\
	&\leq \sum_{n=\tilde{d}_k+1}^\infty e^{-\gamma(n-d_k)}A + 2Mk^{-2}D_k\leq \frac{A}{1-e^{-\gamma}}k^{-1} + 2MK_2k^{-\alpha}=O(k^{-\alpha}).
\end{align*}
\end{proof}

\begin{proof}[Proof of Corollary \ref{coro: error of leftovers}]
Recall that by \eqref{eq:RecDelay}, for each cycle $k$,	\begin{equation*}
	W_n^k =
	\begin{cases}
		\left(W_{n-1}^k + \frac{V^k_n}{\mu_k} - \frac{U^k_n}{\lambda^k_n}\right)^+ & \text{ for }1\leq n\leq Q_k{\wedge D_k};\\
		\left(W_{n-1}^k + \frac{V^k_n}{\mu_k} - \frac{U^k_n}{\lambda_{k}}\right)^+ & \text{ for } (Q_k+1){\wedge(D_k+1)}\leq n\leq D_k.
	\end{cases}, \quad W_0^k = W_{D_{k-1}}^{k-1}.
\end{equation*}
Define
\begin{equation*}
	\tilde{W}_n^k =
	\begin{cases}
		\left(\tilde{W}_{n-1}^k + \frac{V^k_n}{\mu_k} - \frac{U^k_n}{\lambda_{k-1}}\right)^+ & \text{ for }1\leq n\leq Q_k{\wedge D_k};\\
		\left(\tilde{W}_{n-1}^k + \frac{V^k_n}{\mu_k} - \frac{U^k_n}{\lambda_{k}}\right)^+ & \text{ for } (Q_k+1){\wedge(D_k+1)}\leq n\leq D_k.
	\end{cases}, \quad \tilde{W}_0^k = W_{D_{k-1}}^{k-1}.
\end{equation*}
Then, in the case $Q_{k-1}<D_{k-1}$, we have $W_n^k=\tilde{W}_n^k$ for all $1\leq n\leq D_k$. As a consequence, we have
$$|W_n^k-\bar{W}_{D_{k-1}+n}^{k-1}|\leq |\tilde{W}_n^k-\bar{W}_{D_{k-1}+n}^{k-1}|+ |W_n^k-\bar{W}_{D_{k-1}+n}^{k-1}|\cdot 1(Q_{k-1}\geq D_{k-1}).$$
For the second term, by Lemma \ref{lmm: uniform bound}, we have, for $k\geq 2$,
\begin{align*}
	\EE[|W_n^k-\bar{W}_{D_{k-1}+n}^{k-1}|\cdot 1(Q_{k-1}\geq D_{k-1})]&\leq \EE[(W_n^k-\bar{W}_{D_{k-1}+n}^{k-1})^2]^{1/2}\PP(Q_{k-1}\geq D_{k-1})^{1/2}\\
	&\leq (2\EE[(W_n^k)^2]+2\EE[(\bar{W}_{D_{k-1}+n}^{k-1})^2])^{1/2}\left(\exp(-\eta D_{k-1})\EE[\exp(\eta Q_{k-1}]\right)^{1/2}\\
	& \leq 2M(k-1)^{-3}\leq 16Mk^{-3} 
\end{align*}
For the first term, by definition, $\bar{W}_{D_{k-1}+n}^{k-1}$ is a waiting time sequence with service and arrival rates $(\mu_{k-1},\lambda(p_{k-1}))$ and $\tilde{W}_n^k$ is a sequence with rates $(\mu_{k},\lambda(p_k))$ or $(\mu_{k},\lambda(p_{k-1}))$.
As a consequence, by applying Lemma \ref{lmm: Lipschitz}, we have
\begin{align*}
	|\tilde{W}_n^k-\bar{W}_{D_{k-1}+n}^{k-1}|\leq& |\tilde{W}_0^k-\bar{W}_{D_{k-1}}^{k-1}|+\left(\frac{|\mu_k-\mu_{k-1}|}{\underline{\mu}}+\frac{|\lambda(p_k)-\lambda(p_{k-1})|}{\lambda(\bar{p})}\right)\max(\tilde{X}_n^k,\bar{X}_{D_{k-1}+n}^{k-1}) \\
	&+\frac{|\mu_k-\mu_{k-1}|}{\underline{\mu}}\max(\tilde{W}_n^k,\bar{W}_{D_{k-1}+n}^{k-1}).
\end{align*}
By Lemma 1, we have that $\max(\tilde{X}_n^k,\bar{X}_{D_{k-1}+n}^{k-1})\leq \frac{\lambda(\underline{p})}{\lambda(\bar{p})}\hat{X}_n^k$ and $\max(\tilde{W}_n^k,\bar{W}_{D_{k-1}+n}^{k-1})\leq \hat{W}_n^k$, where $\hat{X}_n^k$ and $\hat{W}_n^k$ are the observed busy period and waiting time in a stationary $GI/GI/1$ queue with rate $(\underline{\mu},\underline{p})$ as defined in Lemma \ref{lmm: uniform bound}. 
On the other hand, under Condition (b) of Theorem \ref{thm: non-stationary error},
\begin{equation*}
	\begin{aligned}
		\EE[|\mu_k-\mu_{k-1}|^2]&\leq \EE[\|x_k-x_{k+1}\|^2]\leq K_2k^{-2\alpha}\\
		\EE[|\lambda_k-\lambda_{k-1}|^2]&\leq K_2\left(\max_p\lambda'(p)\right)^2k^{-2\alpha}\equiv K_6k^{-2\alpha}. \label{eq: K6}
	\end{aligned}
\end{equation*}
Therefore,
\begin{align*}
	\EE\left[|\mu_k-\mu_{k-1}|\max(\tilde{X}_n^k,\bar{X}_{D_{k-1}+n}^{k-1}) \right] \leq \EE[(\mu_k-\mu_{k-1})^2]^{1/2}\frac{\lambda(\underline{p})}{\lambda(\bar{p})}\EE[(\hat{X}_n^k)^2]^{1/2}\leq \sqrt{K_2}\sqrt{M}k^{-\alpha};\\
	\EE[|\lambda(p_k)-\lambda(p_{k-1})|\max(\tilde{X}_n^k,\bar{X}_{D_{k-1}+n}^{k-1})]\leq \EE[(\lambda_k-\lambda_{k-1})^2]^{1/2}\frac{\lambda(\underline{p})}{\lambda(\bar{p})}\EE[(\hat{X}_n^k)^2]^{1/2}\leq \sqrt{K_6}\sqrt{M}k^{-\alpha};\\
	\EE\left[|\mu_k-\mu_{k-1}|\max(\tilde{W}_n^k,\bar{W}_{D_{k-1}+n}^{k-1}) \right] \leq \EE[(\mu_k-\mu_{k-1})^2]^{1/2}\EE[(\hat{W}_n^k)^2]^{1/2}\leq \sqrt{K_2}\sqrt{M}k^{-\alpha}.
\end{align*}
Finally, by Corollary \ref{coro: convergence to stationarity}, we have $${\EE}[|\tilde{W}_0^k-\bar{W}_{D_{k-1}}^{k-1}|]=\EE[|\bar{W}_{D_{k-1}}^{k-1}-W_0^k|]=\EE[|\bar{W}_{D_{k-1}}^{k-1}-W_{D_{k-1}}^{k-1}|]\leq (A+2M)(k-1)^{-2}\leq (4A+8M)k^{-2}.$$
In summary, we can conclude
\begin{align*}
	|\EE[W_n^k-w(\mu_k,p_k)]|&\leq \EE[|w(\mu_{k-1},p_{k-1})-w(\mu_k,p_k)|]+\EE[|W_n^k-\bar{W}_{D_{k-1}+n}^{k-1}|] \\
	&\leq B\EE[|\mu_k-\mu_{k-1}|+|\lambda(p_k)-\lambda(p_{k-1})|]+\left(\frac{2\sqrt{K_2}}{\underline{\mu}}+\frac{\sqrt{K_6}}{\lambda(\bar{p})}\right)\sqrt{M}k^{-\alpha}+O(k^{-2})\\
	&\leq B(\sqrt{K_2}+\sqrt{K_6})k^{-\alpha}+\left(\frac{2\sqrt{K_2}}{\underline{\mu}}+\frac{\sqrt{K_6}}{\lambda(\bar{p})}\right)\sqrt{M}k^{-\alpha}+O(k^{-2})\\
	&=O(k^{-\alpha}),
\end{align*}
where the second inequality follows from Lemma \ref{lmm: steady state continuity}.
As a direct consequence, 
$|I_2|=O(k^{-\alpha}\log(k))$ as $\tilde{d}_k=O(\log(k))$. 

\end{proof}
\begin{proof}[Proof of Corollary \ref{coro: I3}]
Note that by Lemma \ref{lmm: uniform bound},
$$\left|h_0\EE[W_\infty(\mu_k,p_k)]+\frac{h_0}{\mu_k}-p_k\right|\leq h_0M+h_0\underline{\mu}^{-1}+\bar{p}=O(1).$$
So it suffices to show that
$$\EE[\left|D_k - \lambda(p_k)T_k \right|]=O(k^{-\alpha}),\quad \EE\left[\sum_{n=1}^{Q_k\wedge D_k}|p_k-p_n^k|\right]=O(k^{-\alpha}).$$
Given $\mu_k$ and $p_k$, $T_k$ is the time for the $D_k$-th customer to enter service. Let $F_n^k$ be the inter-service time between the $(n-1)$-th and the $n$-th customers in cycle $k$. Then, $T_k = \sum_{n=1}^{D_k} F_n^k$ and for each $n$,
$$F_n^k =
\begin{cases}
	\frac{U^k_n}{\lambda^k_n}+W_{n}^k-W_{n-1}^k&\text{ for } 1\leq n\leq Q_k\\
	\frac{U^k_n}{\lambda_{k}}+W_{n}^k-W_{n-1}^k& Q_k+1\leq n\leq  D_k.
\end{cases}$$
Therefore,
\begin{equation*}
	\begin{aligned}
		T_k & = \sum_{k=1}^{D_k} F_n^k=\sum_{n=1}^{Q_k\wedge D_k}\frac{U^k_n}{\lambda_n^k} + \frac{1}{\lambda_k}\sum_{n=Q_k+1}^{D_k}U^k_n + W_{D_k}^k-W_{0}^{k}\\
		&= \frac{1}{\lambda_k}\sum_{n=1}^{D_k}U^k_n + W_{D_k}^k-W_{0}^{k} + \sum_{n=1}^{Q_k\wedge D_k} U^k_n\left(\frac{1}{\lambda_n^k}-\frac{1}{\lambda_k}\right).
	\end{aligned}
\end{equation*}
As a consequence,
\begin{equation*}
	\begin{aligned}
		|\EE\left[(D_k-\lambda_kT_k)\right]|&\leq \lambda_k|\EE[W_{D_k}^k]-\EE[W_{0}^{k}]| + \EE\left[\sum_{k=1}^{Q_k\wedge D_k} U^k_n\Big|\frac{\lambda_k}{\lambda_n^k}-1\Big|\right].
	\end{aligned}
\end{equation*}
Following Corollary \ref{coro: convergence to stationarity} and Lemma \ref{lmm: steady state continuity}, for $k\geq 2$, the first term
\begin{align*}
	|\EE[W_{D_k}^k]-\EE[W_{0}^{k}]| &\leq \EE|W_{D_k}^k-\bar{W}_{D_k}^k|+\EE|W_{D_{k-1}}^{k-1}-\bar{W}_{D_{k-1}}^{k-1}| +|\EE[\bar{W}_{D_k}^k]-\EE[\bar{W}_{D_{k-1}}^{k-1}]|\\
	&=(A+2M)\left(k^{-2} +(k-1)^{-2}\right)+ B\sqrt{K_2}k^{-\alpha}=O(k^{-\alpha}).
\end{align*}
As to the second term, by definition, the  customers 1 to $Q_k-1$ arrive to the system while customer $0$ is waiting in the system, and therefore,
$$0\leq \sum_{i=1}^{(Q_k-1)\wedge D_k}\frac{U_i^k}{\bar{\lambda}}\leq\sum_{i=1}^{(Q_k-1)\wedge D_k}\frac{U_i^k}{\lambda_i^k}\leq W_0^k 
\quad\Rightarrow \quad \EE\left[\left(\sum_{i=1}^{Q_k\wedge D_k} U_i^k\right)^{2}\right]\leq \EE\left[(\bar{\lambda}W_0^k+U^k_{Q_k})^2\right]\leq 4\bar{\lambda}^{2}M.$$
Here, $\EE[(U_{Q_k}^k)^2]$ is bounded since we assume that $U$ is light-tailed (Assumption \ref{assmpt: light tail}). For the simplicity of notation, we just assume that $\EE\left[\Big(\frac{U^2_i}{\bar{\lambda}}\Big)^2\right]<M$ for the same $M$ in Lemma \ref{lmm: uniform bound}. Then,
\begin{align*}
	\EE\left[\sum_{k=1}^{Q_k\wedge D_k} U^k_n\Big|\frac{\lambda_k}{\lambda_n^k}-1\Big|\right]&\leq \EE\left[\sum_{k=1}^{Q_k\wedge D_k} U^k_n\Big|\frac{\lambda_k}{\lambda_{k-1}}-1\Big|\right] + \EE\left[\sum_{k=1}^{Q_k\wedge D_k} U^k_n\cdot\frac{\bar{\lambda}}{\underline{\lambda}}\cdot 1(Q_{k-1}\geq D_{k-1})\right]\\
	&\leq 2\bar{\lambda}\sqrt{M}\EE\left[\Big|\frac{\lambda_k}{\lambda_{k-1}}-1\Big|^2\right]^{1/2}+\frac{2\bar{\lambda}^2}{\underline{\lambda}}\sqrt{M}\PP(Q_{k-1}\geq D_{k-1})^{1/2}\\
	&\leq \frac{2\bar{\lambda}\sqrt{M}K_6^{1/2}}{\underline{\lambda}}k^{-\alpha} +  \frac{16\bar{\lambda}^2}{\underline{\lambda}}Mk^{-3}  = O(k^{-\alpha}).
\end{align*}
Finally, 
\begin{equation*}
	\begin{aligned}
		\EE\left[\sum_{n=1}^{Q_k\wedge D_k}|p_k-p_n^k|\right]&\leq \EE\left[\sum_{n=1}^{Q_k\wedge D_k}|p_k-p_n^k|\cdot 1(Q_{k-1}<D_{k-1})\right]\\
		&~~+\EE\left[\sum_{n=1}^{Q_k\wedge D_k}|p_k-p_n^k|\cdot 1(Q_{k-1}\geq D_{k-1})\right]\\
		&\leq \EE\left[Q_{k-1}^2\right]^{1/2}\EE\left[|p_k-p_{k-1}|^2\right]^{1/2}+\EE\left[\bar{p}^2Q_k^2\right]^{1/2}\PP(Q_{k-1}\geq D_{k-1})^{1/2}\\
		&\leq \sqrt{MK_2}k^{-\alpha}+8\bar{p}Mk^{-3}=O(k^{-\alpha})
	\end{aligned}
\end{equation*}
Therefore, $I_3 = O(k^{-\alpha})$. 
\end{proof}
\begin{proof}[\textbf{Finishing the proof of Theorem \ref{thm: non-stationary error}}]
First, by Corollary \ref{coro: convergence to stationarity}, we have
\begin{equation*}
	\begin{aligned}
		|I_1|\leq \frac{A}{1-e^{-\gamma}}k^{-1} + 2MK_2k^{-\alpha}=O(k^{-\alpha}). 
	\end{aligned}
\end{equation*}
By Corollary \ref{coro: error of leftovers},
\begin{align*}
	|I_2|\leq \frac{5}{\min(\gamma,\eta)}\left(B(\sqrt{K_2}+\sqrt{K_6})+\left(\frac{2\sqrt{K_2}}{\underline{\mu}}+\frac{\sqrt{K_6}}{\underline{\lambda}}\right)\sqrt{M}+(4A+8M)\right)k^{-\alpha}\log(k)=O(k^{-\alpha}\log(k)).
\end{align*}
Following the proof of Corollary \ref{coro: I3}, we have
\begin{align*}
	I_3&\leq (h_0M+h_0\underline{\mu}^{-1}+\bar{p})\left(\frac{2\bar{\lambda}\sqrt{M}K_6^{1/2}}{\underline{\lambda}}+ \frac{16\bar{\lambda}^2}{\underline{\lambda}}M\right)k^{-\alpha} +(\sqrt{MK_2}+8\bar{p}M)k^{-\alpha}=O(k^{-\alpha}).
\end{align*}
Therefore, we can conclude that  $\forall k\geq2$, $R_{1,k}\leq K'\cdot k^{-\alpha}\log(k)$ with
\begin{align}\label{eq: K}
	K' = ~&\frac{Ah_0}{1-e^{-\gamma}} + 2h_0MK_2 +\frac{5h_0}{\min(\gamma,\eta)}\left(B(\sqrt{K_2}+\sqrt{K_6})+\left(\frac{2\sqrt{K_2}}{\underline{\mu}}+\frac{\sqrt{K_6}}{\underline{\lambda}}\right)\sqrt{M}+(4A+8M)\right)\notag\\
	&+(h_0M+h_0\underline{\mu}^{-1}+\bar{p})\left(\frac{2\bar{\lambda}\sqrt{M}K_6^{1/2}}{\underline{\lambda}}+ \frac{16\bar{\lambda}^2}{\underline{\lambda}}M\right)+\sqrt{MK_2}+8\bar{p}M.
\end{align}
Let $M_0>0$ be the upper bound of the regret in the first cycle. Here the constant $M_0<\infty$ since the decision region $\mathcal{B}$ is bounded and by condition (a), $D_1\leq K_2$ is also bounded. Finally, we conclude that 
\begin{align*}
	R_1(L)&\leq M_0+K'\sum_{k=2}^L k^{-\alpha}\log(k)
	\leq K\sum_{k=1}^L k^{-\alpha}\log(k).
\end{align*}
with $K = K'+\frac{2M_0}{\log(2)}$.
\end{proof}

\subsection{Convergence Rate of Observed Busy Period}\label{subsect: busy period}
As an analogue of Lemma \ref{lmm: thm1 1}, we prove a uniform convergence rate for the observed busy period $X_n$, which will be used to bound $B_k$ and $\mathcal{V}_k$ of the gradient estimator \eqref{eq: derivative} that involves terms of $X_n^k$.
\begin{lemma}\label{lmm: convergence of busy period}
Let $X^1_n$ and $X^2_n$ be the observed busy period of the two queueing systems coupled as in Lemma \ref{lmm: thm1 1}, with $X^1_0, X^2_0\leq_{st} \frac{\lambda(\underline{p})}{\lambda(\bar{p})}\hat{X}_0$ and $W^1_0, W^2_0\leq_{st} \hat{W}_0$.
\begin{enumerate}
	\item $|X^1_n-X^2_n|\leq {\bf 1}_{\{\max(\Gamma_1, \Gamma_2)>n\}}\left(\sum_{k=1}^n\tau_k+X^1_0+X^2_0\right)$.
	\item There exists a constant $K_4>0$ such that $|\EE[X^1_n-X^2_n]|^m\leq K_4e^{-0.5\gamma n}n^{2}$ for all $n\geq 1$ and $m\leq 2$.
\end{enumerate}
\end{lemma}
\begin{proof}[Proof of Lemma \ref{lmm: convergence of busy period}]
1. Following the argument in Lemma \ref{lmm: thm1 1}, if $W^1_0\geq W^2_0$, we will have $W^1_{\Gamma_1}=W^2_{\Gamma_1}=0$ and hence $X^1_{\Gamma_1} = X^2_{\Gamma_1}=0$. Since the two systems share the same sequence of arrivals and service times,  $X^1_n= X^2_n$ for all $n\geq \Gamma_1$. Therefore,
$$|X^1_n-X^2_n|\leq {\bf 1}_{\{\max(\Gamma_1,\Gamma_2)>n\}}|X^1_n-X^2_n|\leq  {\bf 1}_{\{\max(\Gamma_1,\Gamma_2)>n\}}\left(\sum_{k=1}^n\tau_k+X^1_0+X^2_0\right).$$
The last inequality follows from $0\leq X^i_n \leq X_0^i+\sum_{k=1}^n\tau_k$ for $i = 1, 2$.\\


2. Following 1 and part 2 of Lemma \ref{lmm: thm1 1}, for $m=1,2$,
\begin{equation*}
	\begin{aligned}
		\EE[|X^1_n-X^2_n|^m]&\leq \EE\left[{\bf 1}_{\{\max(\Gamma_1,\Gamma_2)>n\}}\left(\sum_{k=1}^n\tau_k+X^1_0+X^2_0\right)^m\right]\\
		&\leq \PP(\max(\Gamma_1,\Gamma_2)>n)^{1/2} \EE\left[\left(\sum_{k=1}^n\tau_k+X^1_0+X^2_0\right)^{2m}\right]^{1/2}
	\end{aligned}
\end{equation*}
where
$$\PP(\max(\Gamma_1,\Gamma_2)>n)\leq e^{-n\gamma}\EE[2+e^{\mu\theta W_0^1}+e^{\mu\theta W_0^2}]\leq e^{-n\gamma}(2+2M),$$
and
$$\EE\left[\left(\sum_{k=1}^n\tau_k+X^1_0+X^2_0\right)^{2m}\right]\leq 3^{2m-1}\left(n^{2m} \EE\left[\frac{U_1^{2m}}{\lambda(p)^{2m}}\right]+\EE[(X^1_0)^{2m}]+\EE[(X^2_0)^{2m}]\right).$$
Therefore,
$$\EE[|X^1_n-X^2_n|^m]\leq K_4 e^{-0.5n\gamma}n^{2},$$
with $K_4 = 3^m\left(\max_{1\leq m\leq 2}\EE[U_1^{2m}]/\lambda(\bar{p})^{2m} + 2\frac{\lambda(\underline{p})^{2m}}{\lambda(\bar{p})^{2m}}M\right)^{1/2}(2+2M)^{1/2}$.
\end{proof}
\subsection{Proof of Theorem \ref{thm: broadiezeevi}}\label{subsect: R1 analysis}
The proof follows an induction-based approach similar to \cite{BoradieZeevi2011}. 
For simplicity of notation, we write $\Delta_k=k^{-\beta}$. 
Let $\mathcal{F}_k$ be the filtration up to cycle $k$, i.e. including all events in the first $k-1$ cycles. Since $x_{k+1}=\pi_\mathcal{B}(x_k-\eta_kH_k)$, 
\begin{align*}
&\EE\left[\|x_{k+1}-x^*\|^2]=\EE[\|x_{k}-x^*-\eta_k H_k\|^2\right]\\
=~&\EE\left[\|x_k-x^*\|^2 - 2\eta_k H_k\cdot( x_k-x^*)+ \eta_k^2H_k^2\right]\\
=~&\EE\left[\|x_k-x^*\|^2 -2\eta_k\nabla f(x_k)\cdot (x_k-x^*)\right]- \EE[2\eta_k(H_k-f(x_k))\cdot( x_k-x^*)] + \EE[\eta_k^2H_k^2]\\
=~&(1-2\eta_k K_0)\EE\left[\|x_k-x^*\|^2\right] + \EE[2\eta_k(H_k-f(x_k))\cdot( x^*-x_k)] + \eta_k^2\EE[H_k^2].
\end{align*}
Note that
\begin{align*}
&\EE[2\eta_k(H_k-\nabla f(x_k))\cdot( x^*-x_k)] =~\EE[\EE[2\eta_k(H_k-\nabla f(x_k))\cdot( x^*-x_k)|\mathcal{F}_k]] \\
=~& 2\eta_k\EE[\EE[H_k-\nabla f(x_k)|\mathcal{F}_k]\cdot( x^*-x_k)] 
\leq~  2\eta_k \EE[\|\EE[H_k-\nabla f(x_k)|\mathcal{F}_k]\|^2]^{1/2}\EE[\|x^*-x_k\|^2]^{1/2}\\
\leq~& \eta_k \EE[\|\EE[H_k-\nabla f(x_k)|\mathcal{F}_k]\|^2]^{1/2} (1+\EE[\|x_k-x^*\|^2]).
\end{align*}
The second last inequality follows from 
$ab+cd\leq \sqrt{a^2+c^2}\sqrt{b^2+d^2}$ and the Holder Inequality, the last inequality follows from $2a\leq 1+a^2$.

Let  $b_k=\EE[\|x_k-x^*\|^2]$ and recall that
$B_k = \EE[\|\EE[H_k-\nabla f(x_k)|\mathcal{F}_k]\|^2]^{1/2}, \quad \mathcal{V}_k = \EE[H_k^2]$.
Then, we obtain the recursion
\begin{equation*}\label{eq: BZ recursion}
b_{k+1}\leq (1-2K_0\eta_k + \eta_kB_k)b_k +\eta_kB_k + \eta_k^2 \mathcal{V}_k.
\end{equation*}
By Condition (b) and (c), we have
$$
b_{k+1}\leq (1-2K_0\eta_k + \eta_kB_k)b_k +\eta_kB_k + \eta_k^2 \mathcal{V}_k\leq \left(1-2K_0\eta_k + \frac{K_0}{8}\eta_k\Delta_k \right)b_k +\frac{K_0}{8}\eta_k\Delta_k + K_3\eta_k\Delta_k.
$$
{Because step size $\eta_k\rightarrow0$, for $k$ large enough, $\eta_k K_0\leq 1/2$. Let $k_0=\max\{k:2\eta_kK_0>1\}$. Then, for $k>k_0$, $1-2K_0\eta_k+\frac{K_0}{8}\eta_k\Delta_k>0$.}
By Condition (a), $\Delta_{k}/\Delta_{k+1}= (1+\frac{1}{k})^\beta\leq 1+\frac{1}{k}\leq 1+\frac{K_0}{2}\eta_k$, and by the induction assumption $b_k\leq C\Delta_{k}$, {for $k>k_0$,} we have
\begin{equation*}
\begin{aligned}
	b_{k+1}&\leq \left(1-2K_0\eta_k+\frac{K_0}{8}\eta_k\Delta_k\right)\left(1+\frac{K_0\eta_k}{2}\right)C\Delta_{k+1}+\frac{K_0}{8}\eta_k\Delta_k + K_3\eta_k\Delta_k\\
	&\leq C\Delta_{k+1} -\eta_k\Delta_k\left(\frac{3K_0C}{2}-\frac{K_0C}{8}\Delta_k - \frac{K_0^2C}{16}\eta_k\Delta_k -\frac{K_0}{8}-K_3\right)\\
\end{aligned}
\end{equation*}
Then, we have $b_{k+1}\leq  C\Delta_{k+1}$
as long as
\begin{equation}\label{eq: proof thm1}
\frac{3K_0C}{2}-\frac{K_0C}{8}\Delta_k - \frac{K_0^2C}{16}\eta_k\Delta_k -\frac{K_0}{8}-K_3 \geq 0.
\end{equation}
To check \eqref{eq: proof thm1}, note that, $\Delta_k,K_0\leq 1$ and $C\geq 8K_3/K_0$. {Besides, $\eta_k K_0\leq 1/2<1$ for $k>k_0$.} Then, for $k\geq k_0$, 
$$\frac{3K_0C}{2}-\frac{K_0C}{8}\Delta_k - \frac{K_0^2C}{16}\eta_k\Delta_k -\frac{K_0}{8}-K_3 \geq \frac{3K_0C}{2}-\frac{K_0C}{8}-\frac{K_0C}{16}-\frac{K_0C}{8}-\frac{K_0C}{8} = \frac{17K_0C}{16}>0.$$
Let
\begin{equation}\label{eq: C}
C= \max\left(k_0^{\beta}(|\bar{\mu}-\underline{\mu}|^2+|\bar{p}-\underline{p}|^2), 8K_3/K_0\right).
\end{equation}
Then we have $\|x_k-x^*\|^2\leq C\Delta_k$ for all $1\leq k\leq k_0$, and we can conclude by induction, for all $k\geq k_0$,
$$\EE[\|x_k-x^*\|^2]\leq Ck^{-\beta}.$$
By Assumption \ref{assmpt: convexity}, there exists $\theta_0\in [0,1]$ such that
$$|f(x_k)-f(x^*)|=|\nabla f(\theta_0(x^k-x^*)+x^*)^T(x_k-x^*)|\leq K_1\|x_k-x^*\|^2.$$
As a consequence,
$$R_2(L)\leq \sum_{k=1}^L \EE[T_k]K_1Ck^{-\beta}.$$ Note that $T_k$ equals to the arrival time of customer $D_k$ plus its waiting time. Therefore,
$$\EE[T_k]\leq \EE\left[\frac{D_k}{\lambda_k}\right]+\EE[W_{D_k}^k]\leq \frac{D_k}{\lambda(\bar{p})}+M =O(D_k) ,$$
and we can conclude
$$R_2(L)=O\left(\sum_{k=1}^L D_kk^{-\beta}\right).$$ \hfill $\Box$


\subsection{Proof of Theorem \ref{thm: regret direct}}\label{subsect: direct proof}
(i) For each $k$, note that $x_k\in\mathcal{F}_k$, let's denote by 
\begin{align*}
h_k^1&= -\lambda(p_k) -p_k\lambda'(p_k)+h_0\lambda'(p_k)\left[\frac{1}{ \lceil D_k(1-\xi)\rceil}\sum_{n> \xi D_k}^{D_k} \left(\EE[X_{n}^k|\mathcal{F}_k] +\EE[W_{n}^k|\mathcal{F}_k]\right) +\frac{1}{\mu}\right],\\
h_k^2&=c'(\mu_k) - h_0\frac{\lambda(p_k)}{\mu_k}\left[\frac{1}{\lceil D_k(1-\xi)\rceil}\sum_{n> \xi D_k}^{D_k} \left(\EE[X_{n}^k|\mathcal{F}_k] +\EE[W_{n}^k|\mathcal{F}_k]\right)+\frac{1}{\mu}\right].
\end{align*}
Then,
$$\|\EE[H_k-\nabla f(x_k)|\mathcal{F}_k]\|^2 = \left|h_k^1-\frac{\partial}{\partial p}f(\mu_k,p_k)\right|^2 +\left|h_k^2-\frac{\partial}{\partial \mu}f(\mu_k,p_k)\right|^2.$$
Following \eqref{eq: derivative},
\begin{align*}
|h_k^1-\frac{\partial}{\partial p}f(\mu_k,p_k)|^2&\leq \frac{h_0^2\lambda'(p_k)^2}{\lceil D_k(1-\xi)\rceil}\sum_{n> \xi D_k}^{D_k}\left(|\EE[X_n^k - x_k|\mathcal{F}_k]| + |\EE[W_n^k-w_k|\mathcal{F}_k]|\right)^2,\\
|h_k^2-\frac{\partial}{\partial \mu}f(\mu_k,p_k)|^2&\leq \frac{h_0^2\lambda(p_k)^2}{\mu_k^2\lceil D_k(1-\xi)\rceil}\sum_{n> \xi D_k}^{D_k}\left(|\EE[X_n^k - x_k|\mathcal{F}_k]| + |\EE[W_n^k-w_k|\mathcal{F}_k]|\right)^2,
\end{align*}
where $w_k=\EE[W_\infty(\mu_k,p_k)]$ and $x_k=\EE[X_\infty(\mu_k,p_k)]$.
Note that $\lambda(p)$, $\lambda'(p)$ and $\mu$ are bounded. Let $C_0=\max_{(\mu,p)\in \mathcal{B}}\{h_0\lambda'(p_k),h_0\lambda(p)/\mu\}$, then
\begin{align*}
\|\EE[H_k-\nabla f(x_k)|\mathcal{F}_k]\|^2& \leq \frac{2C_0^2}{\lceil D_k(1-\xi)\rceil}\sum_{n> \xi D_k}^{D_k}\left(|\EE[X_n^k - x_k|\mathcal{F}_k]| + |\EE[W_n^k-w_k|\mathcal{F}_k]|\right)^2\\
&\leq \frac{4C_0^2}{\lceil D_k(1-\xi)\rceil}\sum_{n> \xi D_k}^{D_k}\left(|\EE[X_n^k - x_k|\mathcal{F}_k]|^2 + |\EE[W_n^k-w_k|\mathcal{F}_k]|^2\right)  \\
&= \frac{4C_0^2}{\lceil D_k(1-\xi)\rceil}\sum_{n> \xi D_k}^{D_k}\left(|\EE[X_n^k - \bar{X}_k^n|\mathcal{F}_k]|^2 + |\EE[W_n^k-\bar{W}_k^n|\mathcal{F}_k]|^2\right)  
\end{align*}
where the last equality follows from $\EE[\bar{W}_k^n|\mathcal{F}_k]=w_k$ and $\EE[\bar{X}_k^n|\mathcal{F}_k]=x_k$ and $\bar{W}_k^n$ and $\bar{X}_k^n$ are stationary versions of the waiting times and observed busy periods that are synchronously coupled with $W_k^n$ and $X_k^n$ respectively. Therefore, the bias
\begin{align*}
B_k^2 &= \EE[ \|\EE[H_k-\nabla f(x_k)|\mathcal{F}_k]\|^2]\\
&\leq  \EE\left[\frac{4C_0^2}{\lceil D_k(1-\xi)\rceil}\sum_{n> \xi D_k}^{D_k}\left(|\EE[X_n^k - \bar{X}_k^n|\mathcal{F}_k]|^2 + |\EE[X_n^k-\bar{X}_k^n|\mathcal{F}_k]|^2\right) \right] \\
& \leq \frac{4C_0^2}{\lceil D_k(1-\xi)\rceil}\sum_{n> \xi D_k}^{D_k}\left(|\EE[(X_n^k - \bar{X}_k^n)^2] + \EE[(W_n^k-\bar{W}_k^n)^2]\right)
\end{align*}

Following a similar argument as in the proof of Corollary \ref{coro: convergence to stationarity}, we have, for $n\geq \lceil 0.5\xi D_k\rceil$,
\begin{equation*}
\begin{aligned}
	\EE[(W_n^k-\bar{W}_k^n)^2]&\leq \EE[(W_n^k-\bar{W}_k^n)^2\cdot 1(Q_k<0.5\xi D_k)]
	+ \EE[(W_n^k-\bar{W}_k^n)^2\cdot 1(Q_k\geq 0.5\xi D_k)] \\
	&\leq A\exp(-\gamma\cdot (n-0.5\xi D_k))+2M\exp(-\eta\cdot 0.25\xi D_k))\\
	&\leq (A +2M)\exp(-\min(\gamma,\eta)\cdot 0.25\xi D_k).
\end{aligned}
\end{equation*}
For the observed busy period $X_n^k$, following a similar analysis and Lemma \ref{lmm: convergence of busy period}, we have
\begin{equation*}
\begin{aligned}
	&\EE[(X_n^k-\bar{X}_n^k)^2]\\
	\leq~ & K_4e^{-0.5\gamma\xi D_k} D_k^{2}+(2\EE[(X_n^k)^4]+2\EE[(\bar{X}_k^n)^4])^{1/2}\PP(Q_k\geq 0.5\xi D_k)^{1/2}\\
	\leq~ & \exp(-\min(\gamma,\eta)\cdot 0.25\xi D_k)(2M +K_4D_k^{2})\leq \exp(-\min(\gamma,\eta)\cdot 0.125\xi D_k)(2M +K_4K_5),
\end{aligned}
\end{equation*}
where 
\begin{equation}K_5 = \max_{D>0}\exp(-\min(\gamma,\eta)\cdot 0.125\xi D) D^{2}=\left(\frac{16}{\min(\gamma,\eta)\cdot \xi}\right)^2e^{-2}.
\label{eq: K5}
\end{equation}
If we choose
\begin{equation*}
D_k = a_D + b_D\log(k), \text{ for } a_D\geq \frac{C_D}{\min(\gamma,\eta)\xi} \text{ and } b_D\geq \frac{8}{\min(\gamma,\eta)\xi},
\end{equation*}
with 
\begin{equation}\label{eq: Dk direct}
C_D=\max(8(\log((16A+32M)C_0/K_0),16\log((32M+16K_4K_5)C_0/K_0)),
\end{equation}
then
$$\EE[(W_{n}^k-\bar{W}_n^k)^2]\leq \frac{K_0^2}{256C_0^2k^2}, \quad \EE[(X_n^k-\bar{X}_n^k)^2]\leq \frac{K_0^2}{256C_0k^2}.
$$
As a consequence,
$$\EE[\|\EE[H_k-\nabla f(x_k)|\mathcal{F}_k]\|^2] \leq \frac{4C_0^2}{\lceil D_k(1-\xi)\rceil}\sum_{n> \xi D_k}^{D_k}\left(|\EE[(X_n^k - \bar{X}_k^n)^2] + \EE[(W_n^k-\bar{W}_k^n)^2]\right)\leq \frac{K_0^2}{64k^2}.$$
Therefore, we can conclude that  $$B_k =\EE[\|\EE[H_k-\nabla f(x_k)|\mathcal{F}_k]\|^2]^{1/2}\leq \frac{K_0}{8k}.$$

On the other hand, as $\lambda(p)$, $\lambda'(p)$ and $\mu$ are bounded, $C_1\triangleq \max_{\mu,p\in\mathcal{B}}\{|\lambda(p)+p\lambda'(p)|,|c'(\mu)|\}<\infty$. Recall that $C_0=\max_{(\mu,p)\in \mathcal{B}}\{h_0\lambda'(p_k),h_0\lambda(p)/\mu\}$. Then,
$$	\EE[\|H_k\|^2]\leq 8(C_1+C_0/\underline{\mu})^2 + 8C_0^2\EE\left[ \frac{1}{\lceil (1-\xi)D_k\rceil^2} \left( \sum_{n> \xi D_k}^{D_k} \left(X_{n}^k +W_{n}^k \right)\right)^2\right].$$
By Lemma \ref{lmm: uniform bound}, we have
\begin{align*}
	&\EE\left[ \frac{1}{\lceil (1-\xi)D_k\rceil^2}\left( \sum_{n> \xi D_k}^{D_k} \left(X_{n}^k +W_{n}^k \right)\right)^2\right]
	\leq~\EE\left[\frac{1}{\lceil (1-\xi)D_k\rceil^2} \left( \sum_{n> \xi D_k}^{D_k} \left(\frac{\lambda(\underline{p})}{\lambda(\bar{p})}\hat{X}_{n}^k +\hat{W}_{n}^k \right)\right)^2\right],
\end{align*}
where $\hat{W}_n^k$ and $\hat{X}_n^k$ are defined as in Lemma \ref{lmm: uniform bound}. Note that by definition,  $\hat{W}_n^k$ and $\hat{X}^k_n$ are stationary, we have
\begin{equation*}
	\begin{aligned}
		&\EE\left[\frac{1}{\lceil (1-\xi)D_k\rceil^2} \left( \sum_{n> \xi D_k}^{D_k} \left(\frac{\lambda(\underline{p})}{\lambda(\bar{p})}\hat{X}_{n}^k +\hat{W}_{n}^k \right)\right)^2\right]\\
		\leq &\frac{2}{\lceil (1-\xi)D_k\rceil^2}\EE\left[ \left( \sum_{n> \xi D_k}^{D_k}\frac{\lambda(\underline{p})}{\lambda(\bar{p})}\hat{X}_{n}^k\right)^2\right]+\frac{2}{\lceil(1-\xi)D_k\rceil^2}\EE\left[ \left(\sum_{n> \xi D_k}^{D_k}\hat{W}_{n}^k \right)^2\right]\\
		\leq &2(1-\xi)^{-2}\EE\left[\left(\frac{\lambda(\underline{p})}{\lambda(\bar{p})}\hat{X}_{0}^k\right)^2\right] + 2(1-\xi)^{-2}\EE[(\hat{W}_0^k)^2]\leq 4(1-\xi)^{-2}M.
	\end{aligned}
\end{equation*}

Therefore, $\mathcal{V}_k$ is uniformly bounded. Given that $\eta_k = c_\eta k^{-1}$, we have
$\eta_k\mathcal{V}_k\leq \frac{K_3}{k}$ with
\begin{equation}\label{eq: K_3}
	K_3=(8(C_1+C_0/\underline{\mu})^2+32C_0^2(1-\xi)^{-2}M)c_\eta .
\end{equation}

\noindent (ii) 	According to the update rule, we immediately got
$$\EE[\|x_k-x_{k+1}\|^{2}] \leq \eta^{2}_k\EE[\|H_k\|^{2}]\leq 2k^{-2}K_3/K_0\equiv K_2k^{-2}, \text{ with } K_2=2K_3/K_0.$$

\noindent (iii) We have just proved that the conditions of Theorem \ref{thm: non-stationary error} are satisfied with $\alpha=1$. Therefore, $R_1(L)\leq K\sum_{k=1}^L k^{-1}\log(k)\leq K\log(L)^2$ with the expression of $K$ given in \eqref{eq: K}. Besides, conditions of Theorem \ref{thm: broadiezeevi} are satisfied with $\beta=1$ and $D_k = O(\log(k))$. In particular, $\Delta_k/\Delta_{k+1} = 1 +\frac{1}{k}\leq 1+\frac{K_0}{2\eta_k}$ given that $\eta_k = c_\eta k^{-1}$ with $c_\eta\geq 2/K_0$. Therefore,
$$R_2(L)\leq CK_1\sum_{k=1}^L\left(\frac{D_k}{\lambda(\bar{p})}+M\right)k^{-1}=O(\log(L)^2).$$
As a consequence, the total regret
$$R(L) = R_1(L)+R_2(L)\leq K_{alg}\log(L)^2\leq K_{alg}\log(M_L)^2, \text{with }M_L = \sum_{k=1}^L D_k.$$
The last inequality uses  $\log(L)^2\leq \log(M_L)^2$. Since $M_L=O(L\log(L))$, the relaxation from $L$ to $M_L$ will not change the order of the regret bound. In addition, we can find a closed-expression for $K_{alg}$ as
\begin{equation}\label{eq: K_alg}
	K_{alg}=K+CK_1\cdot\left(\frac{C_D+8}{\lambda(\bar{p})\min(\gamma,\eta)\xi}+M\right),
\end{equation}
where $K$ is defined by \eqref{eq: K}, $C$ by \eqref{eq: C} and $C_D$ by \eqref{eq: Dk direct}.	\hfill $\Box$

\subsection{Details in the Proof of Lemma \ref{lmm: derivative process}}\label{subsec: glasserman}
We first give a rigorous proof of \eqref{eq: IPA} in derivation of the partial derivation $\frac{\partial}{\partial p}\EE[W_\infty(p,\mu)]$. To better explain the proof, we adopt the notions in \cite{Glasserman_1992}. We will take derivative with respective to the parameter $\theta =r=1/\lambda(p)$. With a slight abuse of notation, we redefine $W_n(\theta)=W_n(\mu,p)$ and $\tilde{U}_n(\theta)=\frac{V_n}{\mu}-\theta U_n$ so that $\tilde{U}_n'(\theta)=-U_n$.
And then, the Lindley recursion becomes
$$W_{n+1}(\theta) = \phi(W_n(\theta), \tilde{U}_n(\theta)), \quad \text{with } \phi(w,u)=(w+u)^+.$$
Note that the function $\phi$ is increasing and convex in $w$ and $u$. In addition, the derivative process is denote as $V_n(\theta)=Z_n$. Define
$\psi_w(w,u)=\psi_u(w,u)=1(w+u>0),$ such that $$V_{n+1}(\theta)=\psi_w(W_n(\theta),\tilde{U}_n(\theta))V_n(\theta)+\psi_u(W_n(\theta),\tilde{U}_n(\theta))\tilde{U}_n'(\theta).$$ The stationary versions of the waiting time and derivative process are denoted as $\tilde{W}_0(\theta)$ and $\tilde{V}_0(\theta)$.
Then we can check Conditions (B1) to (B3) on page 377 of \cite{Glasserman_1992}:
\begin{enumerate}
	\item[(B1)] For each $\theta\in[1/\lambda(\underline{p}),1/\lambda(\bar{p})]$, the sequence
	$$\{(\tilde{U}_n(\theta),\tilde{U}_n'(\theta)), -\infty<n<\infty~\} = \Big\{\underbrace{\left(\frac{V_n}{\mu}-rU_n,-U_n\right), -\infty<n<\infty}_{\text{in our notation}}\Big\}$$
	is stationary and ergodic, as we can extend the i.i.d. sequences $V_n$ and $U_n$ to $-\infty<n\leq 0$.
	\item[(B2)] For each $\theta\in[1/\lambda(\underline{p}),1/\lambda(\bar{p})]$, the Lindley recursion has a stationary solution $\tilde{W}_0(\theta)$, which is guaranteed by Assumption \ref{assmpt: uniform}. Besides,  following Lemma \ref{lmm: thm1 1}, for any initial state $W_0(\theta)$, the transient process $W_n(\theta)$ will converge to the stationary version in finite time almost surely.
	\item[(B3)] For all $\theta\in[1/\lambda(\underline{p}),1/\lambda(\bar{p})]$, $$\PP(\psi_w(\tilde{W}_0(\theta),\tilde{U}_0(\theta))=0)=\PP\underbrace{\left(\left(W_\infty(\mu,p)+\frac{V_0}{\mu}-rU_0\right)^+=0\right)}_{\text{(in our notation)}}=\PP(W_\infty(\mu,p)=0)>0.$$ 
	According to the discussion on p.379 of \cite{Glasserman_1992}, Condition (B3) holds for $GI/GI/1$ queues under the usual stability condition that $\mu>\lambda(p)$. Below, we give a detailed verification of this condition under our model setting.\\
	Recall that $\tilde{U}_0(r)=\frac{V_0}{\mu}-rU_0$ and by Assumption 1, $\EE[\tilde{U}_0(r)]<0,~\forall r\in[1/\lambda(\underline{p}),1/\lambda(\bar{p})]$. So there exists a constant $b>0$, such that $\PP(\tilde{U}_0(r)<-b)>0$ for all $r\in[1/\lambda(\underline{p}),1/\lambda(\bar{p})].$
	Let $S$ denote the support of $W_\infty(\mu,p)$ and let $A=\inf S\geq 0$. We first show by contradiction that $A=0$. 
	Since $A$ is the infimum of the support,
	\begin{align*}
		\PP\left(W_\infty(\mu,p)\in[A,A+\varepsilon)\right)>0, \text{ for any }\varepsilon>0.
	\end{align*}
	Besides, if $A>0$,
	\begin{align*}
		\PP(W_\infty(\mu,p)\geq A)&= \PP\left(\left(W_\infty(\mu,p)+\tilde{U}_0(r)\right)^+\geq A\right)\\
		&=\PP\left(W_\infty(\mu,p)+\tilde{U}_0(r)\geq A\right)=1,
	\end{align*}
	On the other hand, we have
	\begin{align*}
		\PP\left(W_\infty(\mu,p)+\tilde{U}_0(r)< A\right)\geq\PP\left(W_\infty(\mu,p)\in\Big[A,A+\frac{b}{2}\Big),~\tilde{U}_0(r)<-b\right)>0,
	\end{align*}
	where the last inequality follows from the fact that $W_\infty(\mu,p)$ and $\tilde{U}_0(r)$ are independent in the $GI/GI/1$ queue. This is a contradiction, so we can conclude that $A=0$. Next, we show that $\PP(W_\infty(\mu,p)=0)>0$.  Following a similar derivation, we can conclude
	\begin{align*}
		\PP(W_\infty(\mu,p)=0)&=\PP\left(\left(W_\infty(\mu,p)+\tilde{U}_0(r)\right)^+=0\right)
		\geq \PP\left(W_\infty(\mu,p)\in\Big[0,\frac{b}{2}\Big),~\tilde{U}_0(r)<-b\right)>0.
	\end{align*}
\end{enumerate}
In addition, we have $\EE[\tilde{W}_0(\theta)]\leq M$ and $\EE[\tilde{V}_0(\theta)]=\EE[\tilde{Z}_{\infty}]=\EE[X_\infty(\mu,p)]\leq M$ following Lemma \ref{lmm: uniform bound}.
As a consequence, we can prove \eqref{eq: IPA} using the following Corollary 5.3 in \cite{Glasserman_1992}:
\begin{lemma}[Corollary 5.3 in \cite{Glasserman_1992}]Suppose that $\phi$ is increasing and (jointly) convex, and that $W_0$ and $U_0$ are almost surely convex. Suppose (B1)-(B3) hold, $\EE[\tilde{W}_0(\theta)], \EE[\tilde{V}_0(\theta)]<\infty$ for all $\theta$ in its range. Then, $\EE[\tilde{V}_0(\theta)]=\EE[\tilde{W}_0(\theta)]'$ and 
	$$\lim_{n\to\infty}\frac{1}{n}\sum_{i=1}^n V_i(\theta)=\EE[\tilde{W}(\theta)]', a.s.$$
	almost everywhere in the range of $\theta$.\end{lemma}

The derivation of $\frac{\partial}{\partial \mu}\EE[W_\infty(p,\mu)]$ follows a similar argument with
$\tilde{U}(\theta) \equiv V_n -\theta U_n/\lambda(p)$.


\section{Relaxing Theoretical Bounds of Hyperparameters}\label{sec: para_robust}
In this section, we conduct numerical experiments to investigate the robustness of GOLiQ's performance to the two main hyperparameters: (i) cycle length $D_k$, and (ii) step size $\eta_k$. 
We follow two steps:
\begin{itemize}
	\item First, we calculate the theoretical bounds of these hyperparameters according to \eqref{par:eta} and \eqref{par:D}. 
	\item Next, we test the algorithm's performance while varying these hyperparameters; we especially consider values that violate their corresponding theoretical bounds. 
\end{itemize}

\subsection{Theoretical bounds for $\eta_k$ and $D_k$}
We follow Section \ref{sec:TwoDim} by considering the $M/M/1$ example having the objective function in \eqref{eq:ObjectiveMM1} and demand function in \eqref{logisticLmd}, with $a=4.1$, size $n=10$ and $c_0=0.1$. {In order to obtain the theoretic bounds for hyper-parameters, we set the region $\mathcal{B}=[6.7,10]\times[3.7,5]$ so that $f(\mu,p)$ is strongly convex on $\mathcal{B}$. }

\paragraph{\textbf{Theoretical bound for $\eta_k$.}}
According to the conditions in Assumption \ref{assmpt: convexity}, we note that the Hessian matrix of the objective $f(\mu,p)$ has a smallest eigenvalue $0.1231$ {in the specific region $\mathcal{B}$}, which implies that $K_0=0.1231$ (and the strong convexity of the objective function {on $\mathcal{B}$}). Hence, following from \eqref{par:eta}, the theoretical lower bound for $\eta_k$ is $c_\eta\geq \tilde{c}_\eta=2/K_0=16.24$.

\paragraph{\textbf{Theoretical bound for $D_k$.}} 
To calculate the lower bounds of $a_D$ and $b_D$ specified in \eqref{par:D}, we first estimate $C$ and $(\gamma,\eta)$. We set $\xi=1$. 
First, according to the expression \eqref{eq: Dk direct} and $K_0=0.1231$, we see that $C\geq8$. 
Next, following \eqref{eq: gamma0}, we select $\min(\gamma,\eta)=0.011$ which gives the smallest theoretical lower bound. 

Hence, \eqref{par:D} requires that $a_D\geq \tilde{a}_D = 8/0.011=727$ and $b_D\geq \tilde{b}_D = 8/0.011=727$, which leads to a bound for the cycle length $D_k\geq 727+727\log(k)$.

\begin{figure}[H]
	\vspace{-0.1in}
	\includegraphics[width=\linewidth]{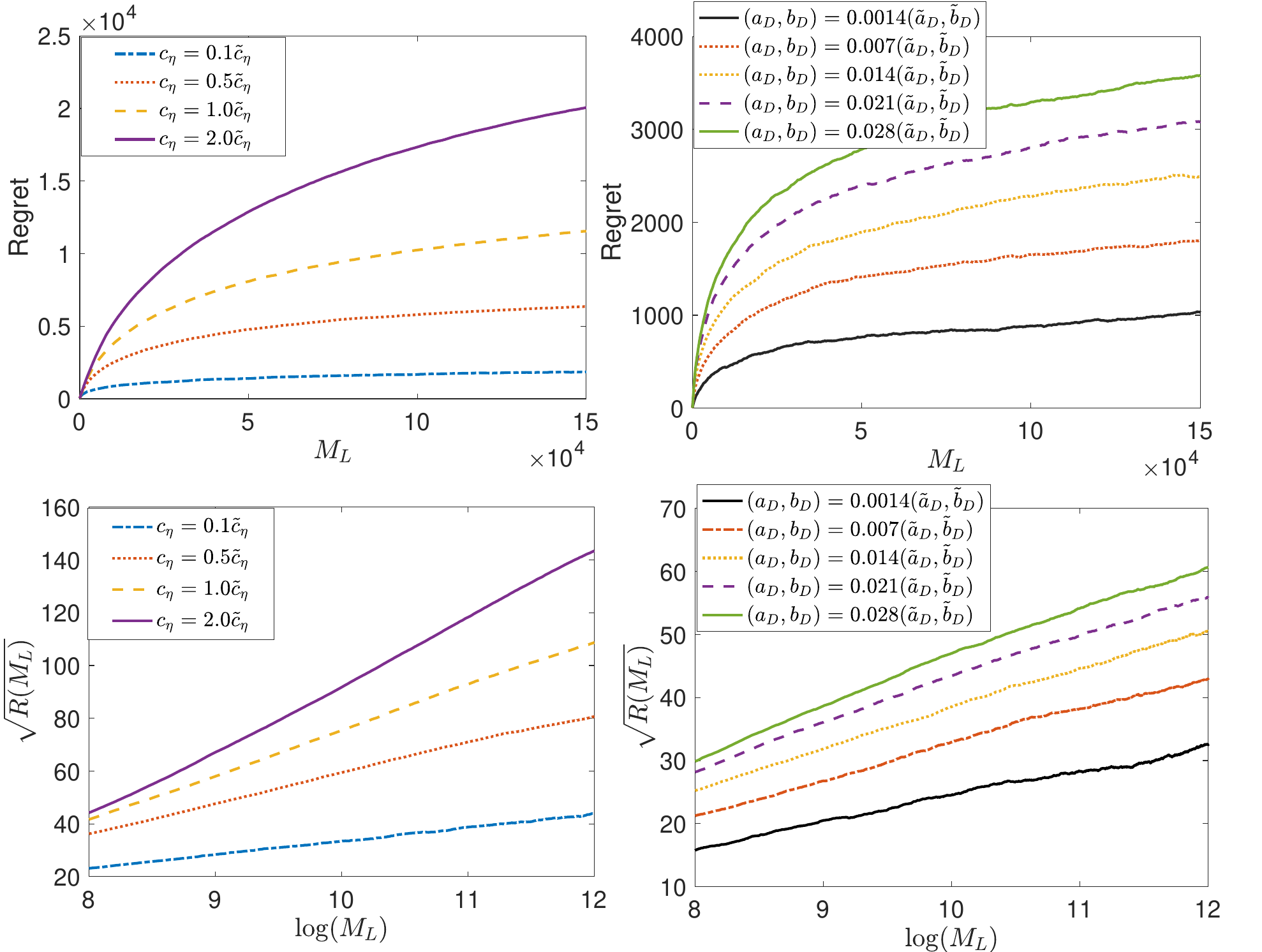}
	\caption{{Simulated regret curves with relaxed bounds on (i) step size $\eta_k$ robustness (top left panel), (ii) cycle length $D_k$ (top right panel),  and their logarithmic scales in sample size $M_L$ (two bottom panels). All regret curves are estimated by averaging 500  independent simulation runs.} }
	\label{fig:para robustness}
	\vspace{-0.2in}
\end{figure}

\subsection{Robustness to the Theoretical Bounds}
Recall that the theoretical bounds in \eqref{par:eta} and \eqref{par:D} 
require that $a_D\geq \tilde{a}_D$, $b_D\geq \tilde{b}_D$ and $c_\eta\geq \tilde{c}_\eta$. We hereby test the criticality of these lower bounds $\tilde{a}_D$, $\tilde{b}_D$ and $\tilde{c}_\eta$ by implementing GOLiQ with $(a_D,b_D,c_\eta)<(\tilde{a}_D,\tilde{b}_D,\tilde{c}_\eta)$. Specifically, in our first experiments, we consider $c_\eta=\{2,1,0.5,0.1\}\tilde{c}_\eta$ for the step-size $\eta_k$, with $D_k= 10+10\log(k)$ (see left-hand panels of Figure \ref{fig:para robustness}); in our second experiment, we consider 
$(a_D,b_D) = \{0.028,0.021,0.014,0.007,0.0014\}(\tilde{a}_D, \tilde{b}_D)$ for the sample-size $D_k$, with $\eta_k = 3/k$ (see right-hand panels of Figure \ref{fig:para robustness}). In both experiments, we plot the average regret curves estimated by 500 independent runs. 

Figure \ref{fig:para robustness} reveals that GOLiQ continues to perform effectively even when the hyperparameters are chosen to be much smaller than their corresponding theoretical lower bounds. For $\eta_k$,  our algorithm generates a logarithm regret even when $c_\eta=0.1\tilde{c}_\eta$. (However, we discover that GOLiQ will fail to converge and yield a linear regret if we keep reducing $c_\eta$ (e.g., to $0.01\tilde{c}_\eta$). 
For $D_k$, all regret curves exhibit a logarithmic order (even when $(a_D,b_D)=0.0014(\tilde{a}_D,\tilde{b}_D)$). In summary, our numerical experiments show that the theoretical bounds for our hyperparameters do  not seem to be too restrictive. {In addition, the experiment in Section \ref{sec:AsymOpt} serves as another piece of evidence supporting the  robustness of GOLiQ. In Section \ref{sec:AsymOpt}, we apply GOLiQ with the same hyperparameters $\eta_k=5k^{-1}$ and $D_k=10+10\log(k)$ for different settings with various $c$, $c_s^2$ and $n$ (see Figure \ref{fig:Compare_Amy}), and GOLiQ exhibits stable performance with similar logarithm regrets.}
Of course, we acknowledge that the specific selection of these hyperparameters in a practical setting will require further  tuning in order to make the most efficient use of GOLiQ.
\begin{remark}[Requirement of information: online learning vs. heavy traffic]
	We provide our view on how online learning relies on the system information, and we treat heavy-traffic methods as a benchmark. 
	First, online learning in general requires less prior information of the distributions than heavy-traffic methods do. For example, to solve the problem in the present study, the diffusion limit in \cite{LeeWard2014} requires the knowledge of the exact values of the second moments of arrival and service times. On the other hand, even though the efficiency of GOLiQ is subject to constraints in terms of certain model parameters, the bounds of these constraints may be relaxed without needing to sacrifice much of the algorithm's performance.
	Second, the required information (e.g., moments) serves as crucial input parameters for the heavy-traffic models, whereas the design and implementation of online learning algorithms do not immediately require the aforementioned information (even though it is still relevant to the tuning of hyperparameters). 
	{All that we require is that the constants in \eqref{par:eta} and \eqref{par:D} are \textbf{not too small}. So as long as we follow the structure specified in \eqref{par:eta}--\eqref{par:D}, it will not be too difficult to find reasonably sound hyperparameters (e.g., by a trial-and-error search) even without precise information of parameters $\eta$ and $\gamma$ as in Assumption 2.  However, trial-and-error will be ineffective for heavy-traffic methods because precise information is needed (e.g., $\sigma^2$). In this sense, online learning depends on the system information to a lesser extent.}
	\label{rem:info}
\end{remark}
\section{Comparison With Online Learning Algorithm in \cite{Huh2009}} \label{sec: cmp_Tim}
In order to highlight the novelty in the regret analysis of this work, we wish to provide a comparison to other existing online learning algorithms developed for queueing systems. Unfortunately, there exists no previous algorithm that aims to solve a similar (not to mention the same) problem as in this paper.

\cite{Huh2009} develops an online learning algorithm with the objective of finding the optimal base-stock policy for an inventory system with a non-zero replenishment lead time. At a glance, \cite{Huh2009} does not seem to be relevant to the present paper at all. Indeed, results in \cite{Huh2009} are by no means directly comparable to GOLiQ, because the two articles consider two different systems. Nevertheless, the fundamental idea in the regret analysis by \cite{Huh2009} may be used as a basis to devise a queueing-version algorithm.To understand why this is possible, we first discuss the similarities of our method and that in \cite{Huh2009}; and we next explain their major distinctions.

\paragraph{\textbf{Similarities.}}
First, \cite{Huh2009} analyzes the transient regret bound of an inventory system operated under a stationary base-stock policy, of which the main framework is analogous to that in the present work.
Second, the heart of the online learning algorithm in \cite{Huh2009} is an SGD method. Last, the regret in \cite{Huh2009} is also defined using the steady-state performance as the benchmark.

\paragraph{\textbf{Distinctions.}} 
Nevertheless, we stress that GOLiQ is not a quick extension of the regret analysis in \cite{Huh2009}. 
First, a queueing model, by its very nature, has completely different dynamics, problem structure, and research questions from inventory systems. For example, the state space of the queueing model here is unbounded,  while the inventory system  in \cite{Huh2009}  is bounded. (The unbounded state space has made the analysis on transient error and gradient variance more challenging.) Next, our analysis of the ``regret of nonstationarity" is a novelty; when establishing our regret bound, we examine more delicately the transient error at the beginning of the cycles, so as to render a smaller regret bound $O(\log(L)^2)$ than the linear bound $O(L)$ as in \cite{Huh2009}. Also see Section \ref{subsec: transient analysis} for additional discussions. In detail, different from the analysis  in \cite{Huh2009}, we further separate the transient error in each cycle into two parts, i.e., the `warm-up' part and `near-stationary' part, and deal with them using different coupling techniques: coupling from the previous cycle and coupling from infinite past for the `warm-up' part and synchronous coupling in the same cycle  for the `near-stationary' part. 
Last, our novel theoretical analysis yields different ``optimal" structure for the hyperparameters $\eta_k=O(k^{-1})$ and $D_k=O(\log{(k)})$. 

{According to their regret analysis, \cite{Huh2009} propose to choose the hyperparameters 
	$\eta_k = O(k^{-1/2})$ and $D_k = O(\sqrt{k})$ which yield a regret bound in the order $O(M_L^{2/3})$. However, we point out that the objective function in \cite{Huh2009} is convex while GOLiQ in the present paper is designed assuming a strongly convex objective function (Assumption \ref{assmpt: convexity}). Therefore, to make a fair comparison between GOLiQ and the online learning algorithm proposed in \cite{Huh2009}, we need to redo the regret analysis in \cite{Huh2009} under the \textit{strong convexity}. This change, as we will show below, will yield a different set of hyperparameters.} 

{Suppose we select $D_k=O(k^{\alpha})$ and $\eta_k=O(k^{-\beta})$. Then, following Lemma 11 of \cite{Huh2009},  $R_1(L)=O(L)$ (compared to $R_1(L)=o(L)$ in our analysis). Given that the objective function is strongly convex, Theorem \ref{thm: broadiezeevi} yields that $R_2(L)=O\left(\sum_{k=1}^{L}k^{\alpha}\cdot k^{-\beta}\right)=O\left(L^{\alpha-\beta+1}\right)$ for $\beta\in(0,1]$. As a result, the overall regret is
$$R(L)=R_1(L)+R_2(L)=O(L^{(\alpha-\beta+1)\vee 1}),\quad\text{and}\quad R(M_L)=O\left(M_L^{\frac{(\alpha-\beta+1)\vee1}{\alpha+1}}\right),$$
with $M_L=O(L^{\alpha+1})$.
Consequently, the order of $R(M_L)$ is minimized by setting
\begin{align}\label{HuhHyperPar}
\eta_k=O(k^{-1}),\quad  \text{and}\quad  D_k=O(k),
\end{align}
which yields an improved regret bound $O(M_L^{\frac{1}{2}})$ (as opposed to the previous regret $O(M_L^{\frac{2}{3}})$ under regular convexity).}

{We refer to Algorithm \ref{alg: direct} with $\eta_k$ and $D_k$ selected according to \eqref{HuhHyperPar} as GOLiQ-H.
To compare GOLiQ with GOLiQ-H, we follow the setting in Section \ref{sec:TwoDim} by considering an $M/M/1$ queue, with $c(\mu)=0.1\mu^2$ and $\lambda(p)=10\lambda_0(p).$ }

{
In Figure \ref{fig:cmp_Tim}, we plot the average regret curves (estimated by averaging 500 independent paths) for both GOLiQ and GOLiQ-H. 
The hyperparameters  are $\eta_k=2k^{-1}$ and  $D_k=10+10\log(k)$  for GOLiQ, and $\eta_k=\{2,4\}k^{-1}$ and $D_k=10+k$ for GOLiQ-H. 
}
{Unsurprisingly, Figure \ref{fig:cmp_Tim} confirms that GOLiQ is more effective than GOLiQ-H. This is consistent with our theoretical analysis because GOLiQ yields a logarithmic regret while GOLi-H has a regret bound of $O(M_L^{\frac{1}{2}})$.}
\begin{figure}[H]
\centering 
\includegraphics[width=1\linewidth]{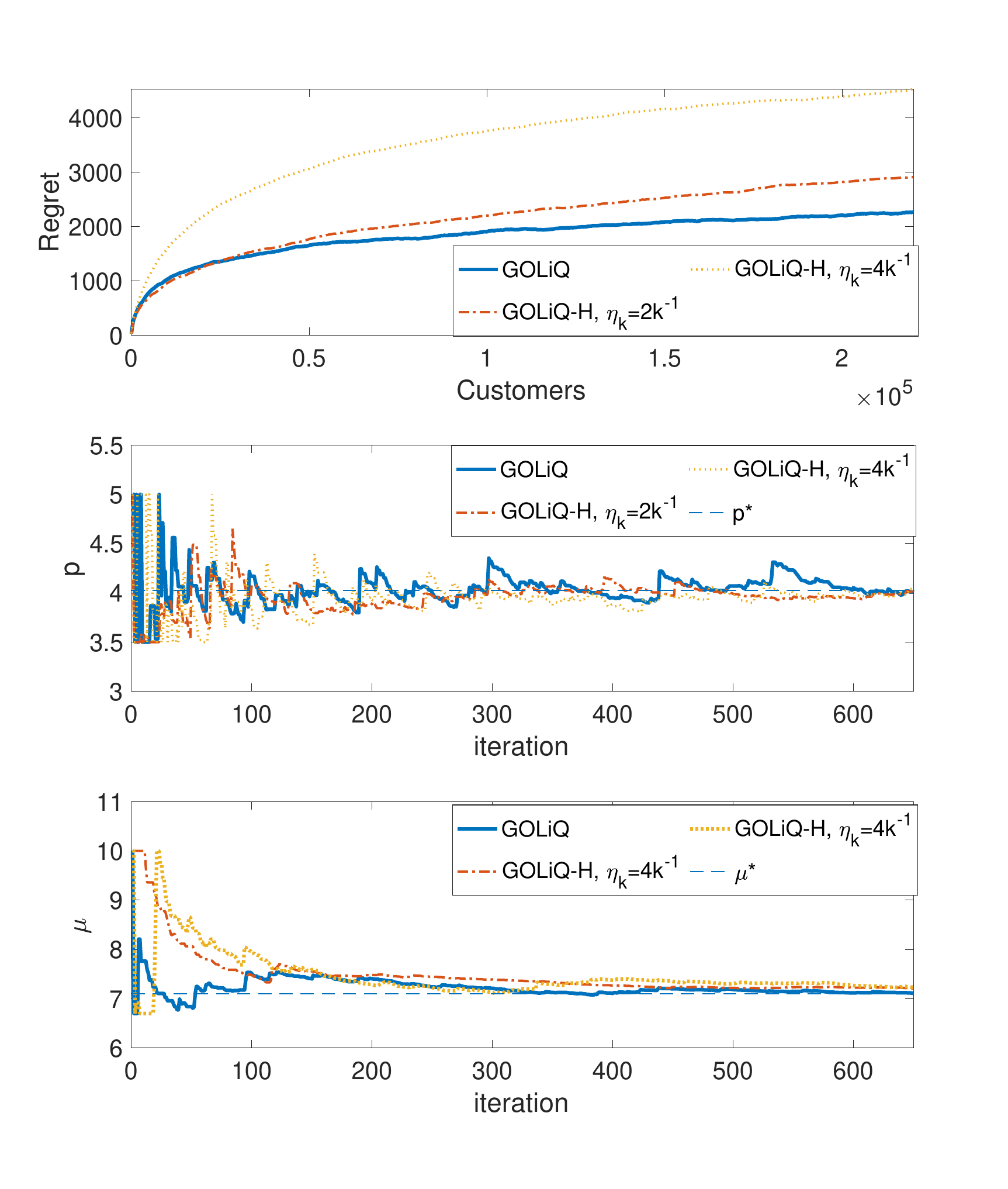}
\vspace{-0.5in}
\caption{{Comparing GOLiQ and GOLiQ-H: (i) regret curves (top panel), (ii) trajectory of price $p_k$ (middle panel), and (iii) trajectory of service rate $\mu_k$ (bottom panel). Hyperparameters are $\eta_k=\{2,4\}k^{-1}$ and $D_k=10+k$ for GOLiQ-H, and are $\eta_k=2k^{-1}$ and $D_k=10+10\log(k)$ for GOLiQ. All regret curves are estimated by averaging 500 independent simulation runs.} 
}
\label{fig:cmp_Tim}
\end{figure}

\section{Additional Numerical Examples}\label{sec: AddNum}
In this section we conduct additional numerical experiments to confirm the practical effectiveness of our algorithm. In what follows, we first test the case where the uniform stability condition is relaxed; we next report the algorithm performance for $GI/GI/1$ queueing models with phase-type and lognormal distributions.

\subsection{Violation of Uniform Stability}\label{sec:UnifStab}
\begin{figure}[H]
\vspace{-0.1in}
\centering
\includegraphics[width=0.9\textwidth]{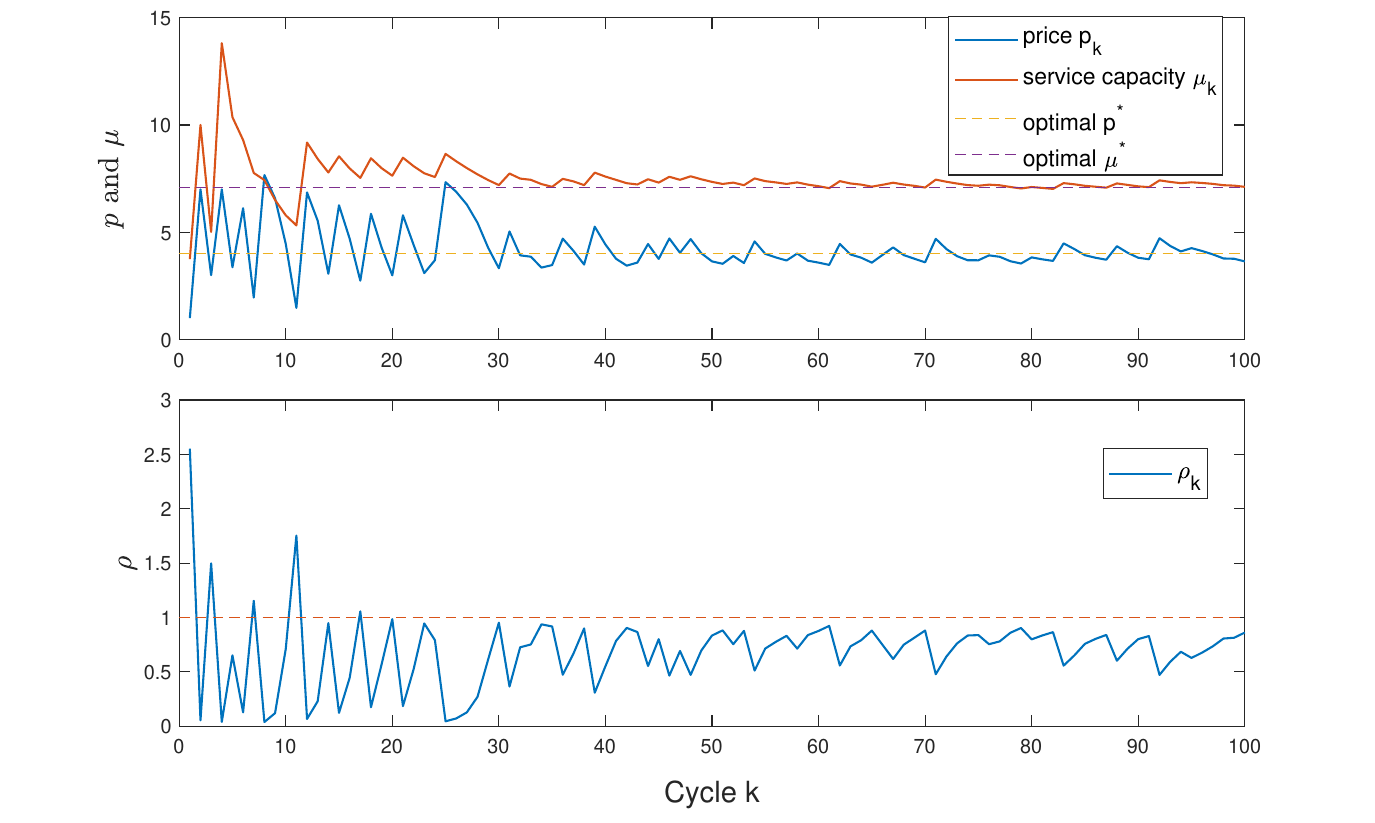}
\vspace{-0.1in}
\caption{Joint pricing and staffing for the $M/M/1$ model in Section \ref{sec:TwoDim} without uniform stability.}\label{fig:uniformstab}
\end{figure}
We extend the $M/M/1$ example considered in Section \ref{sec:TwoDim} with the uniform stability condition relaxed. Specifically, we begin with an initial setting of $(p_0, \mu_0)$ such that $\rho_0\equiv \lambda(p_0)/\mu_0 = 2.55$, which violates the stability condition. As shown in Figure \ref{fig:uniformstab}, the pricing and staffing policies $(p_k,\mu_k)$ remain convergent to $(p^*, \mu^*)$. Consistently, the resulting traffic intensity $\rho_k \equiv \lambda(p_k) / \mu_k$ is quickly controlled to fall below 1; that is, the workload is kept in check despite of the unstable performance in the initial cycle.

\subsection{$M/G/1$ with Phase-Type Service} \label{sec:NumMG1}
To test the performance of our online learning algorithm for queues with non-exponential service times, we consider phase-type distributions: hyperexponential with $n$ phases ($H_n$) and Erlang with $n$ phases ($E_n$). 
In Figure \ref{fig:2dMG1} we report the convergent sequence $(p_k,\mu_k)$ with $H_2$ service with service-time SCV $c_s^2 = 8$ (top panel), $M$ service with $c_s^2=1$ (middle panel), and $E_8$ service with $c_s^2 = 1/8$ (bottom panel). Other parameters include the step size $\eta_k=4/k$, cycle length $D_k=20+10\log(k)$ and initial condition $p_0=4$ and $\mu_0=12$ ($\lambda_0=5.249$).

\begin{figure}[H]
\vspace{-0.1in}
\centering
\includegraphics[width=0.85\textwidth]{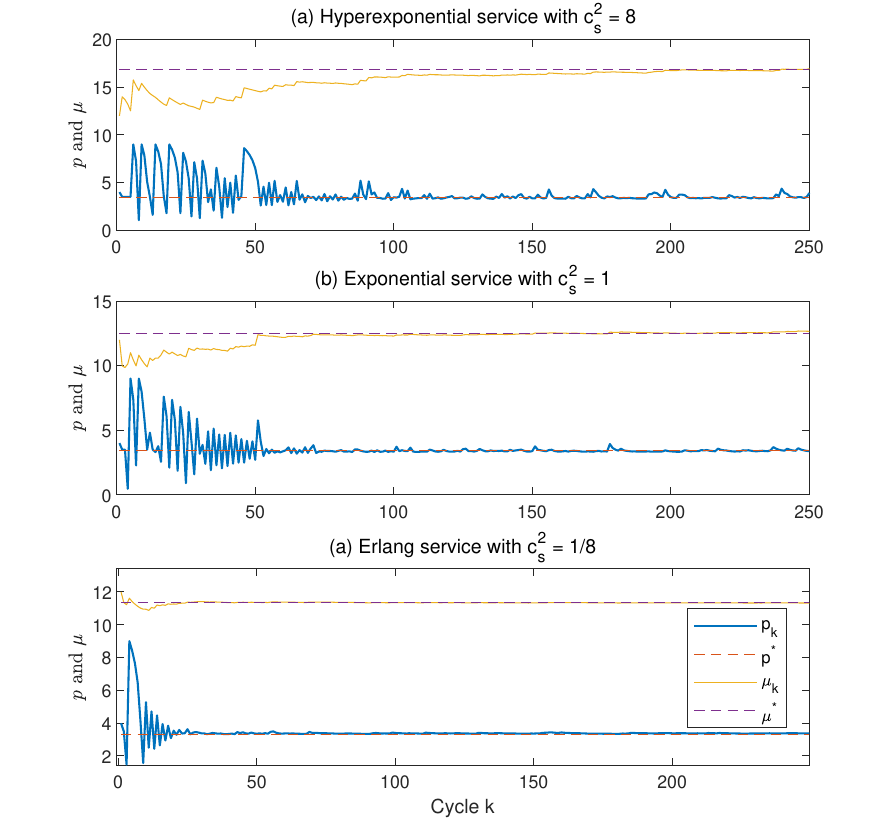}
\vspace{-0.1in}
\caption{Joint pricing and staffing for an $M/G/1$ queue having (a) $H_2$ service with $c_s^2 = 8$ (top panel), (b) $M$ service (middle panel), and (c) Erlang service with $c_s^2 = 1/8$ (bottom panel). Other parameters are step length $\eta_k=4/k$, cycle length $D_k=20+10\log(k)$, initial condition $p_0=4$, $\mu_0=12$. The optimal pricing and staffing solutions are: (i) $(p^*, \mu^*)=(3.44, 16.86)$; (ii) $(p^*, \mu^*) = 3.40, 12.48)$; (iii) $(p^*,\mu^*)=3.38, 11.34)$.}\label{fig:2dMG1}
\end{figure}

Figure \ref{fig:2dMG1} confirms that our algorithm remains effective. In addition, the convergence is faster as the CSV $c_s^2$ decreases. This is intuitive because a less variable service-time distribution yields a smaller $\mathcal{V}_k$ for the gradient estimator.


\subsection{Lognormal Service and Arrival} \label{sec:NumGG1}
Next, we consider an $LN/LN/1$ queue with service and interarrival times following lognormal (LN) distributions. Our consideration here follows from the recent empirical confirmations of LN distributed service times in real service systems. 

We let $c_s^2=c_a^2=2$ with $c_a^2$ being the SCV of the $LN$-distributed interarrival times. The other parameters remain the same as in Section \ref{sec:NumMG1}.
Because the exact optimal solutions $(p^*, \mu^*)$ are unavailable for this model, we are unable to provide an estimate of the regret as done in Figure \ref{fig2dMM1}, nor can we confirm the convex structure of the problem. Nevertheless, Figure \ref{fig:LNLN1} shows that our online algorithm continues to work well, despite the fact that LN is no longer a light-tail distribution (Assumption \ref{assmpt: light tail} does not hold in this case).
\begin{figure}[H]
\vspace{-0.05in}
\centering
\includegraphics[width=0.80\textwidth]{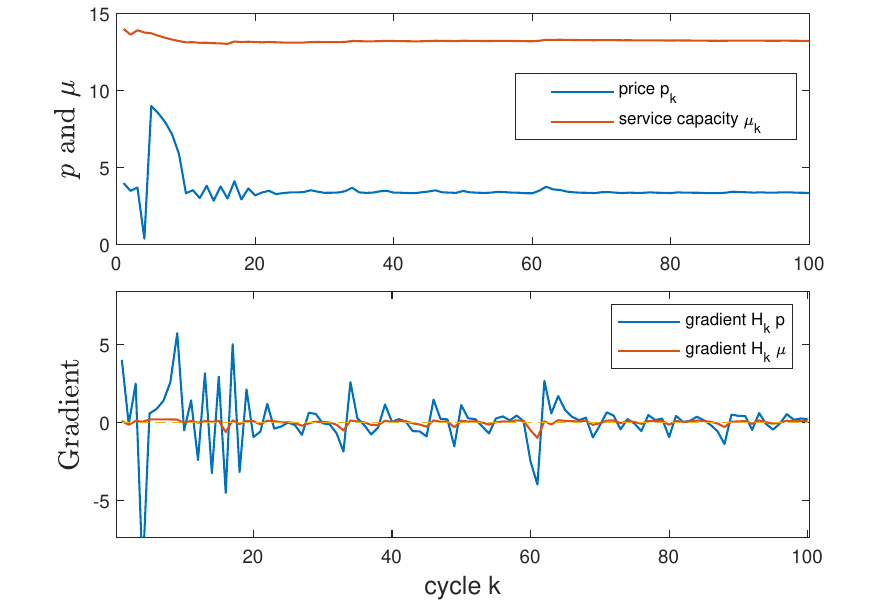}
\vspace{-0.05in}
\caption{Joint online pricing and staffing for an $LN/LN/1$ having lognormal service and interarrival times with CSVs $c_s^2 = c_a^2 = 2$. Other parameters are  $\eta_k=4/k$, $D_k=20+10\log(k)$, $p_0=4,\mu_0=14$.}\label{fig:LNLN1}
\end{figure}


\subsection{Extended Comparison of GOLiQ and  Heavy-traffic Methods}\label{sec:AsymOptMore}
Supplementing our investigations in Section \ref{sec:AsymOpt}, we provide additional numerical results. 
Recall that the heavy-traffic results in \cite{LeeWard2014} are obtained by constructing a sequence of $GI/GI/1$ models indexed by $n$, where the $n^{\rm th}$ model has scaled arrival rate $\lambda_n(p)=n\lambda(p)$ and service rate $\mu_n=n\mu$, so that both $\lambda_n$ and $\mu_n$ grow to $\infty$ as $n$ increases.  
\cite{LeeWard2014} develop asymptotic staffing and pricing solutions for the $GI/GI/1$ queue; they show that, as the scaling factor $n\rightarrow\infty$, the optimal price $p^*_n\rightarrow p_{\infty}$ and service capacity $\mu^*_n/n\rightarrow \mu_{\infty}$, with $\rho_\infty\equiv \lambda(p_{\infty})/\mu_{\infty}=1$.

We repeat our experiment in Section \ref{sec:TwoDim} with the scaling parameter $n \in\{10,50, 100, 500,1000,2000\}$ for the arrival rate function \eqref{logisticLmd}. But we now focus on the optimal traffic intensity as $n$ varies.
In Figure \ref{fig:Compare with Lee and Ward} we plot the optimal price and service rate as $n$ increases. In each experiment, we compute the optimal $p_n$ and $\mu_n$ using their average value in cycles 300--500 of Algorithm \ref{alg: direct}. Consistent with \cite{LeeWard2014}, Figure \ref{fig:Compare with Lee and Ward} shows that $p_n$, $\mu_n/n$ and $\rho_n$ approach $p_{\infty}$, $\mu_{\infty}$ and $\rho_{\infty}=1$. On the other hand, when the scale $n$ is not very large, the heavy-traffic solutions can become inaccurate. For instance, when $n = 100$ the optimal traffic $\rho_{100}$ is around 0.8, which is not close to 1.

\begin{figure}[H]
\centering
\includegraphics[width=0.85\textwidth]{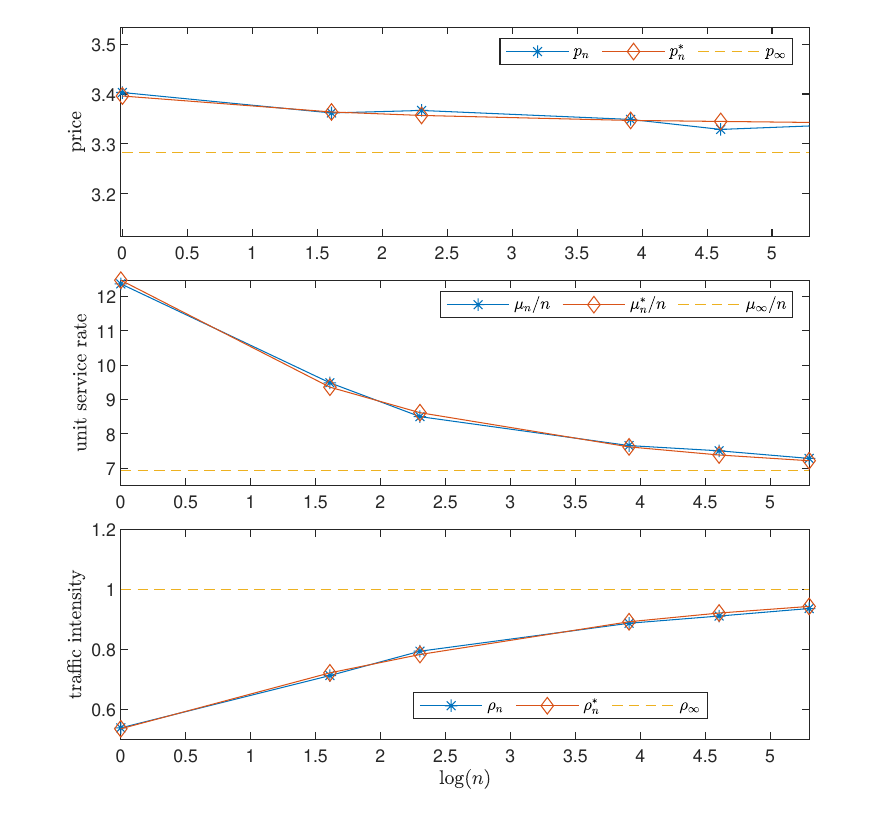}
\caption{Comparison of (i) online optimization solutions $(p_n,\mu_n,\rho_n)$; (ii) exact solutions $(p_n^*, \mu_n^*, \rho_n^*)$; and (iii)  heavy-traffic solutions $(p_{\infty}, \mu_{\infty}, \rho_{\infty}) = (3.282, 6.932, 1)$ in \cite{LeeWard2014}, as the system scale $n= M_0$ increases, with parameters $D_k=n(10+10\log(k))$ and $\eta_k=3k^{-1}$. }\label{fig:Compare with Lee and Ward}
\end{figure}

\subsection{Alternative Definition of Regret}\label{sec:RegretAlternative}
In this subsection, we attempt to rationalize our regret definition in \eqref{regret_def}. We consider a potential alternative to \eqref{regret_def} which 
benchmarks the system revenue under GOLiQ with the nonstationary revenue under $(\mu^*, p^*)$. Because the nonstationary queue length is intractable, we conduct additional numerical experiments to estimate the expected nonstationary regret via Monte-Carlo simulations.

Specifically, we simulate the regret in \eqref{eq:regretk} under $(p^*,\mu^*)$ with the queueing system starting empty (of which the dynamics is nonstationary). 
We use the $M/M/1$ model in Section \ref{sec: num} having a logit demand function \eqref{logisticLmd} with $n=10$ and a quadratic staffing cost in \eqref{convexMu} with $c=0.1$.

In Figure \ref{fig:Reg_Cmp} we graph both versions of the regret under GOLiQ under the same experimental setting, with hyperparameters $\eta_k=3k^{-1}$ and $D_k=10+10\log(k)$.  
Figure \ref{fig:Reg_Cmp} confirms that the these two versions of regret appear to be nearly indistinguishable. This is due to the geometric ergodicity of $G/G/1$ queue.
\begin{figure}[H]
\centering
\includegraphics[width=0.85\linewidth]{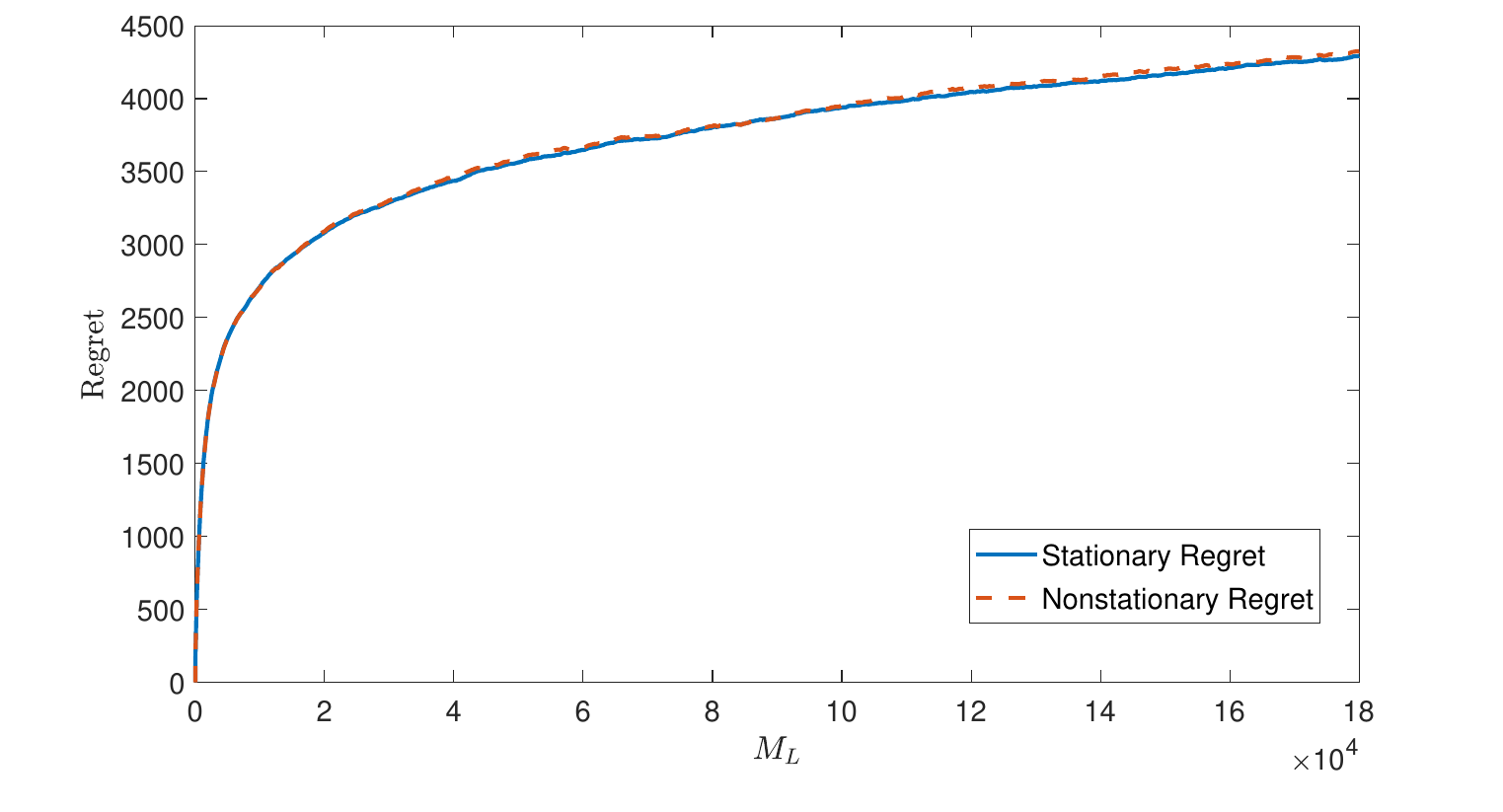}
\vspace{-0.1in}
\caption{Comparing two versions of regret: (i) ``stationary regret" that benchmarks with the steady-state performance under $(\mu^*,p^*)$ (defined in \eqref{regret_def}), and (ii) ``nonstationary regret" that benchmarks with the nonstationary performance under $(\mu^*, p^*)$. Hyperparameters are $\eta_k=3k^{-1}$, $D_k=10+10\log(k)$. Both regret curves are estimated by 1,000 independent runs.}
\label{fig:Reg_Cmp}
\end{figure}
{
\section{Additional discussion on Assumption \ref{assmpt: convexity}}
\label{appx: demand function}
In this section, we provide some sufficient conditions for strong convexity in the $M/GI/1$ case.  
\begin{lemma}\label{lmm: MG1 convex}
	For $M/GI/1$ queues, if $c(\mu)$ is convex,  $\lambda(\underline{p})/\underline{\mu}<1$, and in addition,
	\begin{equation}
		\frac{\lambda'(p)^2}{2\lambda(p)}<\lambda''(p)<\frac{-2\lambda'(p)}{p},\quad \text{and}\quad \frac{\lambda(\bar{p})}{\bar{\mu}}>1-1/\sqrt{2}\approx0.29,
		\label{eq: MG1 sufficient convex}
	\end{equation}
	then $f(\mu,p)$ is strongly convex in $\mathcal{B}$.
\end{lemma}
\begin{proof}[Proof of Lemma \ref{lmm: MG1 convex}]
	Recall that $f(\mu,p)=-p\lambda(p)+h_0\EE[Q_\infty(\mu,p)]+c(\mu)$, and $(-p\lambda(p))''=-p\lambda''-2\lambda'$.
	Under condition \eqref{eq: MG1 sufficient convex}, we have $-p\lambda''-2\lambda'>0$. Therefore, both $-p\lambda(p)$ and $c(\mu)$ are convex, and it suffices to show that $\EE[Q_\infty(\mu,p)]$ is strongly convex in $\mu$ and $p$. For $M/GI/1$ queues, Pollaczek-Khinchine formula yields that 
	$$q(\mu,p)\equiv\EE[Q_\infty(\mu,p)]=C\frac{\lambda(p)}{\mu-\lambda(p)}+(1-C)\frac{\lambda(p)}{\mu},$$
	with $C\equiv \frac{1+c_s^2}{2}$. 
	For any given pair of $(\mu,p)$, let $H_q$ be the Hessian matrix of $q(\mu,p)$. We next verify that $H_q$ is positively definite. By direct calculation, we have
	\begin{align*}
		\partial_{p}^2 q&=C\frac{\mu}{(\mu-\lambda)^3}\left(2(\lambda')^2+(\mu-\lambda)\lambda''\right)+(1-C)\frac{\lambda''}{\mu},\\
		\partial_{\mu}^2 q&=C\frac{2\lambda}{(\mu-\lambda)^3}+(1-C)\frac{2\lambda}{\mu^3},\quad \partial_{p}\partial_{\mu}q=-C\frac{\lambda+\mu}{(\mu-\lambda)^3}\lambda'-(1-C)\frac{\lambda'}{\mu^2},
	\end{align*}
	with $\lambda',\lambda''$ being the first and second order derivatives of $\lambda(p)$.
	As a result, the determinant of Hessian matrix of $H_q$ is 
	\begin{align*}
		|H_q|=&\frac{C^2}{(\mu-\lambda)^5}\left(2\mu\lambda\lambda''-(\mu-\lambda)(\lambda')^2\right)+\frac{(1-C)^2}{\mu^4}\left(2\lambda\lambda''-(\lambda')^2\right)\\
		&\qquad+\frac{2C(1-C)}{\mu^2(\mu-\lambda)^3}\left((2\mu-\lambda)\lambda\lambda''-(\mu-\lambda)(\lambda')^2\right).
	\end{align*}
	To show that $H_q$ is positively definite, it suffices to show that $\partial_\mu^2q$, $\partial_p^2q $ and $|H_q|$ are all positive. First,
	it is clear that 
	$$\partial_{\mu}^2 q=2\lambda C\left(\frac{1}{(\mu-\lambda)^3}-\frac{1}{\mu^3}\right)+\frac{2\lambda}{\mu^3}>0.$$
	Next, we compute
	\begin{align*}
		\partial_{p}^2 q &=C\frac{\mu}{(\mu-\lambda)^3}\left(2(\lambda')^2+(\mu-\lambda)\lambda''\right)+(1-C)\frac{\lambda''}{\mu}\\
		&=\frac{2C\mu}{(\mu-\lambda)^3}(\lambda')^2+\frac{C\mu}{(\mu-\lambda)^2}\lambda''+(1-C)\frac{\lambda''}{\mu}\\
		&\stackrel{(a)}{>}\frac{C\mu}{(\mu-\lambda)^2}\lambda''+(1-C)\frac{\lambda''}{\mu}\\
		&=C\lambda''\Big(\underbrace{\frac{\mu}{\mu-\lambda}}_{>1}\cdot\frac{1}{\mu-\lambda}-\frac{1}{\mu}\Big)+\frac{\lambda''}{\mu}\stackrel{(b)}{>}C\lambda''\left(\frac{1}{\mu-\lambda}-\frac{1}{\mu}\right)+\frac{\lambda''}{\mu}>0.
	\end{align*}
	Here, inequality (a) follows from that $\frac{2C\mu}{(\mu-\lambda)^3}(\lambda')^2>0$.  Inequality (b) holds  due to the facts that $\frac{\mu}{\mu-\lambda}>1$ and that $\lambda''>\frac{(\lambda')^2}{2\lambda}\geq0$. The last inequality holds because $\frac{1}{\mu-\lambda}>\frac{1}{\mu}$. As a result, we have $\partial^2_{p} q, \partial^2_{\mu} q>0$.
	Next, we verify that $|H_q|>0$. Because $2\lambda\lambda''-(\lambda')^2>0$, we have
	$$2\mu\lambda\lambda''>(\mu-\lambda)(\lambda')^2,\quad \text{and}\quad(2\mu-\lambda)\lambda\lambda''>(\mu-\lambda)(\lambda')^2.$$
	Therefore,
	\begin{align*}
		|H_q|=~&\frac{C^2}{(\mu-\lambda)^5}\left(2\mu\lambda\lambda''-(\mu-\lambda)(\lambda')^2\right)+\frac{(1-C)^2}{\mu^4}\left(2\lambda\lambda''-(\lambda')^2\right)\\
		&\qquad+\frac{2C(1-C)}{\mu^2(\mu-\lambda)^3}\left((2\mu-\lambda)\lambda\lambda''-(\mu-\lambda)(\lambda')^2\right)\\
		\stackrel{(c)}{>}&\frac{C^2}{(\mu-\lambda)^5}\left(2\mu\lambda\lambda''-(\mu-\lambda)(\lambda')^2\right)+\frac{2C(1-C)}{\mu^2(\mu-\lambda)^3}\left((2\mu-\lambda)\lambda\lambda''-(\mu-\lambda)(\lambda')^2\right)\\
		>~&\frac{C^2}{(\mu-\lambda)^5}\Big(\underbrace{2\mu\lambda\lambda''}_{>(2\mu-\lambda)\lambda\lambda''}-(\mu-\lambda)(\lambda')^2\Big)-\frac{2C^2}{\mu^2(\mu-\lambda)^3}\left((2\mu-\lambda)\lambda\lambda''-(\mu-\lambda)(\lambda')^2\right)\\
		>~&\frac{C^2}{(\mu-\lambda)^5}\left((2\mu-\lambda)\lambda\lambda''-(\mu-\lambda)(\lambda')^2\right)-\frac{2C^2}{\mu^2(\mu-\lambda)^3}\left((2\mu-\lambda)\lambda\lambda''-(\mu-\lambda)(\lambda')^2\right)\\
		=~&\frac{C^2\left((2\mu-\lambda)\lambda\lambda''-(\mu-\lambda)(\lambda')^2\right)}{(\mu-\lambda)^3}\left(\frac{1}{(\mu-\lambda)^2}-\frac{2}{\mu^2}\right)>0.
	\end{align*}
	Inequality (c) holds because $\frac{(1-C)^2}{\mu^4}(2\lambda\lambda''-(\lambda')^2)$ is positive by \eqref{eq: MG1 sufficient convex}. The last inequality follows from $\frac{\lambda(p)}{\mu}>1-1/\sqrt{2}$. This closes our proof.
\end{proof}
}

\newpage

\begin{table}
\centering
\vspace{-0.2 in}
\tiny{
\begin{tabular}{cll}
	
	\hline 
	& Notation                                                              & Description                                                                    \\
	\hline
	\multicolumn{1}{c}{\multirow[t]{17}{*}{\begin{tabular}[c]{@{}c@{}}\\ \\ \\ \\Model\\ parameters \\ and \\functions\end{tabular}}} & $\mathcal{B}=[\underline{p},\bar{p}]\times[\underline{\mu},\bar{\mu}]$ & Feasible action space                                               \\
	\multicolumn{1}{c}{}                                                                                            & $c(\mu)$                                                               & Staffing cost function                                                                  \\
	\multicolumn{1}{c}{}                                                                                            & $f(\mu,p)$                                                             & Objective (loss) function                                                      \\
	\multicolumn{1}{c}{}                                                                                            & $h_0$                                                                  & Customer holding cost                                                        \\
	\multicolumn{1}{c}{}                                                                                            & $\lambda(p)$                                                           & Demand function                                                                \\
	\multicolumn{1}{c}{}                                                                                            & $\mu$                                                                  & Service rate/capacity                                                                   \\
	\multicolumn{1}{c}{}                                                                                            & $n$                                                                  & Market size/ System scale in Section \ref{sec: num}                                                                 \\
	\multicolumn{1}{c}{}                                                                                            & $p$                                                                    & Service fee                                                                          \\
	\multicolumn{1}{c}{}                                                                                            & $Q_\infty(\mu,p)$                                                      & Stationary queue length under $(p,\mu)$                                        \\
	\multicolumn{1}{c}{}                                                                                            & $S_n^k$                                                                  & Service time of the $(n-1)^{\rm th}$ customer in cycle $k$                                           \\
	\multicolumn{1}{c}{}                                                                                            & $\tau_n$                                                               & Interarrival time between $(n-1)^{\rm th}$ and $n^{\rm th}$ customers                        \\
	\multicolumn{1}{c}{}                                                                                            & $\theta,\gamma,\eta$                                                   & Parameters of light-tail assumptions (Assumption 2)                                        \\
	\multicolumn{1}{c}{}                                                                                            & $U_n,V_n$                                                              & Unscaled random ``seeds" of interarrival and service times                    \\
	\multicolumn{1}{c}{}                                                                                            & $x^*=(p^*,\mu^*)$                                                      & Optimal decision fee and service rate                                                     \\
	\hline
	\multirow[t]{7}{*}{\begin{tabular}[c]{@{}l@{}}\\ \\ \\Algorithmic\\ hyperparameters \end{tabular}}
	& $D_k$                                                                  & Sample size (number of customers served) in cycle $k$         \\
	& $\eta_k$                                                               & Step size or learning rate in cycle $k$                                         \\
	& $H_k$                                                                  & Gradient estimator in cycle $k$                                                \\
	& $M_L$                                                                  & Cumulative number of customers served by cycle $L$                             \\
	& $Q_k$                                                                  & Queue content leftover from cycle $k-1$                   \\
	& $W_n^k$                                                                & Delay of the $n^{\rm th}$ customer in cycle $k$                                          \\
	& $\xi$                                                                  & Warm up rate                                                                   \\
	& $X_n^k$ ($X_n$)                                                        & Server's busy time observed by customer $n$ in cycle $k$                    \\
	\hline
	\multirow[t]{39}{*}{\begin{tabular}[c]{@{}l@{}}\\ \\ \\Constants and \\ bounds in \\regret analysis\end{tabular}}          
	& $a_D,b_D$                                                              & Constants for $D_k$ in equation \eqref{par:D}                                               \\
	& $A=4\sqrt{M}+4M$                                                       & Constant in Corollary 1                                                        \\
	& $B$                                                                    & Constant of stationary waiting times in Lemma 4                     \\
	& $B_k,~\mathcal{V}_k$                                                             & Upper bounds for bias and Variance of $H_k$                                         \\
	& $c_\eta$                                                               & Constant for $\eta_k$ in equation \eqref{par:eta}                                                    \\
	& $c_\mu,c_\lambda$												        &Constants in Lemma 3\\
	& $C=\max\{\|x_0-x^*\|^2,8K_3/K_0\}$                                   & Constant in Theorem 2 \\
	& $C_0=\max _{x\in\mathcal{B}}\{h_0\lambda'(p),h_0\lambda(p)/\mu\}$      & Constant in the proof of Theorem 3 \\
	& $C_1=\max_{x\in\mathcal{B}}\{|\lambda(p)+p\lambda'(p)|,|c'(\mu)|\}$    & Constant in the proof of Theorem 3 \\
	& $C_D$																						&Constant for the selection of $D_k$ (Theorem 3)\\
	& $d_k=\lceil 4\log(k)/\min(\theta,\gamma)\rceil$                        & Constant of warm-up time (Theorem 1)                                                  \\
	& $\tilde{d}_k=\lceil 5\log(k)/\min(\theta,\gamma)\rceil$                & Constant of warm-up time (Theorem 1)                                                  \\
	& $\Gamma_i,~ i=1,2$                                                     & Stopping time of random walks (proof of Lemma 2)                       \\
	& $I_1,I_2,I_3$                                                          & Three terms of the regret of nonstationary                           \\
	& $K_{alg}$                                                              & Bound of the cumulative regret (Theorem 3)                                          \\
	& $K'$																		& Constant for regret of nonstationary (proof of Theorem 1)\\
	& $K=K'+2M_0/\log(2)$                                                           & Constant in the proof of Theorem 1\\
	& $K_0,K_1$                                                              & Constants for convexity and smoothness (Assumption 1)                           \\
	& $K_2=2K_3/K_0$                                                                  & Constant for $D_k$ (Theorem 1)                                                 \\
	& $K_3$                                                              & Constants for variance in Theorem 3 \eqref{eq: K_3}                             \\
	& $K_4$                                                              & Constants for convergence rate of busy time in Lemma EC.1                                \\
	& $K_5=32e^4/(\min(\theta,\gamma))$                                      & Bound in the proof of Theorem 3 \eqref{eq: K5}                         \\
	& $K_6=K_2\max_p\lambda'(p)$                                             & Bound in the proof of Corollary 2                                           \\
	& $\bar{\lambda},\underline{\lambda}$					& Upper and lower bounds for $\lambda(p)$\\
	& $M$                                                                    & Uniform bound for queueing functions in Lemma 1                                \\
	& $M_0$                                                                 & Upper bound of the regret in the first cycle\\
	& $R_k$                                                                  & Total regret during cycle $k$                                        \\
	& $R_{1,k},R_{2,k}$                                                             & Regret of nonstationary/suboptimality in cycle $k$                             \\
	& $R(L),R_1(L),R_2(L)$                                                   & Cumulative/nonstationary/suboptimal regret by cycle $L$ \\
	& $T_k$                                                                  & Length of cycle $k$                                              \\
	\hline
\end{tabular}
}
\normalfont\caption{Glossary of notation.\label{tab: notation table}}
\end{table}

\end{document}